\documentclass{amsart}
\usepackage{amssymb}
\usepackage{amsfonts}
\usepackage[mag=1000]{geometry}
\usepackage{color}
\usepackage{hyperref}
\usepackage{newpnts}
\usepackage{graphicx}
\newtheorem{theorem}{Theorem}
\theoremstyle{plain}

\newtheorem{conclusion}[theorem]{Conclusion}

\newtheorem{corollary}[theorem]{Corollary}

\newtheorem{definition}[theorem]{Definition}

\newtheorem{lemma}[theorem]{Lemma}

\newtheorem{proposition}[theorem]{Proposition}
\newtheorem{remark}[theorem]{Remark}

\allowdisplaybreaks

\begin{document}
\title[Moser method]{The Moser method and boundedness of solutions to
infinitely degenerate elliptic equations}
\author{Lyudmila Korobenko}
\address{Reed College\\
Portland, Oregon, USA}
\email{korobenko@reed.edu}
\author{Cristian Rios}
\address{University of Calgary\\
Calgary, Alberta, Canada}
\email{crios@ucalgary.ca}
\author{Eric Sawyer}
\address{McMaster University\\
Hamilton, Ontario, Canada}
\email{sawyer@mcmaster.ca}
\author{Ruipeng Shen}
\address{Center for Applied Mathematics\\
Tianjin University \\
Tianjin 300072, P.R.China}
\email{srpgow@163.com}
\date{\today }

\begin{abstract}
We show that if $\mathbb{R}^{n}$ is equipped with certain non-doubling
metric and an Orlicz-Sobolev inequality holds for a special family of Young
functions $\Phi $, then weak solutions to quasilinear infinitely degenerate
elliptic divergence equations of the form%
\begin{equation*}
-\mathrm{div}\mathcal{A}\left( x,u\right) \nabla u=\phi _{0}-\mathrm{div}_{A}%
\vec{\phi}_{1}.
\end{equation*}%
are locally bounded. Furthermore, we establish a maximum principle for
solutions whenever a global Orlicz-Soblev estimate is available. We obtain
these results via the implementation of a Moser iteration method, what
constitutes the first instance of such technique applied to infinite
degenerate equations. These results partially extend previously known
estimates for solutions of these equations but for which the right hand side
did not have a drift term. We also obtain a-priori bounds for small negative
powers of nonnegative solutions; these will be applied to obtain continuity
of solutions in a subsequent paper. \newline
As an application of the abstract theorems, we consider the family $\left\{
f_{k,\sigma }\right\} _{k\in \mathbb{N},\sigma >0}$ of functions on the real
line, 
\begin{equation*}
f_{k,\sigma }\left( x\right) =\left\vert x\right\vert ^{\left( \ln ^{\left(
k\right) }\frac{1}{\left\vert x\right\vert }\right) ^{\sigma }},\ \ \ \ \
-\infty <x<\infty ,
\end{equation*}%
that are infinitely degenerate at the origin, and show that all weak
solutions to the associated infinitely degenerate planar quasilinear
equations of the form 
\begin{equation*}
\mathrm{div}A\left( x,y,u\right) \mathrm{grad}u=\phi \left( x,y\right) ,\ \ \
A\left( x,y,z\right) \sim \left[ 
\begin{array}{cc}
1 & 0 \\ 
0 & f_{k,\sigma }\left( x\right) ^{2}%
\end{array}%
\right] ,
\end{equation*}%
with rough data $A,\phi _{0},\vec{\phi}_{1}$, are locally bounded when $k=1$
and $0<\sigma <1$.
\end{abstract}

\maketitle
\tableofcontents

\section{Introduction and main results}

We consider divergence form quasilinear degenerate elliptic equations of the
form%
\begin{equation}
\mathcal{L}u\equiv -\nabla ^{\mathrm{tr}}\mathcal{A}\left( x,u\left(
x\right) \right) \nabla u=\phi _{0}-\mathrm{div}_{A}\vec{\phi}_{1},\ \ \ \ \
x\in \Omega   \label{eq-quasi}
\end{equation}%
in a domain $\Omega \subset \mathbb{R}^{n}$ is a bounded domain. The matrix $%
\mathcal{A}\left( x,z\right) $ is assumed to be equivalent to a degenerate
elliptic matrix $A\left( x\right) $ in the sense of quadratic forms, i.e. $%
\mathcal{A}\left( \cdot ,z\right) \in \mathfrak{A}\left( A,\Lambda ,\lambda
\right) $ uniformly in $z\in \mathbb{R}$, where $\mathfrak{A}\left(
A,\Lambda ,\lambda \right) $ denotes the class of nonnegative symmetric
matrices $\tilde{A}\left( x\right) $ satisfying 
\begin{equation}
0\leq \lambda ~\xi ^{\mathrm{tr}}A\left( x\right) \xi \leq \xi ^{\mathrm{tr%
}}\tilde{A}\left( x\right) \xi \leq \Lambda \,\xi ^{\mathrm{tr}}A\left(
x\right) \xi \ ,  \label{struc_0}
\end{equation}%
for a.e. $x\in \Omega $ , $\xi \in \mathbb{R}^{n}$, and some fixed $%
0<\lambda \leq \Lambda <\infty $. We further assume that the reference
matrix $A$ satisfies that $\sqrt{A}$ is a bounded Lipschitz continuous $%
n\times n$ real-valued nonnegative definite matrix in $\Omega $, and define
the $A$-gradient and the $A$-divergence operators by%
\begin{equation}
\nabla _{A}=\sqrt{A\left( x\right) }\nabla \ ,\qquad \mathrm{div}_{A}=\mathrm{div%
}\left( \sqrt{A\left( x\right) }\cdot \right) ,  \label{def grad A}
\end{equation}

To obtain local bounds for weak solutions $u$ of the second order
quasilinear equation (\ref{eq-quasi}) it suffices to consider the linear
operator 
\begin{equation}
L_{\tilde{A}}u=-\mathrm{div}\tilde{A}\nabla u=-\mathrm{div}_{\tilde{A}}\nabla _{%
\tilde{A}}u=\phi _{0}-\mathrm{div}_{A}\vec{\phi}_{1},\ \ \ \ \ x\in \Omega 
\label{eq-linear}
\end{equation}%
where the matrix $\tilde{A}\in \mathfrak{A}\left( A,\Lambda ,\lambda \right) 
$, i.e. it satisfies the equivalences (\ref{struc_0}). 

We first work in an abstract setting which requires the existence of an
underlying metric $d$ satisfying some geometric compatibility with the
differential structure induced by $A$, including the validity of a certain
Orlicz-Sobolev inequality (Definition \ref{def OS}) for compactly supported
functions on $d$-metric balls.

The Moser technique implemented here is the first instance of it for
infinite degenerate equations, and as such it has an interest on its own.
The technical obstacles in implementing this technique require the use of a
specially designed family of Young functions $\Phi _{m}\left( t\right) \sim
t\exp \left( \left( \ln t\right) ^{\frac{m-1}{m}}\right) $ as $t\rightarrow
\infty $, $m>1$, because these functions are well-behaved under successive
compositions. These Young functions however are much larger than the ones
considered in \cite{KoRiSaSh19}, namely $\Psi _{N}\sim t\left( \ln t\right)
^{N}$, so our results hold in a more restrictive family of geometries,
although they do extend the boundedness results in \cite{KoRiSaSh19} for
such geometries to operators in which the right hand side also has a drift
term.

Another reason to implement the Moser iteration is that it yields $L^{2}$-$%
L^{\infty }$ estimates for \emph{small negative powers} $u^{\alpha }$ of
nonnegative solutions $u$, which combined with similar estimates for small
positive powers of nonnegative solutions can render a Harnack-type
inequality and provide continuity of solutions. However, our present methods
require convexity of the power functions $t^{\alpha }$, limiting our results
to exponents $\alpha <0$ or $\alpha \geq 1$. In a subsequent work we will
establish estimates for small positive powers of nonnegative solutions via
the De Giorgi method, and obtain continuity of solutions combining these
results.

The abstract results are of interest in themselves because of their greater
generality, but they prove their true relevance in actual geometric settings
where they can yield new boundedness theorems. We provide in this paper an
application of our abstract theory to a \emph{two-dimensional} quasilinear
operator comparable to a diagonal linear operator with degeneracy controlled
by a function $f$ that only depends on one of the variables. The current
implementation of the Moser method requires a rather restrictive assumption
on the type of the degeneracy that is allowed, and does not handle as large
a range of degeneracies as is covered by the De Giorgi iteration in \cite%
{KoRiSaSh19}, or by the trace method in \cite{KoSa21}. However, it does
guarantee boundedness of solutions to degenerate quasilinear equations as in 
\cite{KoRiSaSh19}-\cite{KoSa21} while including the case of non-zero right
hand side. In this application, the structural assumptions on $A$ will
ensure that $A$ is elliptic away from the hyperplane $x_{1}=0$, and that the
Carnot-Carath\'{e}odory metric $d_{A}$ induced by $A$ is topologically
equivalent to the Euclidean metric $d_{\mathbb{E}}$, although these will not
be equivalent metrics since the $d_{A}$-balls are \emph{not doubling} when
centered on that hyperplane. We prove that the assumptions necessary for the
abstract theory, including an Orlicz-Sobolev embedding, all hold, thereby
obtaining boundedness of weak solutions to $-\mathrm{div}\mathcal{A}\left(
x,u\right) \nabla u=\phi _{0}-\mathrm{div}_{A}\vec{\phi}_{1}$ for these
operators in the plane (Theorem \ref{th-geom-bound}). The right hand side
pair $\left( \phi _{0},\vec{\phi}_{1}\right) $ is required to be admissible
as given in Definition \ref{def A admiss} below, which basically requires
the $\left( \phi _{0},\vec{\phi}_{1}\right) $ to belong to the dual of the
homogeneous space $W_{A,0}^{1,1}$ (see Section \ref{section abs-setting} for
the definition of these spaces).

We now present the two-dimensional geometric applications, boundedness
Theorem \ref{th-geom-bound} and the maximum principle, Theorem \ref%
{th-geom-max}. For our geometric results we will specifically consider the
geometry of balls induced by diagonal matrices 
\begin{equation}
A(x)=\left[ 
\begin{array}{cc}
1 & 0 \\ 
0 & f\left( x\right) ^{2}%
\end{array}%
\right]  \label{sing-matrix}
\end{equation}%
where $f=f_{k,\sigma }=e^{-F_{k,\sigma }}$ with 
\begin{equation*}
F_{k,\sigma }\left( r\right) =\left( \ln \frac{1}{r}\right) \left( \ln
^{\left( k\right) }\frac{1}{r}\right) ^{\sigma },\qquad r>0,~k\in \mathbb{N}%
\text{, and }\sigma >0.
\end{equation*}%
That is, $f_{k,\sigma }\left( r\right) =e^{-F_{k,\sigma }\left( r\right)
}=r^{\left( \ln ^{\left( k\right) }\frac{1}{r}\right) ^{\sigma }}$ vanishes
to infinite order at $r=0$, and $f_{k,\sigma }$ vanishes faster than $%
f_{k^{\prime },\sigma ^{\prime }}$ if either $k<k^{\prime }$ or \ if $%
k=k^{\prime }$ and $\sigma >\sigma ^{\prime }$. These geometries are
particular examples of the general geometries $F$ considered in our abstract
theory defined by the structural conditions \ref{structure conditions} in
Section \ref{section geom-setting} below. In \cite{KoRiSaSh19} we consider $%
F_{\sigma }=F_{0,\sigma }=r^{-\sigma }$ ($k=0$) with $0<\sigma <1$, so $%
f_{1,\sigma }\left( r\right) \gg f_{\sigma }=e^{-\frac{1}{r^{\sigma }}}$
near $r=0$. The boundedness results obtained here, albeit having a drift
term on the right hand side and being able to treat small negative powers of
supersolutions, do not include the case $k=0$, $0<\sigma <1$ (as in \cite%
{KoRiSaSh19}), due to the current technical limitations for implementing a
Moser iteration in the infinite degenerate setting.

\begin{theorem}[geometric local boundedness]
\label{th-geom-bound}Let $\left\{ \left( 0,0\right) \right\} \subset \Omega
\subset \mathbb{R}^{2}$ and $\mathcal{A}(x,z)$ be a nonnegative semidefinite
matrix in $\Omega \times \mathbb{R}$ that satisfies the degenerate elliptic
condition (\ref{struc_0}) where $A(x)$ is given by (\ref{sing-matrix}) with $%
f=f_{k,\sigma }$. Then every weak subsolution of (\ref{eq-quasi}): 
\begin{equation*}
\mathcal{L}u\equiv -\nabla ^{\mathrm{tr}}\mathcal{A}\left( x,u\left(
x\right) \right) \nabla u=\phi _{0}-\mathrm{div}_{A}\vec{\phi}_{1}
\end{equation*}%
is \emph{locally bounded above} in $\Omega \subset \mathbb{R}^{2}$ provided
that:

\begin{enumerate}
\item the right hand side pair $\left( \phi _{0},\vec{\phi}_{1}\right) $
satisfies $\phi _{0}\in L_{\mathrm{loc}}^{\Phi ^{\ast }}\left( \Omega
\right) $, where $\Phi ^{\ast }$ is the adjoint Young function to $\Phi _{m}$%
, for some $m>2$, and $\left\vert \vec{\phi}_{1}\right\vert \in L_{\mathrm{%
loc}}^{\infty }\left( \Omega \right) $,

\item at least one of the following two conditions hold:%
\setcounter{enumi}{0}\RESUME%

\begin{enumerate}
\item $k\geq 1$ and $0<\sigma <1$,

\item $k\geq 2$ and $\sigma >0$.%
\setcounter{enumii}{0}\RESUME%
\end{enumerate}
\end{enumerate}
\end{theorem}

The two-dimensional geometric applications, boundedness Theorem \ref%
{th-geom-bound} and the geometric maximum principle Theorem \ref{th-geom-max}%
, are consequences of the corresponding general abstract results, Theorems %
\ref{th-abs-bound} and \ref{th-abs-max}. The following is our geometric
maximum principle in the plane for infinitely degenerate quasilinear
equations. The proofs are presented in Section \ref{section geom-setting}.

\begin{theorem}[geometric maximum principle]
\label{th-geom-max}Let $\Omega \subset \mathbb{R}^{2}$ and $\mathcal{A}(x,z)$
be a nonnegative semidefinite matrix in $\Omega \times \mathbb{R}$ that
satisfies the degenerate elliptic condition (\ref{struc_0}), and assume in
addition that $A(x)=\left[ 
\begin{array}{cc}
1 & 0 \\ 
0 & f\left( x\right) ^{2}%
\end{array}%
\right] $ where $f=f_{k,\sigma }$. Let $u$ be a subsolution of 
\begin{equation*}
\mathcal{L}u\equiv -\nabla ^{\mathrm{tr}}\mathcal{A}\left( x,u\left(
x\right) \right) \nabla u=\phi _{0}-\mathrm{div}_{A}\vec{\phi}_{1},\ \ \ \ \
x\in \Omega .
\end{equation*}%
Then we have the maximum principle, 
\begin{equation*}
\mathrm{esssup}_{x\in \Omega }u\left( x\right) \leq \sup_{x\in \partial
\Omega }u\left( x\right) +C\left\Vert \phi \right\Vert _{L^{\Phi ^{\ast
}}\left( \Omega \right) }\ ,
\end{equation*}%
provided that:

\begin{enumerate}
\item the right hand side pair $\left( \phi _{0},\vec{\phi}_{1}\right) $
satisfies $\phi _{0}\in L^{\Phi ^{\ast }}\left( \Omega \right) $, where $%
\Phi ^{\ast }$ is the adjoint Young function to $\Phi _{m}$, for some $m>2$,
and $\left\vert \vec{\phi}_{1}\right\vert \in L^{\infty }\left( \Omega
\right) $,

\item at least one of the following two conditions hold:%
\setcounter{enumi}{0}\RESUME%

\begin{enumerate}
\item $k\geq 1$ and $0<\sigma <1$,

\item $k\geq 2$ and $\sigma >0$.%
\setcounter{enumii}{0}\RESUME%
\end{enumerate}
\end{enumerate}
\end{theorem}

\vspace*{0.1in}

\subsection{Relation to other results in the literature}

Apart from the two papers by the authors \cite{KoRiSaSh19} and \cite{KoSa21}%
, mentioned earlier, there have been very few related results obtained by
other authors, since this current paper first appeared on the arXiv in 2015.
The two most recent and relevant ones are \cite{CrRo21} and \cite{DiFaMoRo23}%
. In \cite{CrRo21} the authors obtain boundedness of weak solutions to a
certain class of degenerate elliptic Dirichlet problems using an adaptation
of the De Giorgi technique developed in \cite{KoRiSaSh19}. The results there
are of abstract type where one assumes a weighted Sobolev inequality, and
these results are similar, but incomparable, to our abstract results.
However, they obtain a quantitative bound for a much larger class of
inhomogeneous data. On the other hand, there are no geometric theorems
there, which would require verification of complicated hypotheses, such as a
Sobolev inequality. In this paper, as in the original version in the arXiv
in 2015 \cite{KoRiSaSh1}, the use of the Moser iteration is crucial, this
despite the comment made in \cite[page 5]{CrRo21} to the effect that "We
were unable to adapt Moser iteration to work in the context of Orlicz norms,
and it remains an open question whether such an approach is possible in this
setting."\newline
More recently in \cite{DiFaMoRo23} the authors consider quasilinear
degenerate equations of this nature, and they use Moser iteration to obtain
abstract results on Harnack inequalities and H\"{o}lder continuity of
solutions. Similar to \cite{CrRo21}, the authors use an axiomatic approach,
where the relevant (weighted) Sobolev and Poincar\'{e} inequalities, as well
as the doubling property of the weights on the metric balls, are assumed to
hold a-priori. Since there are no geometric theorems established in \cite%
{DiFaMoRo23}, their results are also incomparable to those in our paper.%
\newline
From the point of view of abstract results, the current paper also makes a
new significant contribution. In both \cite{CrRo21} and \cite{DiFaMoRo23}
the authors use $(q,p)$ Sobolev inequalities with $q>p$ and do not perform
Moser or De Giorgi iterations using a weaker Orlicz-Sobolev inequality
employed in this paper. Due to the inhomogeneous nature of the Orlicz norm,
adapting these techniques to this new setting was a highly technical
nontrivial task which required new ideas. This allows to establish
regularity of solutions in the case when the metric balls are non-doubling
with respect to Lebesgue measure, that is, the metric space is not of
homogenous type, see \cite{KoMaRi}.

\subsection{The abstract setting\label{section abs-setting}}

We work in an open, bounded domain $\Omega \subset \mathbb{R}^{n}$ and as
described above we consider nonnegative symmetric real valued matrices $A$
in $\Omega $ such that $\sqrt{A\left( x\right) }$ is uniformly bounded and
uniformly Lipschitz in $\Omega $. The degenerate Sobolev space $%
W_{A}^{1,2}\left( \Omega \right) $ associated to $A$ has norm%
\begin{equation*}
\left\Vert v\right\Vert _{W_{A}^{1,2}}\equiv \sqrt{\int_{\Omega }\left\vert
v\right\vert ^{2}+\left( \left( \nabla v\right) ^{\mathrm{tr}}A\nabla
v\right) }=\sqrt{\int_{\Omega }\left( \left\vert v\right\vert
^{2}+\left\vert \nabla _{A}v\right\vert ^{2}\right) }.
\end{equation*}%
Since $\sqrt{A}$ is Lipschitz then $\mathrm{div}\sqrt{A\left( x\right) }\in
\left( L^{\infty }\left( \Omega \right) \right) ^{n}$, hence the space $%
W_{A}^{1,2}\left( \Omega \right) $ is a Hilbert space (see \cite[Theorem 2]%
{SaWh3}) contained in $L^{2}\left( \Omega \right) $, with inner product
given by the bilinear form 
\begin{equation*}
a_{1}\left( u,v\right) =\int_{\Omega }\nabla _{A}v\cdot \nabla
_{A}w~dx+\int_{\Omega }vw~dx,\qquad v,w\in W_{A}^{1,2}\left( \Omega \right)
\end{equation*}%
where $\nabla _{A}v=\sqrt{A}\nabla v$. The associated homogeneous subspace $%
W_{A,0}^{1,2}\left( \Omega \right) $ is defined as the closure in $%
W_{A}^{1,2}\left( \Omega \right) $ of Lipschitz functions with compact
support, $\mathrm{Lip}_{\mathrm{c}}\left( \Omega \right) $. If a
global (1-1)-Sobolev inequality holds in $\Omega $, i.e. 
\begin{equation}
\int_{\Omega }\left\vert g\right\vert ~dx\leq C_{\Omega }\ \int_{\Omega
}\left\vert \nabla _{A}g\right\vert ~dx\qquad \text{for some }C_{\Omega }>0%
\text{ and all }g\in \mathrm{Lip}_{\mathrm{c}}\left( \Omega
\right) ,  \label{def 11-Sob}
\end{equation}
it follows that the Hilbert space structure in $W_{A,0}^{1,2}\left( \Omega
\right) $ has the equivalent inner product%
\begin{equation*}
a\left( u,v\right) =\int_{\Omega }A\left( x\right) \nabla v\cdot \nabla
w~dx=\int_{\Omega }\nabla _{A}v\cdot \nabla _{A}w~dx,\qquad v,w\in
W_{A,0}^{1,2}\left( \Omega \right) .
\end{equation*}%
In this case we have that $\left\Vert v\right\Vert _{W_{A}^{1,2}\left(
\Omega \right) }\approx \left\Vert \nabla _{A}v\right\Vert _{L^{2}\left(
\Omega \right) }$ for all $v\in W_{A,0}^{1,2}\left( \Omega \right) $. In 
\cite[Section 8.2]{KoRiSaSh19} we show that the above inequality holds for a
wide variety of infinitely degenerate geometries.

Note that $\nabla _{A}:W_{A}^{1,2}\left( \Omega \right) \rightarrow \left(
L^{2}\left( \Omega \right) \right) ^{n}$ and $\mathrm{div}_{A}:\left(
L^{2}\left( \Omega \right) \right) ^{n}\rightarrow \left(
W_{A,0}^{1,2}\left( \Omega \right) \right) ^{\ast }$ are bounded linear
operators, where $\left( W_{A,0}^{1,2}\left( \Omega \right) \right) ^{\ast }$
is the dual space of $W_{A,0}^{1,2}\left( \Omega \right) $. The derivatives
in $W_{A}^{1,2}\left( \Omega \right) $ are understood in the weak sense,
i.e., $\vec{f}=\nabla _{A}u$ in $\Omega $ if and only if $\vec{f}\in \left(
L_{1}\left( \Omega \right) _{\mathrm{loc}}\right) ^{n}$ and for all $\vec{v}%
\in \left( \mathrm{Lip}_{\mathrm{c}}\left( \Omega \right) \right)
^{n}$%
\begin{equation*}
\int_{\Omega }f\cdot \vec{v}~dx=\int_{\Omega }u\mathrm{div}_{A}\vec{v}~dx,
\end{equation*}%
note that the right hand side is integrable since $\mathrm{div}_{A}\vec{v}\in
L^{\infty }\left( \Omega \right) $ and $u\in L^{2}\left( \Omega \right) $.
When $u\in W_{A}^{1,2}\left( \Omega \right) $ and $\tilde{A}\in \mathfrak{A}%
\left( A,\Lambda ,\lambda \right) $ we define the equivalent $\tilde{A}$%
-gradient and $\mathrm{div}_{\tilde{A}}$ operators associated to by setting $%
\nabla _{\tilde{A}}v=\sqrt{\tilde{A}}\nabla v$ and $\left\langle \mathrm{div}_{%
\tilde{A}}\vec{w},v\right\rangle =-\int \vec{w}\cdot \nabla _{\tilde{A}}v$
for all $v\in \mathrm{Lip}_{\mathrm{c}}\left( \Omega \right) $.
From (\ref{struc_0}) it is clear that $\left\vert \nabla _{\tilde{A}}v\left(
x\right) \right\vert \approx \left\vert \nabla _{A}v\left( x\right)
\right\vert $ for a.e. $x\in \Omega $ . Each $u\in W_{A}^{1,2}\left( \Omega
\right) $ then defines the bilinear form 
\begin{equation*}
\tilde{a}\left( v,w\right) =\int_{\Omega }\tilde{A}\left( x\right) \nabla
v\cdot \nabla w=\int_{\Omega }\nabla _{\tilde{A}}v\cdot \nabla _{\tilde{A}%
}w~dx,\qquad v,w\in W_{A}^{1,2}\left( \Omega \right) .
\end{equation*}%
The assumptions (\ref{struc_0}) imply that $\tilde{a}\approx a$ as bilinear
forms, which are bounded on $W_{A}^{1,2}\left( \Omega \right) $, that is $%
\left\vert \tilde{a}\left( v,w\right) \right\vert \lesssim \left\vert
a\left( v,w\right) \right\vert \lesssim \left\Vert v\right\Vert
_{W_{A}^{1,2}}\left\Vert w\right\Vert _{W_{A}^{1,2}}$. In the presence of a
(1-1)-Sobolev inequality (\ref{def 11-Sob}) we moreover have that $a$ and $%
\tilde{a}$ are coercive on $W_{A,0}^{1,2}\left( \Omega \right) $, i.e. $%
\tilde{a}\left( v,v\right) \gtrsim a\left( v,v\right) \gtrsim \left\Vert
v\right\Vert _{W_{A}^{1,2}\left( \Omega \right) }^{2}$.

\begin{definition}[Weak solutions]
\label{def weak solution}Let $\Omega $ be a bounded domain in $\mathbb{R}^{n}
$. Assume that $\phi _{0},\left\vert \vec{\phi}_{1}\right\vert \in L_{%
\mathrm{loc}}^{2}\left( \Omega \right) $. We say that $u\in
W_{A}^{1,2}\left( \Omega \right) $ is a \emph{weak solution} to $L_{\tilde{A}%
}u=-\mathrm{div}\tilde{A}\nabla u=\phi _{0}-\mathrm{div}_{A}\vec{\phi}_{1}$
provided 
\begin{equation}
\int_{\Omega }\tilde{A}\left( x\right) \nabla u\cdot \nabla
w~dx=\int_{\Omega }\phi _{0}w+\vec{\phi}_{1}\cdot \nabla _{A}w~dx
\label{equation}
\end{equation}%
for all $w\in \mathrm{Lip}_{\mathrm{c}}\left( \Omega \right) $.
Equation (\ref{equation}) may be written as $\tilde{a}\left( u,w\right)
=F\left( w\right) $ where $F$ is the operator defined by the right hand side
of (\ref{equation}), which is a bounded linear operator on $%
W_{A,0}^{1,2}\left( \Omega \right) $. With this notation we similarly define
the notion of subsolution (supersolution) by saying that $u\in
W_{A}^{1,2}\left( \Omega \right) $ is a (\emph{weak})\emph{\ subsolution} (%
\emph{supersolution}) to $L_{\tilde{A}}u=\phi _{0}-\mathrm{div}_{A}\vec{\phi}%
_{1}$, and write $L_{\tilde{A}}u\leq \phi _{0}-\mathrm{div}_{A}\vec{\phi}_{1}$
($L_{\tilde{A}}u\geq \phi _{0}-\mathrm{div}_{A}\vec{\phi}_{1}$), if and only
if 
\begin{equation*}
\tilde{a}\left( u,w\right) \leq F\left( w\right) \qquad (\tilde{a}\left(
u,w\right) \geq F\left( w\right) )\qquad \text{for all \emph{nonnegative} }%
w\in \mathrm{Lip}_{\mathrm{c}}\left( \Omega \right) .
\end{equation*}%
Finally, we say that $u\in W_{A}^{1,2}\left( \Omega \right) $ is a weak
solution (subsolution, supersolution) to $\mathcal{L}u=-\mathrm{div}\mathcal{A}%
\left( x,u\right) \nabla u=\phi _{0}-\mathrm{div}_{A}\vec{\phi}_{1}$ provided $%
u$ is a weak solution (subsolution, supersolution) to $L_{\tilde{A}}u=\phi
_{0}-\mathrm{div}_{A}\vec{\phi}_{1}$ for $\tilde{A}\left( x\right) =\mathcal{A}%
\left( x,u\left( x\right) \right) $.
\end{definition}

Note that our structural condition (\ref{struc_0}) implies that the integral
on the left above is absolutely convergent, and our assumption that $\phi
_{0},\left\vert \vec{\phi}_{1}\right\vert \in L_{\mathrm{loc}}^{2}\left(
\Omega \right) $ implies that the integrals on the right above are
absolutely convergent. In Definition \ref{def A admiss} below we weaken the
assumptions on the right hand side pair $\left( \phi _{0},\vec{\phi}%
_{1}\right) $.

In this abstract setting we work with the differential structure defined
through the matrix $A$, inducing the Sobolev spaces $W_{A}^{1,2}\left(
\Omega \right) $. We further assume the existence of a metric $d:\mathbb{R}%
^{n}\times \mathbb{R}^{n}\rightarrow \left[ 0,\infty \right) $ satisfying
certain geometric compatibility with this differential structure, namely
conditions (\ref{cond-1}), (\ref{cond-2}), and (\ref{cond-3}) in Theorem \ref%
{th-abs-bound}. We now describe each assumption in more detail.

\begin{definition}[Standard sequence of accumulating Lipschitz functions]
\label{def Lip cutoff}Let $\Omega $ be a bounded domain in $\mathbb{R}^{n}$
and let $d:\mathbb{R}^{n}\times \mathbb{R}^{n}\rightarrow \left[ 0,\infty
\right) $ be a metric. Fix , $r>0$, $\nu \in \left( 0,1\right) $, and $x\in
\Omega $. We define an $\left( A,d\right) $-\emph{\ standard} sequence of
Lipschitz cutoff functions $\left\{ \psi _{j}\right\} _{j=1}^{\infty }$ at $%
\left( x,r\right) $, along with sets $B(x,r_{j})\supset \mathrm{supp}\psi
_{j}$, to be a sequence satisfying $\psi _{j}=1$ on $B(x,r_{j+1})$, $r_{1}=r$%
, $r_{\infty }\equiv \lim_{j\rightarrow \infty }r_{j}=\nu r$, $r_{j}-r_{j+1}=%
\frac{c}{j^{2}}\left( 1-\nu \right) r$ for a uniquely determined constant $c$%
, and $\left\Vert \nabla _{A}\psi _{j}\right\Vert _{\infty }\lesssim \frac{%
j^{2}}{\left( 1-\nu \right) r}$ with $\nabla _{A}$ as in (\ref{def grad A})
(see e.g. \cite{SaWh4}).
\end{definition}

A sufficient condition for the existence of these cutoffs would be the
existence of a constant $C_{d}>0$ such that whenever $0<r<R<\infty $ and $%
B\left( x,R\right) \subset \Omega $, then there exists a Lipschitz function $%
\psi =\psi _{x,r,R}\in \mathrm{Lip}_{\mathrm{c}}\left( B_{R}\right) $ such
that $0\leq \psi \leq 1$, $\psi \equiv 1$ in $B_{r}$ and $\left\Vert \nabla
_{A}\psi \right\Vert _{\infty }\leq \frac{C_{d}}{R-r}$. This is indeed the
case $d=d_{A}$ is the Carnot-Carath\'{e}odory metric induced by a continuous
matrix $A$, and this metric is topologically equivalent to the Euclidean
metric (see Lemma \ref{spec_cutoff_lemma}).

We will need to assume the following single scale $\left( \Phi ,A,\varphi
\right) $-Orlicz-Sobolev bump inequality:

\begin{definition}[Orlicz-Sobolev inequality]
\label{def OS}Let $\Omega $ be a bounded domain in $\mathbb{R}^{n}$, the $%
\left( \Phi ,A\right) $-Orlicz-Sobolev bump inequality for $\Omega $ is 
\begin{equation}
\Phi ^{\left( -1\right) }\left( \int_{\Omega }\Phi \left( w\right)
~dx\right) \leq C\ \left\Vert \nabla _{A}w\right\Vert _{L^{1}\left( \Omega
\right) },\ \ \ \ \ w\in \mathrm{Lip}_{\mathrm{c}}\left( \Omega
\right) ,  \label{OS global}
\end{equation}%
where $dx$ is Lebesgue measure in $\mathbb{R}^{n}$ and $C$ depends on $n$, $%
A $, $\Phi $, and $\Omega $ but not on $w$.\newline
Fix $x\in \Omega $ and $r>0$ such that $B\left( x,r\right) \subset \Omega $,
the $\left( \Phi ,A,\varphi \right) $-Orlicz-Sobolev bump inequality at $%
\left( x,r\right) $ is: 
\begin{equation}
\Phi ^{\left( -1\right) }\left( \int_{B\left( x,\rho \right) }\Phi \left(
w\right) d\mu _{x,\rho }\right) \leq \varphi \left( \rho \right) \
\left\Vert \nabla _{A}w\right\Vert _{L^{1}\left( \mu _{x,\rho }\right) },\ \
\ \ \ 0<\rho \leq r\text{, }w\in \mathrm{Lip}_{\mathrm{c}}\left(
B\left( x,\rho \right) \right) ,  \label{OS ineq}
\end{equation}%
where $d\mu _{x,\rho }\left( y\right) =\frac{1}{\left\vert B\left( x,\rho
\right) \right\vert }\mathbf{1}_{B\left( x,\rho \right) }\left( y\right) dy$%
, and the function ${\varphi }\left( r\right) $, dubbed the superradius, is
continuous, nondecreasing, and it satisfies ${\varphi }\left( 0\right) =0$, $%
{\varphi }\left( \rho \right) \geq \rho $ for all $0\leq \rho \leq r$.%
\newline
Finally, we say that the \emph{single scale}\footnote{%
as opposed to the multi-scale Sobolev bump inequalities assumed for
continuity, that require $0<\rho <r_{0}$.} $\left( \Phi ,A,\varphi \right) $%
-Orlicz-Sobolev bump inequality holds at $\left( x,r\right) $ if (\ref{OS
ineq}) holds for $\rho =r$ (and not necessarily for $0<\rho <r$).
\end{definition}

The particular family of Orlicz bump functions $\Phi _{m}$ required above
that is crucial for our theorem is the family%
\begin{equation}
\Phi _{m}\left( t\right) =e^{\left( \left( \ln t\right) ^{\frac{1}{m}%
}+1\right) ^{m}},\ \ \ \ \ t>E_{m}=e^{2^{m}},\text{ }m>1,
\label{def Orlicz bumps}
\end{equation}%
which is then extended in (\ref{def Phi m ext}) below to be \emph{linear} on
the interval $\left[ 0,E_{m}\right] $, continuous and submultiplicative on $%
\left[ 0,\infty \right) $; we discuss this in more detail in Section \ref%
{Sec Orlicz}.

Finally, we describe the notion of admissible right hand side pair.

\begin{definition}[Admissible right hand sides]
\label{def A admiss}Let $\Omega $ be a bounded domain in $\mathbb{R}^{n}$
and let $\phi _{0}:\Omega \rightarrow \mathbb{R}$, $\vec{\phi}_{1}:\Omega
\rightarrow \mathbb{R}^{n}$ be locally integrable. We call$\left( \phi _{0},%
\vec{\phi}_{1}\right) $ a right hand side pair (although we may just refer
them as just a "pair"). Fix $x\in \Omega $ and $\rho >0$, we say that the
right hand side pair $\left( \phi _{0},\vec{\phi}_{1}\right) $ is $A$\emph{%
-admissible} at $\left( x,\rho \right) $ if 
\begin{equation}
\left\Vert \left( \phi ,\vec{\phi}_{1}\right) \right\Vert _{\mathcal{X}%
\left( B\left( x,\rho \right) \right) }\equiv \sup_{v\in \mathcal{W}%
_{1}}\left\vert \int_{B\left( x,\rho \right) }v\phi _{0}\,dy\right\vert
+\sup_{v\in \mathcal{W}_{1}}\left\vert \int_{B\left( x,\rho \right) }\nabla
_{A}v\cdot \vec{\phi}_{1}\,dy\right\vert <\infty .  \label{admiss}
\end{equation}%
where $\mathcal{W}_{1}=\left\{ v\in \left( W_{A,0}^{1,1}\right) \left(
B\left( x,\rho \right) \right) :\int_{B\left( x,\rho \right) }\left\vert
\nabla _{A}v\right\vert \,dy=1\right\} $. We also write $\left\Vert \phi
_{0}\right\Vert _{\mathcal{X}\left( B\left( x,\rho \right) \right)
}=\left\Vert \left( \phi ,\mathbf{0}\right) \right\Vert _{\mathcal{X}\left(
B\left( x,\rho \right) \right) }$ and $\left\Vert \vec{\phi}_{1}\right\Vert
_{\mathcal{X}\left( B\left( x,\rho \right) \right) }=\left\Vert \left( 0,%
\vec{\phi}_{1}\right) \right\Vert _{\mathcal{X}\left( B\left( x,\rho \right)
\right) }$. Similarly, we say the pair $\left( \phi _{0},\vec{\phi}%
_{1}\right) $ is $A$\emph{-admissible} for $\Omega $ if (\ref{admiss}) holds
with $\Omega $ replacing $B\left( x,\rho \right) $. \newline
For convenience we also introduce the concept of \emph{strongly} $A$\emph{%
-admissible} pair. We say that $\left( \phi _{0},\vec{\phi}_{1}\right) $ is
strongly $A$\emph{-}admissible at $\left( x,\rho \right) $ if%
\begin{equation*}
\left\Vert \left( \phi ,\vec{\phi}_{1}\right) \right\Vert _{\mathcal{X}%
^{\ast }\left( B\left( x,\rho \right) \right) }\equiv \sup_{v\in \mathcal{W}%
_{1}}\int_{B\left( x,\rho \right) }\left\vert v\phi _{0}\right\vert
\,dy+\sup_{v\in \mathcal{W}_{1}}\int_{B\left( x,\rho \right) }\left\vert
\nabla _{A}v\cdot \vec{\phi}_{1}\right\vert \,dy<\infty .
\end{equation*}%
It is clear that if $\left( \phi _{0},\vec{\phi}_{1}\right) $ is strongly $A$%
\emph{-}admissible at $\left( x,\rho \right) $ then it is $A$\emph{-}%
admissible at $\left( x,\rho \right) $.
\end{definition}

In the above definition an $A$-admissible right hand side pair at $\left(
x,r\right) $ defines a bounded linear operator $T_{\left( \phi ,\vec{\phi}%
_{1}\right) }$ on the space $W_{A,0}^{1,1}\left( B\left( x,r\right) \right) $
by setting%
\begin{equation*}
T_{\left( \phi ,\vec{\phi}_{1}\right) }\left( v\right) =\int_{B\left( x,\rho
\right) }v\phi _{0}\,dy+\int_{B\left( x,\rho \right) }\nabla _{A}v\cdot \vec{%
\phi}_{1}\,dy.
\end{equation*}

Recall that a measurable function $u$ in $\Omega $ is \emph{\ locally
bounded above} at $x$ if $u$ can be modified on a set of measure zero so
that the modified function $\widetilde{u}$ is bounded above in some
neighbourhood of $x$.

\begin{theorem}[abstract local boundedness]
\label{th-abs-bound}Let $\Omega $ be a bounded domain in $\mathbb{R}^{n}$.
Suppose that $\mathcal{A}(x,z)$ is a nonnegative semidefinite matrix in $%
\Omega \times \mathbb{R}$ that satisfies the degenerate elliptic condition (%
\ref{struc_0}). Let $d(x,y)$ be a symmetric metric in $\Omega $, and suppose
that $B\left( x,r\right) =\{y\in \Omega :d(x,y)<r\}$ with $x\in \Omega $ are
the corresponding metric balls. Fix $x\in \Omega $. Then every weak \emph{%
subsolution} (\emph{supersolution}) of (\ref{eq-quasi}) is \emph{locally
bounded above} (\emph{locally bounded below}) at $x$ provided there is $%
r_{0}>0$ such that:%
\renewcommand{\theenumi}{\roman{enumi}}%

\begin{enumerate}
\item \label{cond-1}the right hand side pair $\left( \phi _{0},\vec{\phi}%
_{1}\right) $ is $A$-admissible at $\left( x,r_{0}\right) $,

\item \label{cond-2}the single scale $\left( \Phi ,A,\varphi \right) $%
-Orlicz-Sobolev bump inequality (\ref{OS ineq}) holds at $\left(
x,r_{0}\right) $ with $\Phi =\Phi _{m}$ for some $m>2$,

\item \label{cond-3}there exists an $\left( A,d\right) $-\emph{standard}
accumulating sequence of Lipschitz cutoff functions at $\left(
x,r_{0}\right) $.\newline
Similarly, under the above three conditions every weak \emph{supersolution}
of (\ref{eq-quasi}) is \emph{locally bounded below} at $x$, and every weak 
\emph{solution} of (\ref{eq-quasi}) is \emph{locally bounded} at $x$.\newline
In particular, every weak \emph{solution} (\emph{supersolution}) of (\ref%
{eq-quasi}) is \emph{locally bounded} at $x$.%
\renewcommand{\theenumi}{\arabic{enumi}}%
\setcounter{enumi}{0}\RESUME%
\end{enumerate}
\end{theorem}

\begin{proof}
This local boundedness result is an immediate consequence of Theorem \ref%
{L_infinity} for $\beta =1$, proven in Section \ref{section local bound}.
Indeed, setting $\tilde{A}\left( x\right) =\mathcal{A}\left( x,u\left(
x\right) \right) $ because of the equivalences (\ref{struc_0}) we have that $%
\tilde{A}$ satisfies (\ref{struc_0}). By hypothesis, the $\left( \Phi
,A,\varphi \right) $-Orlicz-Sobolev bump inequality (\ref{OS ineq}) holds at 
$\left( x,r_{0}\right) $ with $\Phi =\Phi _{m}$ for some $m>2$ and an $%
\left( A,d\right) $-\emph{standard} accumulating sequence of Lipschitz
cutoff functions at $\left( x,r_{0}\right) $.

Thus, if $u$ is a weak subsolution of (\ref{eq-quasi}), then it is a weak
subsolution of $L_{\tilde{A}}u=-\mathrm{div}\tilde{A}\nabla u=\phi _{0}-\mathrm{%
div}_{A}\vec{\phi}_{1}$, and all the hypotheses of Theorem \ref{L_infinity}
are satisfied, therefore $u$ is locally bounded above ($u^{+}\in L_{\mathrm{%
loc}}^{\infty }\left( \Omega \right) $). In fact, Theorem \ref{L_infinity}
provides precise estimates: for $\nu _{0}\leq \nu <1$, with $\nu _{0}=1-%
\frac{\delta _{0}\left( r\right) }{r}$, where $\delta _{x}\left( r\right) $
is the doubling increment of $B\left( x,r\right) $, defined by (\ref%
{nondoub_order}), we have that there exists a constant $C=C\left( \varphi
,m,\lambda ,\Lambda ,r,\nu \right) $ such that 
\begin{equation*}
\left\Vert u^{+}+\phi ^{\ast }\right\Vert _{L^{\infty }\left( B\left( x,\nu
r\right) \right) }\leq C~\left\Vert u^{+}+\phi ^{\ast }\right\Vert
_{L^{2}(B\left( x,r\right) ,d\mu _{r})}<\infty
\end{equation*}%
where . The last inequality follows form the fact that since $u\in
W_{A}^{1,2}\left( B\left( x,r\right) \right) $, then $\left\Vert u^{+}+\phi
^{\ast }\right\Vert _{L^{2}(B\left( x,r\right) ,,d\mu _{r})}=\left( \frac{1}{%
\left\vert B\left( x,r\right) \right\vert }\int_{B\left( x,r\right) }\left(
u^{+}+\phi ^{\ast }\right) ^{2}dx\right) ^{\frac{1}{2}}<\infty $. Similarly,
if $u$ is a weak supersolution of (\ref{eq-quasi}) we conclude that%
\begin{equation*}
\left\Vert u^{-}+\phi ^{\ast }\right\Vert _{L^{\infty }\left( B\left( x,\nu
r\right) \right) }\leq C~\left\Vert u^{-}+\phi ^{\ast }\right\Vert
_{L^{2}(B\left( x,r\right) ,d\mu _{r})}<\infty .
\end{equation*}
\end{proof}

\begin{remark}
The hypotheses required for local boundedness of weak solutions to $L_{%
\tilde{A}}u=\phi _{0}-\mathrm{div}_{A}\vec{\phi}_{1}$ at a single \emph{fixed}
point $x$ in $\Omega $ are quite weak; namely we only need the existence of
cutoff functions for $B\left( x,r_{0}\right) $ for some $r_{0}>0$, that the
inhomogeneous couple $\left( \phi _{0},\vec{\phi}_{1}\right) $ is $A$\emph{%
-admissible} at \textbf{just one} point $\left( x,r_{0}\right) $ , and the 
\emph{single} \emph{scale} condition relating the geometry to the equation
at \textbf{the one} point $\left( x,r_{0}\right) $.
\end{remark}

\begin{remark}
We could of course take the metric $d$ to be the Carnot-Carath\'{e}odory
metric associated with $A$, but the present formulation allows for
additional flexibility in the choice of balls used for Moser iteration.
\end{remark}

In the special case that a weak subsolution $u$ to (\ref{eq-quasi}) is \emph{%
\ nonpositive} on the boundary of $\Omega $, we can obtain a global
boundedness inequality $\left\Vert u\right\Vert _{L^{\infty }\left( B\left(
x,r_{0}\right) \right) }\lesssim \Vert \left( \phi _{0},\vec{\phi}%
_{1}\right) \Vert _{X\left( B\left( x,r_{0}\right) \right) }$ from the
arguments used for Theorem \ref{th-abs-bound}, simply by noting that
integration by parts no longer requires premultiplication by a Lipschitz
cutoff function. Moreover, the ensuing arguments work just as well for an
arbitrary bounded open set $\Omega $ in place of the ball $B\left(
x,r_{0}\right) $, provided only that we assume our Sobolev inequality for $%
\Omega $ instead of for the ball $B\left( x,r_{0}\right) $. Of course there
is no role played here by a superradius $\varphi $. This type of result is
usually referred to as a \emph{maximum principle}, and we now formulate our
theorem precisely, the proof is presented in Section \ref{section maxp}.

We say a function $u\in W_{A}^{1,2}\left( \Omega \right) $ is \emph{bounded
by a constant }$\ell \in \mathbb{R}$ on the boundary $\partial \Omega $ if $%
\left( u-\ell \right) ^{+}=\max \left\{ u-\ell ,0\right\} \in \left(
W_{A}^{1,2}\right) _{0}\left( \Omega \right) $. We define $\sup_{x\in
\partial \Omega }u\left( x\right) $ to be $\inf \left\{ \ell \in \mathbb{R}%
:\left( u-\ell \right) ^{+}\in \left( W_{A}^{1,2}\right) _{0}\left( \Omega
\right) \right\} $.

\begin{theorem}[abstract maximum principle]
\label{th-abs-max}Let $\Omega $ be a bounded domain in $\mathbb{R}^{n}$.
Suppose that $\mathcal{A}(x,z)$ is a nonnegative semidefinite matrix in $%
\Omega \times \mathbb{R}$ that satisfies the degenerate elliptic condition (%
\ref{struc_0}). Let $u$ be a subsolution of (\ref{eq-quasi}). Then the
following maximum principle holds, 
\begin{equation*}
\mathrm{esssup}_{x\in \Omega }u\left( x\right) \leq \sup_{x\in \partial
\Omega }u\left( x\right) +C\left\Vert \left( \phi _{0},\vec{\phi}_{1}\right)
\right\Vert _{\mathcal{X}\left( \Omega \right) }\ ,
\end{equation*}%
where the constant $C$ depends only on $n,~m,~\lambda ,~\Lambda ,~A,~\Phi $
and $\Omega $, provided that:

\begin{enumerate}
\item the pair $\left( \phi _{0},\vec{\phi}_{1}\right) $ is $A$-admissible
for $\Omega $,

\item the global $\left( \Phi ,A\right) $-Orlicz-Sobolev bump inequality (%
\ref{OS global}) in $\Omega $ holds with $\Phi =\Phi _{m}$ for some $m>2$.%
\setcounter{enumi}{0}\RESUME%
\end{enumerate}
\end{theorem}

The proof of the abstract maximum principle is given in Section \ref{section
maxp}.

The specific relation between the metric and the Orlicz-Sobolev embedding
will be given in terms of the concept of \emph{doubling increment} of a ball
and its connection with the superradius ${\varphi }$. The bounds in Theorems %
\ref{L_infinity} and \ref{L_infinity-beta} are the embedding norms of $%
L^{\infty }\left( B_{r-\delta _{x}\left( r\right) }\right) $ into $%
L^{2}\left( B_{r}\right) $.

\begin{definition}
\label{definition doubling}Let $\Omega $ be a bounded domain in $\mathbb{R}%
^{n}$. Let $\delta _{x}\left( r\right) $ be defined implicitly by 
\begin{equation}
\left\vert B\left( x,r-\delta _{x}\left( r\right) \right) \right\vert =\frac{%
1}{2}\left\vert B\left( x,r\right) \right\vert ,  \label{nondoub_order}
\end{equation}%
We refer to $\delta _{x}\left( r\right) $ as the \emph{doubling increment}
of the ball $B\left( x,r\right) $.
\end{definition}

\section{Caccioppoli inequalities for weak subsolutions and supersolutions 
\label{Sec Caccioppoli}}

In this section we establish various Caccioppoli inequalities for
subsolutions and supersolutions of (\ref{eq-linear}) (see Definition \ref%
{def weak solution}). In order to prove a Caccioppoli inequality, we assume
that the inhomogeneous pair $\left( \phi _{0},\vec{\phi}_{1}\right) $ in (%
\ref{equation}) is admissible for $A$ in the whole domain $\Omega $ in sense
of Definition \ref{def A admiss}.

What is usually called a Caccioppoli inequality is a reverse Sobolev
inequality which is valid only for functions satisfying an equation of the
form $L_{\tilde{A}}u\geq \phi _{0}-\mathrm{div}_{A}\vec{\phi}_{1}$ or $L_{%
\tilde{A}}u\leq \phi _{0}-\mathrm{div}_{A}\vec{\phi}_{1}$. The Moser iteration
is based on these type of inequalities obtained from the equation when the
test function is an appropriate function of the solution. If $u\in
W_{A}^{1,2}\left( \Omega \right) $, and $h$ is a $C^{0,1}$ or $C^{1,1}$
function on $\left[ 0,\infty \right) $, then $h\left( u\right) $ formally
satisfies the equation 
\begin{equation*}
L_{\tilde{A}}\left( h\left( u\right) \right) =-\mathrm{div}\tilde{A}\nabla
\left( h\left( u\right) \right) =-\mathrm{div}\tilde{A}h^{\prime }\left(
u\right) \nabla u=h^{\prime }\left( u\right) Lu-h^{\prime \prime }\left(
u\right) \left\vert \nabla _{\tilde{A}}u\right\vert ^{2}.
\end{equation*}%
Indeed, if $w\in W_{A,0}^{1,2}\left( \Omega \right) $ and $u$ is a positive
subsolution or supersolution of (\ref{eq-linear}) in $\Omega $, we have%
\begin{eqnarray*}
\int \nabla _{\tilde{A}}w\cdot \nabla _{\tilde{A}}h\left( u\right)  &=&\int
h^{\prime }\left( u\right) \nabla _{\tilde{A}}w\cdot \nabla _{\tilde{A}%
}u=\int \nabla _{\tilde{A}}\left( h^{\prime }\left( u\right) w\right) \cdot
\nabla _{\tilde{A}}u-\int wh^{\prime \prime }\left( u\right) \nabla _{\tilde{%
A}}u\cdot \nabla _{\tilde{A}}u \\
&\leq &\int wh^{\prime }\left( u\right) \phi _{0}+\int \nabla _{A}\left(
wh^{\prime }\left( u\right) \right) \cdot \vec{\phi}_{1}-\int wh^{\prime
\prime }\left( u\right) \left\vert \nabla _{\tilde{A}}u\right\vert ^{2}
\end{eqnarray*}%
provided that $wh^{\prime }\left( u\right) \in W_{A,0}^{1,2}\left( \Omega
\right) $ and that it is nonnegative if $u$ is a subsolution, and
nonpositive if $u$ is a supersolution. Note that $wh^{\prime }\left(
u\right) \in W_{A,0}^{1,2}\left( \Omega \right) $ if in addition we have
that $h^{\prime }$ is bounded.

We will establish two Caccioppoli inequalities. Lemma \ref{reverse Sobolev}
holds for \emph{convex increasing} functions $h$ applied to $u^{\pm }$; this
estimate is utilized to implement a Moser iteration scheme to obtain
boundedness of solutions without restrictions on their sign. The other
result, Lemma \ref{Cacc-positive}, applies to convex functions of
nonnegative subsolutions or supersolutions, and the function $h$ will
satisfy suitable structural properties which will allow us to obtain
(through a Moser iteration) inner ball inequalities for negative powers $%
u^{\beta }$ of the solution,.

\begin{lemma}
\label{reverse Sobolev}Assume that $u\in W_{A}^{1,2}\left( B\right) $ is a
weak \emph{subsolution} to $L_{\tilde{A}}u=\phi _{0}-\mathrm{div}_{A}\vec{\phi}%
_{1}$ in $B=B\left( x,r\right) $, where $\left( \phi _{0},\vec{\phi}%
_{1}\right) $ is an admissible pair and $\tilde{A}\in \mathfrak{A}\left(
A,\Lambda ,\lambda \right) $ (i.e. it satisfies the equivalences (\ref%
{struc_0}) for some $0<\lambda \leq \Lambda <\infty $). Let $h(t)\geq 0$ be
a Lipschitz convex function which satisfies $0<h_{-}^{\prime }\left(
t\right) \leq C_{h}\frac{h\left( t\right) }{t}$, for $t>0$ and it is
piecewise twice continuously differentiable except possibly at finitely many
points, where $C_{h}\geq 1$ is a constant. Then the following reverse
Sobolev inequality holds for any $\psi \in \mathrm{Lip}_{\mathrm{c}%
}\left( B\right) $: 
\begin{equation}
\int_{B}\psi ^{2}\left\vert \nabla _{A}\left[ h\left( u^{+}+\phi ^{\ast
}\right) \right] \right\vert ^{2}dx\leq C_{\lambda ,\Lambda
}C_{h}^{2}\int_{B}h\left( u^{+}+\phi ^{\ast }\right) ^{2}\left( \left\vert
\nabla _{A}\psi \right\vert ^{2}+\psi ^{2}\right) ,  \label{reverse Sobs}
\end{equation}%
where $\phi ^{\ast }=\phi ^{\ast }\left( x,r\right) =\left\Vert \left( \phi ,%
\vec{\phi}_{1}\right) \right\Vert _{X\left( B\left( x,r\right) \right) }$ as
given in Definition \ref{def A admiss}. Moreover, if $u\in W_{A}^{1,2}\left(
B\right) $ is a weak \emph{supersolution} to $L_{\tilde{A}}u=\phi _{0}-\mathrm{%
div}_{A}\vec{\phi}_{1}$ in $B$, then (\ref{reverse Sobs}) holds with $u^{+}$
replaced by $u^{-}$.
\end{lemma}

\begin{proof}
From the hypothesis, if $t>0$ is a discontinuity point of $h^{\prime }$,
then $h^{\prime }$ has simple jump discontinuity there, and both the left
and right derivatives are defined with $h_{+}^{\prime }\left( t\right)
-h_{-}^{\prime }\left( t\right) >0$. Following the proof of Theorem 8.15 in 
\cite{GiTr}, for $N\gg \phi ^{\ast }\geq 0$ larger than the last point of
discontinuity of $h^{\prime }$, we define $H\in C^{0,1}\left( \left[ \phi
^{\ast },\infty \right) \right) $ by%
\begin{equation*}
H\left( t\right) =\left\{ 
\begin{array}{ll}
h\left( t\right) -h\left( \phi ^{\ast }\right)  & t\in \left[ \phi ^{\ast },N%
\right]  \\ 
h\left( N\right) -h\left( \phi ^{\ast }\right) +h_{-}^{\prime }\left(
N\right) \left( t-N\right)  & t>N%
\end{array}%
\right. ,
\end{equation*}%
and let $\omega \left( t\right) =\int_{\phi ^{\ast }}^{t}\left( H^{\prime
}\left( s\right) \right) ^{2}ds$ for $t\geq \phi ^{\ast }$, i.e. 
\begin{equation*}
\omega \left( t\right) =\left\{ 
\begin{array}{ll}
\int_{\phi ^{\ast }}^{t}\left( h^{\prime }\left( s\right) \right) ^{2}ds & 
t\in \left[ \phi ^{\ast },N\right]  \\ 
\int_{\phi ^{\ast }}^{N}\left( h^{\prime }\left( s\right) \right)
^{2}ds+\left( h_{-}^{\prime }\left( N\right) \right) ^{2}\left( t-N\right)
\quad  & t>N%
\end{array}%
\right. .
\end{equation*}%
Then $\omega $ is continuous and piecewise differentiable for all $t\geq 0$,
with $\omega ^{\prime }\left( t\right) $ having at most finitely many simple
jump discontinuities. Since $h$ is convex we have that $H^{\prime }\left(
t\right) $ is increasing, and therefore%
\begin{equation}
\omega \left( t\right) =\int_{\phi ^{\ast }}^{t}\left( H^{\prime }\left(
s\right) \right) ^{2}ds\leq H^{\prime }\left( t\right) \int_{\phi ^{\ast
}}^{t}H^{\prime }\left( s\right) ds=H^{\prime }\left( t\right) H\left(
t\right) .  \label{we-H}
\end{equation}%
Note also that, since $h$ is convex, $H\left( t\right) \leq h\left( t\right) 
$ for all $t\geq 0$. Now, since both $h$ and $h^{\prime }$ are locally
bounded on $\left[ 0,\infty \right) $, it follows the function $w\left(
x\right) =\omega \left( u^{+}\left( x\right) +\phi ^{\ast }\right) \in
W_{A}^{1,2}\left( \Omega \right) $ whenever $u\in W_{A}^{1,2}\left( \Omega
\right) $, moreover, $\mathrm{supp}w=\mathrm{supp}u^{+}$ and $\nabla
_{A}w=\left( H^{\prime }\left( u^{+}\left( x\right) +\phi ^{\ast }\right)
\right) ^{2}\nabla _{A}u^{+}$.

If $u$ is a subsolution to $L_{\tilde{A}}u=\phi _{0}-\mathrm{div}_{A}\vec{\phi}%
_{1}$ in $B(0,r)$ and $\psi \in \mathrm{Lip}_{\mathrm{c}}\left(
B\left( 0,r\right) \right) $, then we have that $\psi ^{2}w\in
W_{A,0}^{1,2}\left( \Omega \right) $ and we have%
\begin{equation*}
\int \nabla _{\tilde{A}}u\cdot \nabla _{\tilde{A}}\left( \psi ^{2}w\right)
\leq \int \psi ^{2}w\phi _{0}+\int \nabla _{A}\left( \psi ^{2}w\right) \cdot 
\vec{\phi}_{1}.
\end{equation*}%
Write $v\left( x\right) =H\left( u^{+}\left( x\right) +\phi ^{\ast }\right) $%
, and $v^{\prime }\left( x\right) =H^{\prime }\left( u^{+}\left( x\right)
+\phi ^{\ast }\right) $ then the left hand side equals 
\begin{eqnarray*}
\int \nabla _{\tilde{A}}u\cdot \nabla _{\tilde{A}}\left( \psi ^{2}w\right) 
&=&\int \psi ^{2}\nabla _{\tilde{A}}u\cdot \nabla _{\tilde{A}}w+2\int \psi
w\nabla _{\tilde{A}}u\cdot \nabla _{\tilde{A}}\psi  \\
&=&\int \psi ^{2}\left( v^{\prime }\right) ^{2}\nabla _{\tilde{A}}u^{+}\cdot
\nabla _{\tilde{A}}u^{+}+2\int \psi w\nabla _{\tilde{A}}u^{+}\cdot \nabla
_{A}\psi  \\
&=&\int \psi ^{2}\left\vert \nabla _{\tilde{A}}v\right\vert ^{2}+2\int \psi
w\nabla _{\tilde{A}}u^{+}\cdot \nabla _{A}\psi ,
\end{eqnarray*}%
where we used that $\mathrm{supp}w=\mathrm{supp}u^{+}$; we obtain%
\begin{equation*}
\int \psi ^{2}\left\vert \nabla _{\tilde{A}}v\right\vert ^{2}\leq -2\int
\psi w\nabla _{\tilde{A}}u^{+}\cdot \nabla _{\tilde{A}}\psi +\int \psi
^{2}w\phi _{0}+\int \nabla _{A}\left( \psi ^{2}w\right) \cdot \vec{\phi}_{1}.
\end{equation*}%
From (\ref{we-H}) we have $w\left( x\right) =\omega \left( u^{+}\left(
x\right) +\phi ^{\ast }\right) \leq H^{\prime }\left( u^{+}\left( x\right)
+\phi ^{\ast }\right) H\left( u^{+}\left( x\right) +\phi ^{\ast }\right)
=v^{\prime }\left( x\right) v\left( x\right) $, so we can estimate the first
term on the right hand side by 
\begin{eqnarray*}
2\int \psi w\left\vert \nabla _{A}u^{+}\right\vert \left\vert \nabla
_{A}\psi \right\vert  &\leq &2\int \psi vv^{\prime }\left\vert \nabla _{%
\tilde{A}}u^{+}\right\vert \left\vert \nabla _{\tilde{A}}\psi \right\vert
=2\int \psi v\left\vert \nabla _{\tilde{A}}v\right\vert \left\vert \nabla _{%
\tilde{A}}\psi \right\vert  \\
&\leq &\frac{1}{2}\int \psi ^{2}\left\vert \nabla _{\tilde{A}}v\right\vert
^{2}+2\int \left\vert \nabla _{\tilde{A}}\psi \right\vert ^{2}v^{2},
\end{eqnarray*}%
Substituting above and absorbing into the left, we obtain%
\begin{equation}
\int \psi ^{2}\left\vert \nabla _{\tilde{A}}v\right\vert ^{2}\leq 4\int
\left\vert \nabla _{\tilde{A}}\psi \right\vert ^{2}v^{2}+2\int \psi
^{2}w\phi _{0}+2\int \nabla _{A}\left( \psi ^{2}w\right) \cdot \vec{\phi}%
_{1}.  \label{weak-eq}
\end{equation}

Now, since $\left( \phi _{0},\vec{\phi}_{1}\right) $ is admissible, we have
that%
\begin{eqnarray}
\left\vert \int \psi ^{2}w\phi _{0}\right\vert +\left\vert \int \nabla
_{A}\left( \psi ^{2}w\right) \cdot \vec{\phi}_{1}\right\vert &\leq &\phi
^{\ast }\int \left\vert \nabla _{A}\left( \psi ^{2}w\right) \right\vert 
\notag \\
&\leq &2\phi ^{\ast }\int \psi \left\vert \nabla _{A}\psi \right\vert w+\phi
^{\ast }\int \psi ^{2}\left\vert \nabla _{A}w\right\vert .  \label{we-RHS}
\end{eqnarray}%
We assume now that $\phi ^{\ast }>0$, if this is not the case, then we
substitute $\phi ^{\ast }$ by a small constant $c>0$ and let $c\rightarrow 0$
at the end of the proof. By the inequality $h^{\prime }\left( t\right) \leq
C_{h}\frac{h\left( t\right) }{t}$and the definition of $H$ we have that%
\begin{equation}
H^{\prime }\left( t\right) =\left\{ 
\begin{array}{ll}
h^{\prime }\left( t\right) & t\in \left[ \phi ^{\ast },N\right] \\ 
h^{\prime }\left( N\right) & t>N%
\end{array}%
\right. \leq C_{h}\left\{ 
\begin{array}{ll}
\frac{h\left( t\right) }{t} & t\in \left[ \phi ^{\ast },N\right] \\ 
\frac{h\left( N\right) }{N} & t>N%
\end{array}%
\right. \leq C_{h}\frac{h\left( t\right) }{\phi ^{\ast }}.  \label{we-Hp}
\end{equation}%
Then by (\ref{we-H}) we have that $v^{\prime }\left( x\right) \leq C_{h}%
\frac{h\left( u^{+}\left( x\right) +\phi ^{\ast }\right) }{\phi ^{\ast }}$,
and writing $\tilde{v}\left( x\right) =h\left( u^{+}\left( x\right) +\phi
^{\ast }\right) $, the first term on the right of (\ref{we-RHS}) is bounded
by%
\begin{eqnarray*}
2\phi ^{\ast }\int \psi \left\vert \nabla _{A}\psi \right\vert w &\leq
&2\phi ^{\ast }\int \psi \left\vert \nabla _{A}\psi \right\vert vv^{\prime
}\leq 2C_{h}\int \psi \left\vert \nabla _{A}\psi \right\vert v\tilde{v} \\
&\leq &C_{h}\int \left( \psi ^{2}\tilde{v}^{2}+\left\vert \nabla _{A}\psi
\right\vert ^{2}v^{2}\right) .
\end{eqnarray*}%
Similarly, the second term on the right of (\ref{we-RHS}) is bounded by%
\begin{eqnarray*}
\phi ^{\ast }\int \psi ^{2}\left\vert \nabla _{A}w\right\vert &=&\phi ^{\ast
}\int \psi ^{2}\left( v^{\prime }\right) ^{2}\left\vert \nabla
_{A}u^{+}\right\vert =\phi ^{\ast }\int \psi ^{2}v^{\prime }\left\vert
\nabla _{A}v\right\vert \leq C_{h}\int \psi ^{2}\tilde{v}\left\vert \nabla
_{A}v\right\vert \\
&\leq &\frac{\lambda }{4}\int \psi ^{2}\left\vert \nabla _{A}v\right\vert
^{2}+\frac{C_{h}^{2}}{\lambda }\int \psi ^{2}\tilde{v}^{2}.
\end{eqnarray*}%
where $\lambda >0$ is as in (\ref{struc_0}) and we also used (\ref{we-Hp}).
Plugging these estimates into (\ref{we-RHS}) and substituting into (\ref%
{weak-eq}) yields%
\begin{equation*}
\int \psi ^{2}\left\vert \nabla _{\tilde{A}}v\right\vert ^{2}\leq 4\int
\left\vert \nabla _{\tilde{A}}\psi \right\vert ^{2}v^{2}+2C_{h}\int \left(
\psi ^{2}\tilde{v}^{2}+\left\vert \nabla _{A}\psi \right\vert
^{2}v^{2}\right) +\frac{\lambda }{2}\int \psi ^{2}\left\vert \nabla
_{A}v\right\vert ^{2}+\frac{2C_{h}^{2}}{\lambda }\int \psi ^{2}\tilde{v}^{2}.
\end{equation*}%
Using the structural assumptions (\ref{struc_0}) yields%
\begin{equation*}
\lambda \int \psi ^{2}\left\vert \nabla _{A}v\right\vert ^{2}\leq 4\Lambda
\int \left\vert \nabla _{A}\psi \right\vert ^{2}v^{2}+2C_{h}\int \left( \psi
^{2}\tilde{v}^{2}+\left\vert \nabla _{A}\psi \right\vert ^{2}v^{2}\right) +%
\frac{\lambda }{2}\int \psi ^{2}\left\vert \nabla _{A}v\right\vert ^{2}+%
\frac{2C_{h}^{2}}{\lambda }\int \psi ^{2}\tilde{v}^{2},
\end{equation*}%
absorbing in to the left we obtain%
\begin{equation*}
\int \psi ^{2}\left\vert \nabla _{A}v\right\vert ^{2}\leq 16C_{h}^{2}\left( 
\frac{\Lambda }{\lambda }+\frac{1}{\lambda ^{2}}\right) \int \left( \psi
^{2}+\left\vert \nabla _{A}\psi \right\vert ^{2}\right) \tilde{v}^{2},
\end{equation*}%
where we used the inequality $v\left( x\right) =H\left( u^{+}\left( x\right)
+\phi ^{\ast }\right) \leq h\left( u^{+}\left( x\right) +\phi ^{\ast
}\right) =\tilde{v}\left( x\right) $. This is%
\begin{equation*}
\int \psi ^{2}\left\vert \nabla _{A}\left[ H\left( u^{+}+\phi ^{\ast
}\right) \right] \right\vert ^{2}dx\leq C_{\lambda ,\Lambda }C_{h}^{2}\int
\left( \psi ^{2}+\left\vert \nabla _{A}\psi \right\vert ^{2}\right) \left(
h\left( u^{+}\left( x\right) +\phi ^{\ast }\right) \right) ^{2},
\end{equation*}%
the lemma the follows in this case by by letting $N\rightarrow \infty $.

When $u$ is a weak supersolution to $L_{\tilde{A}}u=\phi _{0}-\mathrm{div}_{A}%
\vec{\phi}_{1}$ in $B$, then $-u$ is a weak subsolution to $L\left(
-u\right) =-\phi _{0}-\mathrm{div}_{A}\left( -\vec{\phi}_{1}\right) $ with the
same admissible norm $\phi ^{\ast }$, and $\left( -u\right) ^{+}=u^{-}$, so (%
\ref{reverse Sobs}) holds in this case with $u^{+}$ replaced by $u^{-}$.
\end{proof}

\begin{remark}
\label{remark-power-zero}Taking $h\left( t\right) \equiv t$ in Lemma \ref%
{reverse Sobolev} we have that 
\begin{equation*}
\int_{B(0,r)}\psi ^{2}\left\vert \nabla _{A}u^{+}\right\vert ^{2}dx\leq
C_{\lambda ,\Lambda }\int_{B(0,r)}\left( u^{+}+\phi ^{\ast }\right)
^{2}\left( |\nabla _{A}\psi |^{2}+\psi ^{2}\right) ~dx
\end{equation*}%
when $u$ is a subsolution to $L_{\tilde{A}}u=\phi _{0}-\mathrm{div}_{A}\vec{%
\phi}_{1}$, and the same estimate holds for $u^{-}$ ($\left\vert
u\right\vert $) when $u$ is a supersolution (solution).
\end{remark}

The following variation of Caccioppoli requires stronger hypotheses on the
function $h$, however $h$ is allowed to be \emph{decreasing} when applied to
supersolutions. In particular, $h$ needs to be $C^{1,1}$ since the second
derivative of $h$ explicitly appears within the integrals in the
calculations. When $h$ is $C^{1,1}$ the second derivative may be
discontinuous (piece-wise discontinuous in our applications)\ but
discontinuities will only be jump discontinuities, which do not affect the
integrals.

\begin{lemma}
\label{Cacc-positive} Assume that $u\in W_{A}^{1,2}\left( \Omega \right) $
is a nonnegative weak subsolution or supersolution to $L_{\tilde{A}}u=\phi
_{0}-\mathrm{div}_{A}\vec{\phi}_{1}$ in $B(0,r)$, where $\left( \phi _{0},\vec{%
\phi}_{1}\right) $ is an $A$-admissible pair with norm $\phi ^{\ast }$ and $%
\tilde{A}$ satisfies the equivalences (\ref{struc_0}) for some $0<\lambda
\leq \Lambda <\infty $. Let $h(t)\geq 0$ be a convex monotonic $C^{1}$ and
piecewise twice continuously differentiable function on $\left( 0,\infty
\right) $ that satisfies the following conditions except possibly at
finitely many points when $t\in \left( 0,\infty \right) $:%
\renewcommand{\theenumi}{\Roman{enumi}}%

\begin{enumerate}
\item \label{I}$\Upsilon \left( t\right) =h\left( t\right) h^{\prime \prime
}\left( t\right) +\left( h^{\prime }\left( t\right) \right) ^{2}$ satisfies $%
c_{1}\left( h^{\prime }\left( t\right) \right) ^{2}\leq \Upsilon \left(
t\right) \leq C_{1}\left( h^{\prime }\left( t\right) \right) ^{2}$ at every
point of continuity of $h^{\prime \prime }$, where $0<c_{1}\leq 1\leq
C_{1}<\infty $ are constant;

\item \label{II}The derivative $h^{\prime }\left( t\right) $ satisfies the
inequality $0<\left\vert h^{\prime }\left( t\right) \right\vert \leq C_{2}%
\frac{h\left( t\right) }{t}$, where $C_{2}\geq 1$ is a constant;\newline
Furthermore, we assume that

\item \label{III}if $u$ is a weak subsolution then $h^{\prime }\geq 0$, and
if $u$ is a weak supersolution then $h^{\prime }\leq 0$.\newline
Then the following reverse Sobolev inequality holds for any $\psi \in 
\mathrm{Lip}_{\mathrm{c}}(B(0,r))$: 
\begin{equation}
\int_{B\left( x,r\right) }\psi ^{2}\left\vert \nabla _{A}\left[ h\left(
u+\phi ^{\ast }\right) \right] \right\vert ^{2}dx\leq C_{\lambda ,\Lambda }%
\frac{C_{1}^{2}C_{2}^{2}}{c_{1}^{2}}\int_{B(x,r)}h\left( u+\phi ^{\ast
}\right) ^{2}\left( \left\vert \nabla _{A}\psi \right\vert ^{2}+\psi
^{2}\right) .  \label{reverse Sobs positive}
\end{equation}
\end{enumerate}
\end{lemma}

\begin{proof}
We will prove the lemma with an extra assumption that $h^{\prime }(t)$ is
bounded and $h\left( u+\phi ^{\ast }\right) \in L^{2}\left( B\left(
0,r\right) \right) $. These assumptions can be dropped by the following
limiting argument. Using standard truncations as in \cite{SaWh4}. If $h$ is
increasing we define for $N\gg 1$,%
\begin{equation*}
h_{N}\left( t\right) \equiv \left\{ 
\begin{array}{lcc}
h\left( t\right)  & \text{ if } & 0\leq t\leq N \\ 
h\left( N\right) +h_{-}^{\prime }\left( N\right) \left( t-N\right)  & \text{
if } & t\geq N%
\end{array}%
\right. .
\end{equation*}%
while if $h$ is decreasing we let%
\begin{equation*}
h_{N}\left( t\right) \equiv \left\{ 
\begin{array}{lcc}
h\left( \frac{1}{N}\right) +h_{-}^{\prime }\left( \frac{1}{N}\right) \left(
t-\frac{1}{N}\right)  & \text{ if } & 0\leq t\leq \frac{1}{N} \\ 
h\left( t\right)  & \text{ if } & t\geq \frac{1}{N}%
\end{array}%
\right. 
\end{equation*}%
We note that either function $h_{N}$ still satisfy conditions (\ref{I})-(\ref%
{III}) in the lemma with the same constants $C_{1}$ and $C_{2}$, if we can
obtain a reverse Sobolev inequality similar to (\ref{reverse Sobs positive})
for $h_{N}$, then the dominated converge theorem applies to establish (\ref%
{reverse Sobs positive}) in general. Moreover, note that since $h_{N}$ is
linear for large $t$ when $h$ is increasing and for small $t$ when $h$ is
decreasing, then $h\left( u+\phi ^{\ast }\right) \in L^{2}\left( B\left(
0,r\right) \right) \iff u\in L^{2}\left( B\left( 0,r\right) \right) $.
Hence, if $\phi ^{\ast }>0$ from (\ref{II}) it follows that also $\left(
h^{\prime }\left( u+\phi ^{\ast }\right) \right) ^{2}$ and $\Upsilon \left(
u+\phi ^{\ast }\right) \in L^{1}\left( B\left( 0,r\right) \right) $. If $%
\phi ^{\ast }=0$ we replace it by a small positive $\varepsilon >0$ and then
let $\varepsilon \rightarrow 0$ at the end. Thus, in what follows we will
assume that $h^{\prime }\left( t\right) $ and $h^{\prime \prime }\left(
t\right) $ are bounded on the range of $u+\phi ^{\ast }$, and that all
integrals below are finite.

Assume that $h$ is $C^{1}$, convex, and piecewise twice differentiable in $%
\left( 0,\infty \right) $ with bounded first and second derivatives. By
these assumptions it follows that $h$ is twice differentiable everywhere
except a finitely many points where $h^{\prime \prime }$ has finite jump
discontinuities.

Let $\psi \in \mathrm{Lip}_{\mathrm{c}}(B(0,r))$, $v\left(
x\right) =h\left( u\left( x\right) +\phi ^{\ast }\right) $ and write $%
v^{\prime }\left( x\right) =h^{\prime }\left( u\left( x\right) +\phi ^{\ast
}\right) $, $v^{\prime \prime }\left( x\right) =h^{\prime \prime }\left(
u\left( x\right) +\phi ^{\ast }\right) $. Then we have that $w\left(
x\right) =\psi ^{2}\left( x\right) v\left( x\right) v^{\prime }\left(
x\right) $ is in the space $W_{A,0}^{1,2}\left( B(0,r)\right) $. Now, by
assumption (\ref{III}) we have that $w\geq 0$ when $u$ is a subsolution, and 
$w\leq 0$ when $u$ is a supersolution, then we have 
\begin{equation}
\int \nabla _{\tilde{A}}u\cdot \nabla _{\tilde{A}}w\leq \int w\phi _{0}+\int
\nabla _{A}w\cdot \vec{\phi}_{1}  \label{weakequation'}
\end{equation}%
Since $\nabla _{\tilde{A}}v=v^{\prime }\nabla _{\tilde{A}}u$, and $\left(
v^{\prime }\right) ^{2}+vv^{\prime \prime }=\Upsilon \left( u+\phi ^{\ast
}\right) $, the left side of (\ref{weakequation'}) equals%
\begin{eqnarray*}
\int \nabla _{\tilde{A}}u\cdot \nabla _{\tilde{A}}w &=&\int \nabla _{\tilde{A%
}}u\cdot v^{\prime }\nabla _{\tilde{A}}\left( \psi ^{2}v\right) +\int \psi
^{2}vv^{\prime \prime }\nabla _{\tilde{A}}u\cdot \nabla _{\tilde{A}}u \\
&=&\int \nabla _{\tilde{A}}v\cdot \nabla _{\tilde{A}}\left( \psi
^{2}v\right) +\int \psi ^{2}vv^{\prime \prime }\left\vert \nabla _{\tilde{A}%
}u\right\vert ^{2} \\
&=&2\int \psi v\nabla _{\tilde{A}}v\cdot \nabla _{\tilde{A}}\psi +\int \psi
^{2}\Upsilon \left( u+\phi ^{\ast }\right) \left\vert \nabla _{\tilde{A}%
}u\right\vert ^{2}.
\end{eqnarray*}%
Combining this and (\ref{weakequation'}), we obtain%
\begin{equation}
\int \psi ^{2}\Upsilon \left( u+\phi ^{\ast }\right) \left\vert \nabla _{%
\tilde{A}}u\right\vert ^{2}\leq -2\int \psi v\nabla _{\tilde{A}}v\cdot
\nabla _{\tilde{A}}\psi +\int w\phi _{0}+\int \nabla _{A}w\cdot \vec{\phi}%
_{1}.  \label{weakequation''}
\end{equation}%
By property (\ref{I}) and the equivalences (\ref{struc_0}) we obtain:%
\begin{equation*}
c_{1}\lambda \int \psi ^{2}\left\vert \nabla _{A}v\right\vert ^{2}\leq
2\Lambda \int \psi v\left\vert \nabla _{A}v\right\vert \left\vert \nabla
_{A}\psi \right\vert +\int w\phi _{0}+\int \nabla _{A}w\cdot \vec{\phi}_{1}
\end{equation*}%
By Schwartz inequality we can estimate the first term on the right hand side
by 
\begin{equation*}
2\Lambda \int \psi v\left\vert \nabla _{A}v\right\vert \left\vert \nabla
_{A}\psi \right\vert \leq \frac{c_{1}\lambda }{2}\int \psi ^{2}\left\vert
\nabla _{A}v\right\vert ^{2}+\frac{4\Lambda ^{2}}{c_{1}\lambda }\int
v^{2}\left\vert \nabla _{A}\psi \right\vert ^{2}.
\end{equation*}%
Substituting above and absorbing into the left, we obtain%
\begin{equation}
\frac{c_{1}\lambda }{2}\int \psi ^{2}\left\vert \nabla _{A}v\right\vert
^{2}\leq \frac{4\Lambda ^{2}}{c_{1}\lambda }\int v^{2}\left\vert \nabla
_{A}\psi \right\vert ^{2}+\int w\phi _{0}+\int \nabla _{A}w\cdot \vec{\phi}%
_{1}.  \label{weake}
\end{equation}%
Since $\left( \phi _{0},\vec{\phi}_{1}\right) $ is admissible, we have that%
\begin{eqnarray*}
\left\vert \int \psi ^{2}vv^{\prime }\phi _{0}\right\vert +\left\vert \int
\nabla _{A}\left( \psi ^{2}vv^{\prime }\right) \cdot \vec{\phi}%
_{1}\right\vert &\leq &\phi ^{\ast }\int \left\vert \nabla _{A}\left( \psi
^{2}vv^{\prime }\right) \right\vert \\
&\leq &2\phi ^{\ast }\int \psi \left\vert \nabla _{A}\psi \right\vert
v\left\vert v^{\prime }\right\vert +\phi ^{\ast }\int \psi ^{2}\Upsilon
\left( u+\phi ^{\ast }\right) \left\vert \nabla _{A}u\right\vert
\end{eqnarray*}%
By property (\ref{II}) we have that $\left\vert v^{\prime }\right\vert
=\left\vert h^{\prime }\left( u+\phi ^{\ast }\right) \right\vert \leq C_{2}%
\frac{h\left( u+\phi ^{\ast }\right) }{u+\phi ^{\ast }}=C_{2}\frac{v}{u+\phi
^{\ast }}$; applying this to the first term on the right, and properties (%
\ref{I}-\ref{II}) to the second, we obtain%
\begin{eqnarray}
&&\left\vert \int w\phi _{0}\right\vert +\left\vert \int \nabla _{A}w\cdot 
\vec{\phi}_{1}\right\vert  \notag \\
&\leq &2C_{2}\phi ^{\ast }\int \psi \left\vert \nabla _{A}\psi \right\vert 
\frac{v^{2}}{u+\phi ^{\ast }}+C_{1}\phi ^{\ast }\int \psi ^{2}\left(
v^{\prime }\right) ^{2}\left\vert \nabla _{A}u\right\vert  \notag \\
&\leq &2C_{2}\int \psi \left\vert \nabla _{A}\psi \right\vert
v^{2}+C_{1}C_{2}\phi ^{\ast }\int \psi ^{2}\frac{v}{u+\phi ^{\ast }}%
\left\vert \nabla _{A}v\right\vert .  \notag \\
&\leq &C_{2}\int \left( \psi ^{2}+\left\vert \nabla _{A}\psi \right\vert
^{2}\right) v^{2}+C_{1}C_{2}\int \psi ^{2}v\left\vert \nabla _{A}v\right\vert
\notag \\
&\leq &C_{2}\int \left( \psi ^{2}+\left\vert \nabla _{A}\psi \right\vert
^{2}\right) v^{2}+\frac{2C_{1}^{2}C_{2}^{2}}{c_{1}\lambda }\int \psi
^{2}v^{2}+\frac{c_{1}\lambda }{4}\int \psi ^{2}\left\vert \nabla
_{A}v\right\vert ^{2}.  \label{weak-w}
\end{eqnarray}%
Replacing this on the right of (\ref{weake}) and operating yields%
\begin{eqnarray*}
\int \psi ^{2}\left\vert \nabla _{A}v\right\vert ^{2} &\leq &\frac{16\Lambda
^{2}}{c_{1}^{2}\lambda ^{2}}\int v^{2}\left\vert \nabla _{A}\psi \right\vert
^{2}+4\frac{C_{2}}{c_{1}\lambda }\int \left( \psi ^{2}+\left\vert \nabla
_{A}\psi \right\vert ^{2}\right) v^{2}+\frac{8C_{1}^{2}C_{2}^{2}}{%
c_{1}^{2}\lambda ^{2}}\int \psi ^{2}v^{2} \\
&\leq &16\frac{C_{1}^{2}C_{2}^{2}}{c_{1}^{2}}\left( \frac{\Lambda ^{2}}{%
\lambda ^{2}}+\frac{1}{\lambda ^{2}}+\frac{1}{\lambda }\right) \int \left(
\psi ^{2}+\left\vert \nabla _{A}\psi \right\vert ^{2}\right) v^{2}.
\end{eqnarray*}
\end{proof}

\section{Preliminaries on Young functions\label{section-Orlicz}}

\renewcommand{\theenumi}{\arabic{enumi}}%
\setcounter{enumi}{0}\RESUME%
In this section introduce some basic concepts from Orlicz spaces and define
the particular families of Young function that we will use in your
applications. We also compute successive compositions of these functions and
their inverses and obtain estimates for their derivatives.

\subsection{The Orlicz norm and the Orlicz quasidistance\label{section-quasi}%
}

Suppose that $\mu $ is a $\sigma $-finite measure on a set $X$, and $\Phi :%
\left[ 0,\infty \right) \rightarrow \left[ 0,\infty \right) $ is a Young
function, which for our purposes is an increasing convex piecewise
differentiable (meaning there are at most finitely many points where the
derivative of $\Phi $ may fail to exist, but right and left hand derivatives
exist everywhere) function such that $\Phi \left( 0\right) =0$. Since $\Phi $
is convex it follows that $\Phi $ is absolutely continuous and $\Phi
^{\prime }\left( t\right) $ is a nonnegative nondecreasing function, in fact 
$\Phi ^{\prime }\left( t\right) >0$ for almost all $t>0$. The homogeneous
Luxemburg norm associated to a Young function $\Phi $ is given by%
\begin{equation}
\left\Vert f\right\Vert _{L^{\Phi }\left( X,d\mu \right) }=\inf \left\{
t>0:\int_{X}\Phi \left( \frac{\left\vert f\right\vert }{t}\right) d\mu \leq
1\right\} \in \left[ 0,\infty \right] ,  \label{Lux}
\end{equation}%
where it is understood that $\inf \left( \emptyset \right) =\infty $. The
completion of the space of $\mu $-measurable functions in $X$ with respect
to this norm is the Orlicz space $L^{\Phi }\left( X,d\mu \right) $ which is
a Banach space by definition. The conjugate Young function $\Phi ^{\ast }$
is defined through the relation $\left( \Phi ^{\ast }\right) ^{\prime
}=\left( \Phi ^{\prime }\right) ^{-1}$ and it can be used to give an
equivalent norm%
\begin{equation*}
\left\Vert f\right\Vert _{L_{\ast }^{\Phi }\left( \mu \right) }\equiv \sup
\left\{ \int_{X}\left\vert fg\right\vert d\mu :\int_{X}\Phi ^{\ast }\left(
\left\vert g\right\vert \right) d\mu \leq 1\right\} .
\end{equation*}%
The conjugate function $\Phi ^{\ast }$ is equivalently defined as%
\begin{equation}
\Phi ^{\ast }\left( s\right) =\sup_{t>0}st-\Phi \left( t\right) ,\qquad 
\text{for all }s>0.  \label{Phi-conj}
\end{equation}

Given a Young function $\Phi $ and a measure $\mu $ we will define a
non-homogeneous norm as follows. We let $L_{\ast }^{\Phi }\left( \mu \right) 
$ be the set of measurable functions $f:X\rightarrow \mathbb{R}$ such that
the integral%
\begin{equation*}
\int_{X}\Phi \left( \left\vert f\right\vert \right) ~d\mu ,
\end{equation*}%
is finite, where as usual, functions that agree almost everywhere are
identified. The set $L_{\ast }^{\Phi }\left( \mu \right) $ may not in
general be closed under scalar multiplication, but if $\Phi $ is $K$%
-submultiplicative for some constant $K>0$, i.e. 
\begin{equation*}
\Phi \left( st\right) \leq K\Phi \left( s\right) \Phi \left( t\right) \qquad 
\text{for all }s,t\geq 0
\end{equation*}%
then clearly $\int_{X}\Phi \left( \left\vert Cf\right\vert \right) ~d\mu
\leq K\Phi \left( C\right) \int_{X}\Phi \left( \left\vert f\right\vert
\right) ~d\mu $ and $L_{\ast }^{\Phi }\left( \mu \right) $ is a vector space
because if $f,g\in L_{\ast }^{\Phi }\left( \mu \right) $ then%
\begin{eqnarray*}
\int_{X}\Phi \left( \left\vert f+g\right\vert \right) ~d\mu
&=&\int_{0}^{\infty }\Phi ^{\prime }\left( t\right) \mu \left\{ \left\vert
f+g\right\vert >t\right\} ~dt \\
&\leq &\int_{0}^{\infty }\Phi ^{\prime }\left( t\right) \mu \left\{
\left\vert f\right\vert >\frac{t}{2}\right\} ~dt+\int_{0}^{\infty }\Phi
^{\prime }\left( t\right) \mu \left\{ \left\vert g\right\vert >\frac{t}{2}%
\right\} ~dt \\
&=&\int_{X}\Phi \left( 2\left\vert f\right\vert \right) ~d\mu +\int_{X}\Phi
\left( 2\left\vert g\right\vert \right) ~d\mu \\
&<&K\Phi \left( 2\right) \left\{ \int_{X}\Phi \left( \left\vert f\right\vert
\right) ~d\mu +\int_{X}\Phi \left( \left\vert g\right\vert \right) ~d\mu
\right\} <\infty .
\end{eqnarray*}%
We claim that if $\Phi $ is an $A$-submultiplicative Young function then the
function 
\begin{equation}
\left\Vert f\right\Vert _{\mathcal{D}^{\Phi }\left( \mu \right) }:=\Phi
^{-1}\left( \int_{X}\Phi \left( \left\vert f\right\vert \right) ~d\mu \right)
\label{quasi-metric}
\end{equation}%
is a \emph{nonhomogeneous quasi-norm }in $L_{\ast }^{\Phi }\left( \mu
\right) $, that is, $\left\Vert \cdot \right\Vert _{\mathcal{D}^{\Phi
}\left( \mu \right) }:L_{\ast }^{\Phi }\left( \mu \right) \rightarrow \left[
0,\infty \right) $ satisfies%
\begin{eqnarray*}
\left\Vert f\right\Vert _{\mathcal{D}^{\Phi }\left( \mu \right) } &=&0\qquad
\iff \qquad f\equiv 0 \\
\left\Vert f+g\right\Vert _{\mathcal{D}^{\Phi }\left( \mu \right) } &\leq
&C_{\Phi }\left( \left\Vert f\right\Vert _{\mathcal{D}^{\Phi }\left( \mu
\right) }+\left\Vert g\right\Vert _{\mathcal{D}^{\Phi }\left( \mu \right)
}\right) .
\end{eqnarray*}%
Indeed, it is clear that $\left\Vert f\right\Vert _{\mathcal{D}^{\Phi
}\left( \mu \right) }\geq 0$ and $\left\Vert f\right\Vert _{\mathcal{D}%
^{\Phi }\left( \mu \right) }=0\iff f=0$, and that $\left\Vert f-g\right\Vert
_{\mathcal{D}^{\Phi }\left( \mu \right) }=\left\Vert g-f\right\Vert _{%
\mathcal{D}^{\Phi }\left( \mu \right) }$. From the above computation we also
have that 
\begin{eqnarray*}
\Phi \left( \left\Vert f+g\right\Vert _{\mathcal{D}^{\Phi }\left( \mu
\right) }\right) &=&\int_{X}\Phi \left( \left\vert f+g\right\vert \right)
~d\mu \leq K\Phi \left( 2\right) \left\{ \int_{X}\Phi \left( \left\vert
f\right\vert \right) ~d\mu +\int_{X}\Phi \left( \left\vert g\right\vert
\right) ~d\mu \right\} \\
&=&K\Phi \left( 2\right) \left\{ \Phi \left( \left\Vert f\right\Vert _{%
\mathcal{D}^{\Phi }\left( \mu \right) }\right) +\Phi \left( \left\Vert
g\right\Vert _{\mathcal{D}^{\Phi }\left( \mu \right) }\right) \right\} \\
&\leq &2K\Phi \left( 2\right) ~\Phi \left( \left\Vert f\right\Vert _{%
\mathcal{D}^{\Phi }\left( \mu \right) }+\left\Vert g\right\Vert _{\mathcal{D}%
^{\Phi }\left( \mu \right) }\right) \\
&\leq &\Phi \left( 2K\Phi \left( 2\right) \left\{ \left\Vert f\right\Vert _{%
\mathcal{D}^{\Phi }\left( \mu \right) }+\left\Vert g\right\Vert _{\mathcal{D}%
^{\Phi }\left( \mu \right) }\right\} \right)
\end{eqnarray*}%
where we used that $\Phi $ is increasing and that $C\Phi \left( t\right)
\leq \Phi \left( Ct\right) $ since $\Phi $ is increasing convex with $\Phi
\left( 0\right) =0$. Thus, we have%
\begin{equation*}
\left\Vert f+g\right\Vert _{\mathcal{D}^{\Phi }\left( \mu \right) }\leq
C_{\Phi }\left( \left\Vert f\right\Vert _{\mathcal{D}^{\Phi }\left( \mu
\right) }+\left\Vert g\right\Vert _{\mathcal{D}^{\Phi }\left( \mu \right)
}\right) \qquad \text{for all }f,g\in L_{\ast }^{\Phi }\left( \mu \right) .
\end{equation*}%
The same proof provides an inequality for any general finite sum of
functions $\sum_{j=1}^{N}f_{j}$ :%
\begin{equation}
\left\Vert \sum_{j=1}^{N}f_{j}\right\Vert _{\mathcal{D}^{\Phi }\left( \mu
\right) }\leq C_{\Phi ,N}\left( \sum_{j=1}^{N}\left\Vert f_{j}\right\Vert _{%
\mathcal{D}^{\Phi }\left( \mu \right) }\right) \qquad \text{whenever }%
f_{j}\in L_{\ast }^{\Phi }\left( \mu \right) \text{, }j=1,\cdots ,N,
\label{finite}
\end{equation}
where $C_{\Phi ,N}=NK~\Phi \left( N\right) $.

The function $\left\Vert \cdot \right\Vert _{\mathcal{D}^{\Phi }\left( \mu
\right) }$ in general would not be a quasinorm because it may fail to be
absolutely homogeneous, i.e., in general$\left\Vert Cf\right\Vert _{\mathcal{%
D}^{\Phi }\left( \mu \right) }=\left\vert C\right\vert \left\Vert
f\right\Vert _{\mathcal{D}^{\Phi }\left( \mu \right) }$ may not hold. It is
clear though that $d_{\Phi }\left( f,g\right) =\left\Vert f-g\right\Vert _{%
\mathcal{D}^{\Phi }\left( \mu \right) }$ is a quasi-distance in $L_{\ast
}^{\Phi }\left( \mu \right) $, i.e. the function $d_{\Phi }\left( \cdot
,\cdot \right) :L_{\ast }^{\Phi }\left( \mu \right) \times L_{\ast }^{\Phi
}\left( \mu \right) \rightarrow \left[ 0,\infty \right) $ is symmetric, $%
d_{\Phi }\left( f,g\right) =0\iff f\equiv g$, and satisfies a triangle
inequality with a constant $C_{\Phi }$ that may be bigger than 1. We note
that the same conclusion may be reached if $\Phi $ is $K$%
-supermultiplicative, i.e.%
\begin{equation*}
K\Phi \left( st\right) \geq \Phi \left( s\right) \Phi \left( t\right) \qquad 
\text{for all }s,t>0.
\end{equation*}%
Indeed, we have that for any $C>0$ and $f\in L_{\ast }^{\Phi }\left( \mu
\right) $%
\begin{equation*}
\int_{X}\Phi \left( \left\vert Cf\right\vert \right) ~d\mu =\frac{1}{\Phi
\left( \frac{1}{C}\right) }\int_{X}\Phi \left( \frac{1}{C}\right) \Phi
\left( \left\vert Cf\right\vert \right) ~d\mu \leq \frac{K}{\Phi \left( 
\frac{1}{C}\right) }\int_{X}\Phi \left( \left\vert f\right\vert \right)
~d\mu <\infty ,
\end{equation*}%
and it similarly follows as above that $f+g\in L_{\ast }^{\Phi }\left( \mu
\right) $ for all $f,g\in L_{\ast }^{\Phi }\left( \mu \right) $. We have
shown the following:

\begin{proposition}
\label{prop quasi}If $\Phi $ is a $K$-submultiplicative or $K$%
-supermultiplicative Young function in $\left[ 0,\infty \right) $ for some $%
K>0$ then the space 
\begin{equation*}
L_{\ast }^{\Phi }\left( \mu \right) =\left\{ f:\int_{X}\Phi \left(
\left\vert f\right\vert \right) ~d\mu <\infty \right\}
\end{equation*}%
is a vector space and the function $\left\Vert \cdot \right\Vert _{\mathcal{D%
}^{\Phi }\left( \mu \right) }:L_{\ast }^{\Phi }\left( \mu \right)
\rightarrow \left[ 0,\infty \right) $ defined in (\ref{quasi-metric}) is a
nonhomogeneous quasi-norm in $L_{\ast }^{\Phi }\left( \mu \right) $.
\end{proposition}

In this paper we consider Young functions which satisfy the hypotheses of
the above proposition, so our Moser iteration may be considered as an
iteration scheme in quasi-metric spaces. The homogeneity of the norm $%
\left\Vert f\right\Vert _{L^{\Phi }\left( \mu \right) }$ is not that
important, but rather it is the iteration of Orlicz expressions that is
critical. The following lemma shows the relations between the Orlicz norm
and the quasi-norm when the Young function is sub- or supermultiplicative.

\begin{lemma}
\label{lem:subsup}If a Young function ${\Phi }$ is $K$-submultiplicative for
some constant $K\geq 1$, then 
\begin{equation*}
\Phi ^{\left( -1\right) }\left( \int_{B\left( x,\rho \right) }\Phi \left(
v\right) d\mu _{x,\rho }\right) \leq K\left\Vert v\right\Vert _{L^{\Phi
}\left( \mu _{x,\rho }\right) }.
\end{equation*}%
On the other hand, if $\Phi $ is a $K$-supermultiplicative Young function
for some $K\geq 1$, then%
\begin{equation*}
\left\Vert v\right\Vert _{L^{\Phi }\left( \mu _{x,\rho }\right) }\leq K\Phi
^{\left( -1\right) }\left( \int_{B\left( x,\rho \right) }\Phi \left(
v\right) d\mu _{x,\rho }\right) .
\end{equation*}
\end{lemma}

\begin{proof}
Recall that we have by definition%
\begin{equation*}
\left\Vert v\right\Vert _{L^{\Phi }\left( \mu _{x,\rho }\right) }=\inf
\left\{ t>0:\int_{B\left( x,\rho \right) }\Phi \left( \frac{\left\vert
v\right\vert }{t}\right) ~d\mu _{x,\rho }\leq 1\right\} .
\end{equation*}%
Let $\kappa =\left\Vert v\right\Vert _{\mathcal{D}^{\Phi }\left( \mu \right)
}=\Phi ^{-1}\left( \int_{B\left( x,\rho \right) }\Phi \left( \left\vert
v\right\vert \right) ~d\mu _{x,\rho }\right) $, by the submultiplicativity
of $\Phi $ we have%
\begin{eqnarray*}
\Phi \left( \left\vert v\right\vert \right) &=&\Phi \left( \frac{\left\vert
v\right\vert }{\kappa }\cdot \kappa \right) \leq K\Phi \left( \frac{%
\left\vert v\right\vert }{\kappa }\right) \Phi \left( \kappa \right) =K\Phi
\left( \frac{\left\vert v\right\vert }{\kappa }\right) \int_{B\left( x,\rho
\right) }\Phi \left( \left\vert v\right\vert \right) ~d\mu _{x,\rho } \\
&\leq &\Phi \left( K\frac{\left\vert v\right\vert }{\kappa }\right)
\int_{B\left( x,\rho \right) }\Phi \left( \left\vert v\right\vert \right)
~d\mu _{x,\rho },
\end{eqnarray*}%
where we used that $C\Phi \left( t\right) \leq \Phi \left( Ct\right) $ for
all $C\geq 1$. Integrating gives%
\begin{equation*}
\int_{B\left( x,\rho \right) }\Phi \left( \left\vert v\right\vert \right)
~d\mu _{x,\rho }\leq \int_{B\left( x,\rho \right) }\Phi \left( \frac{%
\left\vert v\right\vert }{\kappa }\right) ~d\mu _{x,\rho }\cdot
\int_{B\left( x,\rho \right) }\Phi \left( \left\vert v\right\vert \right)
~d\mu _{x,\rho }
\end{equation*}%
so that $\int_{B\left( x,\rho \right) }\Phi \left( K\frac{\left\vert
v\right\vert }{\kappa }\right) ~d\mu _{x,\rho }\geq 1$, which yields $%
\left\Vert v\right\Vert _{L^{\Phi }\left( \mu _{x,\rho }\right) }\geq \frac{%
\kappa }{K}=\frac{1}{K}\Phi ^{-1}\left( \int_{B\left( x,\rho \right) }\Phi
\left( \left\vert v\right\vert \right) ~d\mu _{x,\rho }\right) $.

Now assume that $\Phi $ is a $K$-supermultiplicative Young function, i.e. $%
K\Phi \left( st\right) \geq \Phi \left( s\right) \Phi \left( t\right) $ for
all $s,t\geq 0$. We have%
\begin{eqnarray*}
K\int_{B\left( x,\rho \right) }\Phi \left( \left\vert v\right\vert \right)
~d\mu _{x,\rho } &=&K\int_{B\left( x,\rho \right) }\Phi \left( \left\Vert
v\right\Vert _{L^{\Phi }\left( \mu _{x,\rho }\right) }\frac{\left\vert
v\right\vert }{\left\Vert v\right\Vert _{L^{\Phi }\left( \mu _{x,\rho
}\right) }}\right) ~d\mu _{x,\rho } \\
&\geq &\Phi \left( \left\Vert v\right\Vert _{L^{\Phi }\left( \mu _{x,\rho
}\right) }\right) \int_{B\left( x,\rho \right) }\Phi \left( \frac{\left\vert
v\right\vert }{\left\Vert v\right\Vert _{L^{\Phi }\left( \mu _{x,\rho
}\right) }}\right) ~d\mu _{x,\rho } \\
&=&\Phi \left( \left\Vert v\right\Vert _{L^{\Phi }\left( \mu _{x,\rho
}\right) }\right) .
\end{eqnarray*}

Then, since $\frac{1}{C}\Phi \left( t\right) \leq \Phi \left( \frac{t}{C}%
\right) $ for all $0<C\leq 1$. Thus, we have that%
\begin{equation*}
\Phi ^{-1}\left( \int_{B\left( x,\rho \right) }\Phi \left( \left\vert
v\right\vert \right) ~d\mu _{x,\rho }\right) \geq \Phi ^{-1}\left( \frac{1}{K%
}\Phi \left( \left\Vert v\right\Vert _{L^{\Phi }\left( \mu _{x,\rho }\right)
}\right) \right) \geq \frac{\left\Vert v\right\Vert _{L^{\Phi }\left( \mu
_{x,\rho }\right) }}{K}.
\end{equation*}
\end{proof}

\subsection{Orlicz norms and admissibility}

The next proposition gives sufficient conditions for strong admissibility.

\begin{proposition}
\label{prop:admiss-suff}Given a right hand side pair $\left( \phi _{0},\vec{%
\phi}_{1}\right) $ defined in a bounded domain $\Omega $. Suppose that $\vec{%
\phi}_{1}\in L^{\infty }\left( \Omega \right) $ and that there exists a bump
function $\Phi $ and a constant $C_{\Omega }$ such that the global $\left(
\Phi ,A\right) $-Orlicz-Sobolev bump inequality (\ref{OS global}) holds, and
such that $\phi \in L^{\Phi ^{\ast }}\left( \Omega \right) $ where $\Phi
^{\ast }$ is the conjugate Young function to $\Phi $. Then $\left( \phi _{0},%
\vec{\phi}_{1}\right) $ is strongly admissible in $\Omega $ as given in
Definition \ref{def A admiss} with norm%
\begin{equation*}
\left\Vert \left( \phi ,\vec{\phi}_{1}\right) \right\Vert _{\mathcal{X}%
^{\ast }\left( \Omega \right) }\leq 2C_{\Omega }\left\Vert {\phi }%
_{0}\right\Vert _{L^{\Phi ^{\ast }}\left( \Omega \right) }+\left\Vert {\vec{%
\phi}}_{1}\right\Vert _{L^{\infty }\left( \Omega \right) }<\infty .
\end{equation*}
\end{proposition}

\begin{proof}
First, note that for any $v\in \mathrm{Lip}_{\mathrm{c}}\left( \Omega
\right) $ 
\begin{equation*}
\int_{\Omega }\left\vert \nabla _{A}v{\cdot \vec{\phi}}_{1}\right\vert {~}%
dx\leq \left\Vert {\vec{\phi}}_{1}\right\Vert _{L^{\infty }\left( \Omega
\right) }\left\Vert \nabla _{A}v\right\Vert _{L^{1}\left( \Omega \right) },
\end{equation*}%
so $\left\Vert \phi _{1}\right\Vert _{\mathcal{X}^{\ast }\left( \Omega
\right) }\leq \left\Vert {\vec{\phi}}_{1}\right\Vert _{L^{\infty }\left(
\Omega \right) }$. Next, by the Orlicz-H\"{o}lder inequality and (\ref{OS
ineq}), for any $v\in \mathrm{Lip}_{\mathrm{c}}\left( \Omega \right) $%
\begin{eqnarray*}
\int_{\Omega }\left\vert v{\phi }_{0}\right\vert {~}dx &\leq &2\left\Vert {%
\phi }_{0}\right\Vert _{L^{\Phi ^{\ast }}\left( \Omega \right) }\left\Vert
v\right\Vert _{L^{\Phi }\left( \Omega \right) } \\
&\leq &2C_{\Omega }\left\Vert {\phi }_{0}\right\Vert _{L^{\Phi ^{\ast
}}\left( \Omega \right) }\left\Vert \nabla _{A}v\right\Vert _{L^{1}\left(
\Omega \right) }
\end{eqnarray*}%
this is $\left\Vert \phi _{0}\right\Vert _{\mathcal{X}^{\ast }\left( B\left(
y,R_{0}\right) \right) }\leq 2C_{\Omega }\left\Vert {\phi }_{0}\right\Vert
_{L^{\Phi ^{\ast }}\left( \Omega \right) }$.
\end{proof}

\subsection{Submultiplicative extensions}

In our application to Moser iteration the convex bump function $\Phi \left(
t\right) $ is assumed to satisfy in addition:

\begin{itemize}
\item The function $\frac{\Phi (t)}{t}$ is positive, nondecreasing and tends
to $\infty $ as $t\rightarrow \infty $;

\item $\Phi$ is submultiplicative on an interval $\left( E,\infty \right) $
for some $E>1$: 
\begin{equation}
\Phi \left( ab\right) \leq \Phi \left( a\right) \Phi \left( b\right) ,\ \ \
\ \ a,b>E.  \label{submult}
\end{equation}
\end{itemize}

Note that if we consider more generally the quasi-submultiplicative
condition or $K$-submultiplicativity,%
\begin{equation}
\Phi \left( ab\right) \leq K\Phi \left( a\right) \Phi \left( b\right) ,\ \ \
\ \ a,b>E,  \label{submult quasi}
\end{equation}%
for some constant $K$, then $\Phi \left( t\right) $ satisfies (\ref{submult
quasi}) if and only if $\Phi _{K}\left( t\right) \equiv K\Phi \left(
t\right) $ satisfies (\ref{submult}). Thus we can always rescale a
quasi-submultiplicative function to be submultiplicative.

Now let us consider the \emph{linear extension} of $\Phi $ defined on $\left[
E,\infty \right) $ to the entire positive real axis $\left( 0,\infty \right) 
$ defined by%
\begin{equation*}
\Phi \left( t\right) =\frac{\Phi \left( E\right) }{E}t,\ \ \ \ \ 0\leq t\leq
E.
\end{equation*}%
We claim that this extension of $\Phi $ is submultiplicative on $%
\left(0,\infty \right) $, i.e. 
\begin{equation*}
\Phi \left( ab\right) \leq \Phi \left( a\right) \Phi \left( b\right) ,\ \ \
\ \ a,b>0.
\end{equation*}%
In fact, the identity $\Phi(t)/t = \Phi(\max\{t,E\})/\max\{t,E\}$ and the
monotonicity of $\Phi(t)/t$ imply 
\begin{equation*}
\frac{\Phi(ab)}{ab} \leq \frac{\Phi (\max\{a,E\} \max\{b, E\})}{\max\{a,E\}
\max\{b, E\}} \leq \frac{\Phi(\max\{a,E\})}{\max\{a,E\}} \cdot \frac{
\Phi(\max\{b,E\})}{\max\{b,E\}} = \frac{\Phi(a)}{a} \frac{\Phi(b)}{b}.
\end{equation*}

\begin{conclusion}
\label{sub extensions}If $\Phi :[E,\infty )\rightarrow {\mathbb{R}}^{+}$ is
a submultiplicative piecewise differentiable convex function so that $\Phi
(t)/t$ is nondecreasing, then we can extend $\Phi $ to a submultiplicative
piecewise differentiable convex function on $\left[ 0,\infty \right) $ that
vanishes at $0$ \emph{if and only if} 
\begin{equation}
\Phi ^{\prime }\left( E\right) \geq \frac{\Phi \left( E\right) }{E}.
\label{extension}
\end{equation}
\end{conclusion}

\subsection{An explicit family of Orlicz bumps\label{section-xYoung}}

We now consider the \emph{near} power bump case $\Phi \left( t\right) =\Phi
_{m}\left( t\right) =e^{\left( \left( \ln t\right) ^{\frac{1}{m}}+1\right)
^{m}}$ for $m>1$. In the special case that $m>1$ is an integer we can expand
the $m^{th}$ power in 
\begin{equation*}
\ln \Phi \left( e^{s}\right) =\left( s^{\frac{1}{m}}+1\right)
^{m}=\sum_{k=0}^{m}\left( 
\begin{array}{c}
m \\ 
k%
\end{array}%
\right) s^{\frac{k}{m}},
\end{equation*}%
and using the inequality $1\leq \left( \frac{s}{s+t}\right) ^{\alpha
}+\left( \frac{t}{s+t}\right) ^{\alpha }$ for $s,t>0$ and $0\leq \alpha \leq
1$, we see that $\Theta _{m}\left( s\right) \equiv \ln \Phi _{m}\left(
e^{s}\right) $ is subadditive on $\left( 0,\infty \right) $, hence $\Phi
_{m} $ is submultiplicative on $\left( 1,\infty \right) $. In fact, it is
not hard to see that for $m>1$, $\Theta _{m}\left( s\right) =\left( s^{\frac{%
1}{m}}+1\right) ^{m}$ is subadditive on $\left( 0,\infty \right) $, and so $%
\Phi _{m}$ is submultiplicative on $\left( 1,\infty \right) $.

We will show that $\Phi $ is increasing and convex in $\left[ E,\infty
\right) $. For any $t>1$ we have%
\begin{eqnarray}
\Phi ^{\prime }\left( t\right) &=&\Phi \left( t\right) m\left( \left( \ln
t\right) ^{\frac{1}{m}}+1\right) ^{m-1}\frac{1}{m}\left( \ln t\right) ^{%
\frac{1}{m}-1}\frac{1}{t}  \notag \\
&=&\frac{\Phi \left( t\right) }{t}\left( 1+\frac{1}{\left( \ln t\right) ^{%
\frac{1}{m}}}\right) ^{m-1}:=\frac{\Phi \left( t\right) }{t}\Omega \left(
t\right) ,  \label{Phi1}
\end{eqnarray}%
with $\Omega \left( t\right) =\Omega _{m}\left( t\right) =\left( 1+\left(
\ln t\right) ^{-\frac{1}{m}}\right) ^{m-1}>1$; and so for any $E>1$ we have 
\begin{equation}
\Phi ^{\prime }\left( E\right) >\frac{\Phi \left( E\right) }{E}.
\label{PhiE}
\end{equation}%
Next, we compute%
\begin{eqnarray*}
\Phi ^{\prime \prime }\left( t\right) &=&\frac{\Phi \left( t\right) }{t^{2}}%
\left( \left( \Omega \left( t\right) \right) ^{2}-\Omega \left( t\right)
+t\Omega ^{\prime }\left( t\right) \right) \\
&=&\frac{\Phi _{m}\left( t\right) }{t^{2}}\left( \left( \Omega \left(
t\right) \right) ^{2}-\Omega \left( t\right) +t\Omega ^{\prime }\left(
t\right) \right) .
\end{eqnarray*}%
Since $\Omega ^{\prime }\left( t\right) =-\frac{m-1}{m}\frac{1}{t}\left(
1+\left( \ln t\right) ^{-\frac{1}{m}}\right) ^{m-2}\left( \ln t\right) ^{-%
\frac{1}{m}-1}=-\frac{m-1}{m}\frac{1}{t}\Omega ^{\frac{m-2}{m-1}}\left( \ln
t\right) ^{-1-\frac{1}{m}}$, for $t>1$ we have%
\begin{eqnarray}
\Phi ^{\prime \prime }\left( t\right) &=&\frac{\Phi \left( t\right) }{t^{2}}%
\left( \left( \Omega \left( t\right) \right) ^{2}-\Omega \left( t\right) -%
\frac{m-1}{m}\Omega ^{\frac{m-2}{m-1}}\left( \ln t\right) ^{-\frac{1}{m}%
-1}\right)  \notag \\
&=&\frac{\Phi _{m}\left( t\right) }{t^{2}}\Omega \left( t\right) \left(
\Omega \left( t\right) -1-\frac{\frac{m-1}{m}}{\Omega ^{\frac{1}{m-1}}\left(
\ln t\right) ^{1+\frac{1}{m}}}\right)  \notag \\
&=&\frac{\Phi _{m}\left( t\right) }{t^{2}}\Omega \left( t\right) \Gamma
\left( t\right) ,  \label{Phi2}
\end{eqnarray}%
where%
\begin{equation*}
\Gamma \left( t\right) =\Gamma _{m}\left( t\right) =\Omega \left( t\right)
-1-\frac{\frac{m-1}{m}}{\Omega ^{\frac{1}{m-1}}\left( \ln t\right) ^{1+\frac{%
1}{m}}}.
\end{equation*}%
since $\Omega \left( t\right) -1=\left( 1+\left( \ln t\right) ^{-\frac{1}{m}%
}\right) ^{m-1}-1\geq \left( m-1\right) \left( \ln t\right) ^{-\frac{1}{m}}$%
, it follows that%
\begin{equation*}
\Gamma \left( t\right) \geq \frac{m-1}{\left( \ln t\right) ^{\frac{1}{m}}}%
\left( 1-\frac{\frac{1}{m}}{\Omega ^{\frac{1}{m-1}}\ln t}\right) >\frac{%
C_{m,E}}{\left( \ln t\right) ^{\frac{1}{m}}}>0
\end{equation*}%
for all $t\geq e$ and $m>1$. This shows that $\Phi $ is convex on $\left[
e,\infty \right) $, and so by (\ref{PhiE}) and Conclusion \ref{sub
extensions} we can extend $\Phi $ to a positive increasing \emph{%
submultiplicative} convex function on $\left[ 0,\infty \right) $. However,
due to technical calculations below, it is convenient to take $%
E=E_{m}=e^{2^{m}}$, $F=F_{m}=e^{3^{m}}$, and so we will work from now on
with the definition%
\begin{equation}
\Phi \left( t\right) =\Phi _{m}\left( t\right) \equiv \left\{ 
\begin{array}{ccc}
e^{\left( \left( \ln t\right) ^{\frac{1}{m}}+1\right) ^{m}} & \text{ if } & 
t\geq E=e^{2^{m}} \\ 
\frac{F}{E}t & \text{ if } & 0\leq t\leq E=e^{2^{m}}%
\end{array}%
\right. ,  \label{def Phi m ext}
\end{equation}%
where $m>1$ will be explicitly mentioned or understood from the context. 
\begin{figure}
\includegraphics[
  width=8cm,
  keepaspectratio,
]{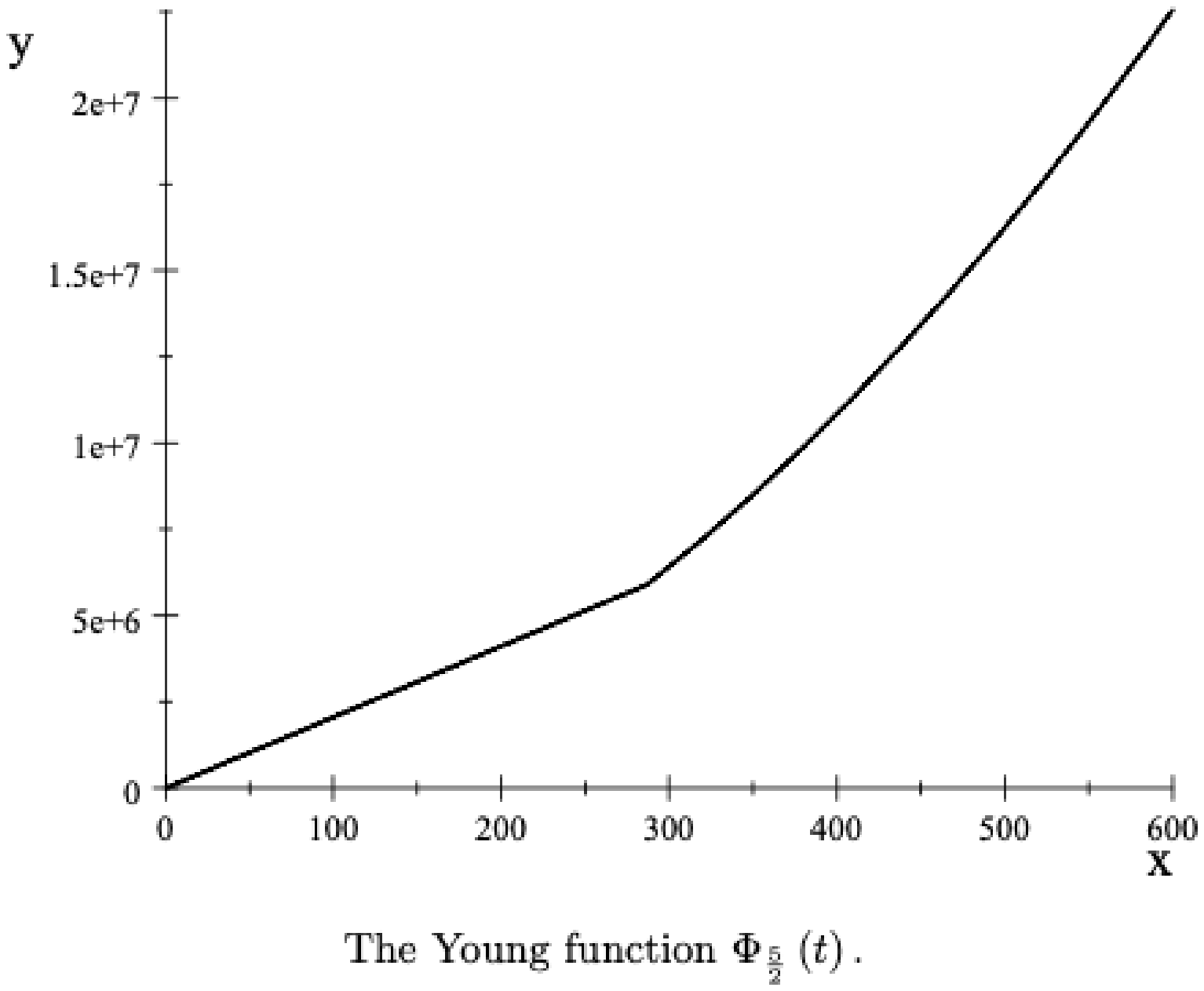}
\end{figure}
The function $\Phi _{m}$ is
clearly continuous and piecewise $C^{\infty }$. In some of our applications
we will require that the Young function should be $C^{1}$ and piece-wise
smooth, so the second derivative only has at most jump discontinuities. For
this reason we define a variation $\tilde{\Phi}_{m}$ of the Young function $%
\Phi _{m}$ which has the same growth as $t\rightarrow \infty $, and has the
required smoothness. We define%
\begin{equation}
\tilde{\Phi}_{m}\left( t\right) \equiv \left\{ 
\begin{array}{lll}
\Phi _{m}\left( t\right) & \text{ if } & t\geq E \\ 
\varrho _{m}\left( t\right) & \text{ if } & \frac{2E^{2}}{F}\leq t\leq E \\ 
\frac{1}{2}\frac{F}{E}t & \text{ if } & 0\leq t\leq \frac{2E^{2}}{F}%
\end{array}%
\right.  \label{Phi-tilde}
\end{equation}%
where $\varrho _{m}\left( t\right) $ is an increasing convex function
satisfying%
\begin{equation*}
\begin{array}{ll}
\varrho _{m}\left( \frac{2E^{2}}{F}\right) =E,\qquad & \varrho _{m}\left(
E\right) =F \\ 
\varrho _{m}^{\prime }\left( \frac{2E^{2}}{F}\right) =\frac{1}{2}\frac{F}{E}%
\qquad & \varrho _{m}^{\prime }\left( E\right) =\frac{F}{E}\left( \frac{3}{2}%
\right) ^{m-1}%
\end{array}%
\end{equation*}%
For example, we can take%
\begin{equation*}
\varrho _{m}\left( t\right) =E+\frac{1}{2}\frac{F}{E}\left( t-a\right)
+\left\{ 
\begin{array}{ll}
0 & a\leq t\leq \alpha \\ 
\frac{A}{E-\alpha }\frac{\left( t-\alpha \right) ^{2}}{2} & \alpha <t\leq E%
\end{array}%
\right.
\end{equation*}%
where $a=\frac{2E^{2}}{F}$, $A=\frac{F}{E}\left( \left( \frac{3}{2}\right)
^{m-1}-\frac{1}{2}\right) $, and $\alpha =2\frac{A}{E}-E$. 
\begin{figure}
\includegraphics[
  width=8cm,
  keepaspectratio,
]{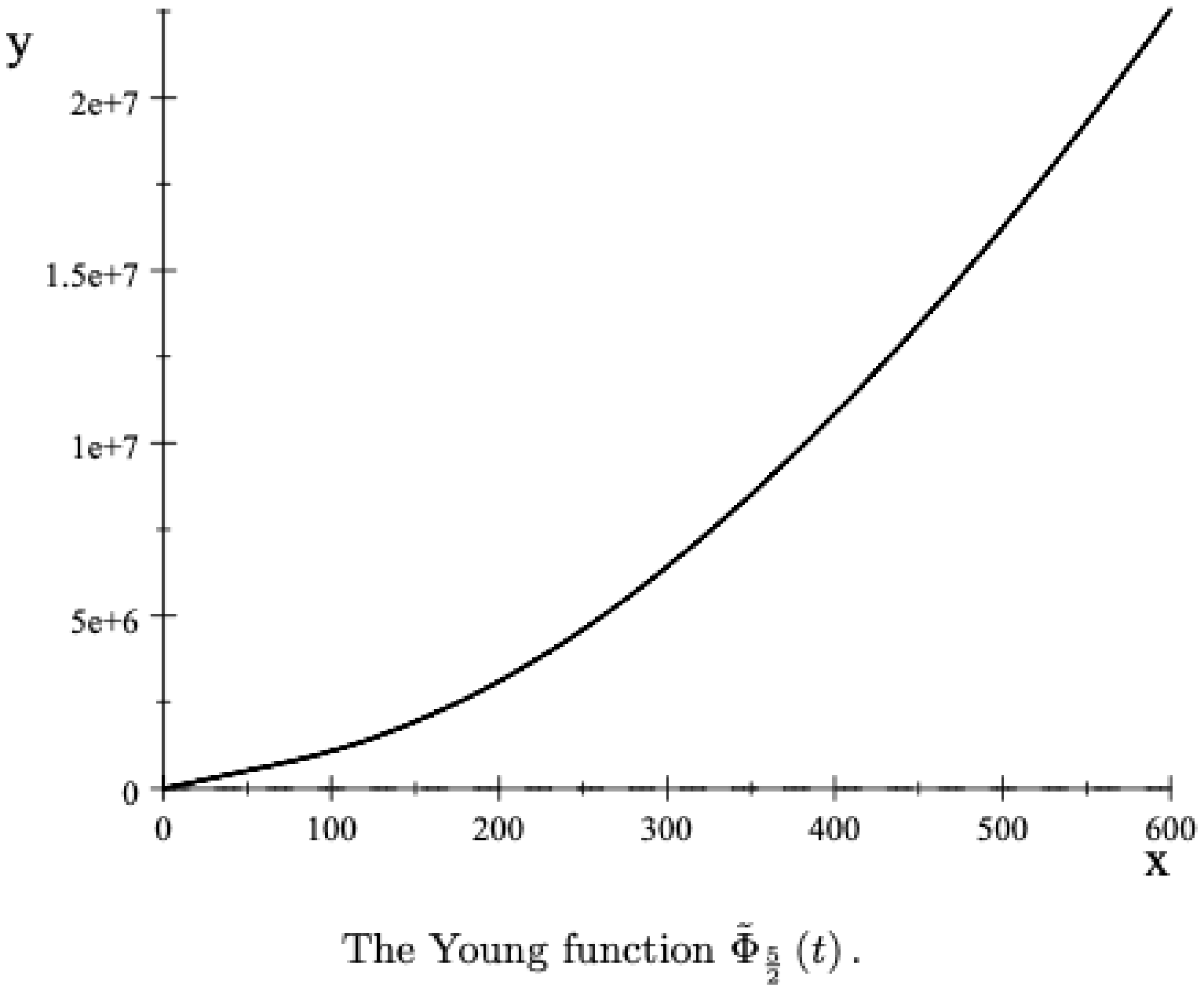}
\end{figure}
Notice that $\tilde{\Phi}_{m}\left( t\right) \equiv \Phi _{m}\left( t\right) 
$ for $t\geq E$, while $\frac{1}{C_{m}}\Phi _{m}\left( t\right) \leq \tilde{%
\Phi}_{m}\left( t\right) \leq \Phi _{m}\left( t\right) $ for all $t\geq 0$.
It follows that if we have an Orlicz Sobolev inequality for $\Phi _{m}$ with
superradius ${\varphi }$, we have that the Orlicz Sobolev inequality holds
for $\tilde{\Phi}_{m}$ with superradius $C_{m}{\varphi }$ for some constant $%
C_{m}$. Indeed, if $v\in \mathrm{Lip}_{\mathrm{c}}\left( B\left( x,r\right)
\right) $ for a ball $B\left( x,r\right) $, then 
\begin{eqnarray}
\tilde{\Phi}^{\left( -1\right) }\left( \int_{B\left( x,r\right) }\tilde{\Phi}%
\left( \left\vert v\right\vert \right) \frac{dx}{\left\vert B\left(
x,r\right) \right\vert }\right) &\leq &C_{m}\Phi ^{\left( -1\right) }\left(
\int_{B\left( x,r\right) }\Phi \left( \left\vert v\right\vert \right) \frac{%
dx}{\left\vert B\left( x,r\right) \right\vert }\right)  \notag \\
&\leq &C_{m}{\varphi }\left( x,r\right) \int_{B\left( x,r\right) }\left\vert
\nabla _{A}v\right\vert \frac{dx}{\left\vert B\left( x,r\right) \right\vert }%
.  \label{same-OS}
\end{eqnarray}%
Moreover, $\tilde{\Phi}_{m}$ is defined to be linear on $\left[ 0,a\right] $
with $\tilde{\Phi}\left( a\right) =E$ to facilitate computing successive
compositions $\tilde{\Phi}_{m}^{\left( n\right) }\left( t\right) $; indeed,
for $t$ small these compositions are just linear, for $t\geq E$ these are $%
\tilde{\Phi}_{m}^{\left( n\right) }\left( t\right) =\Phi _{m}^{\left(
n\right) }\left( t\right) $, and when $a\leq t\leq E$ then $\Phi \left(
t\right) \geq E$, so the modified formula in the \emph{middle} appears at
most once in any composition. See Corollary \ref{cor hinc} for details.

\subsection{Iterates of increasing functions}

In this subsection we consider the specific families of test functions $h$
that arise in our proofs. To implement the Moser iteration scheme we are
interested in estimates for the iterates $h_{j}\left( t\right) =h\circ
h\circ \cdots \circ h$ ($j$ times), in particular, to apply the previous
Caccioppoli inequalities, we want to estimate the quotients $\frac{%
th_{j}^{\prime }\left( t\right) }{h_{j}\left( t\right) }$ and $\frac{%
th_{j}^{\prime \prime }\left( t\right) }{h_{j}^{\prime }\left( t\right) }$,
as well as the function $\Upsilon _{j}\left( t\right) =\left( \frac{1}{2}%
h_{j}^{2}\left( t\right) \right) ^{\prime \prime }=h_{j}\left( t\right)
h_{j}^{\prime \prime }\left( t\right) +\left( h_{j}^{\prime }\left( t\right)
\right) ^{2}$.

One family of test functions we consider is 
\begin{equation}
h_{j}\left( t\right) =\Gamma _{m}^{(j)}\left( t\right) \equiv \Gamma
_{m}\circ \Gamma _{m}\circ \ldots \Gamma _{m}(t)\quad \text{(}j\text{ times),%
}  \label{hinc}
\end{equation}%
where the function $\Gamma _{m}(t)\equiv \sqrt{\Phi _{m}(t^{2})}$ for $m>1$.
When $t>\sqrt{E_{m}}=e^{2^{m-1}}$, we have the explicit formula 
\begin{equation*}
\Gamma _{m}\left( t\right) \equiv \sqrt{\Phi _{m}\left( t^{2}\right) }=e^{%
\frac{1}{2}\left( \left( 2\ln t\right) ^{\frac{1}{m}}+1\right) ^{m}}>t.
\end{equation*}

\begin{proposition}
\label{prop hinc}Let $m>1$, the function $h\left( t\right) =h_{j}\left(
t\right) =\sqrt{\Phi _{m}^{\left( j\right) }\left( t^{2}\right) }$ defined
in (\ref{hinc}) for each $j\geq 1$ satisfies%
\begin{equation*}
h^{\prime }(t)^{2}\leq \Upsilon \left( t\right) \leq 2h^{\prime
}(t)^{2}\qquad \text{and\qquad }1\leq \frac{th^{\prime }\left( t\right) }{%
h\left( t\right) }\leq C_{m}j^{m-1},
\end{equation*}%
where $\Upsilon \left( t\right) =\left( \frac{1}{2}h^{2}\left( t\right)
\right) ^{\prime \prime }=h\left( t\right) h^{\prime \prime }\left( t\right)
+\left( h^{\prime }\left( t\right) \right) ^{2}$. Moreover, we have that%
\begin{equation}
h_{j}^{\prime }\left( t\right) =\frac{h_{j}\left( t\right) }{t}\tilde{\Omega}%
_{j}^{\ast }\left( t\right) \qquad \text{with }h_{j}\left( t\right) \tilde{%
\Omega}_{j}^{\ast }\left( t\right) \text{ increasing,}  \label{hin}
\end{equation}%
and $h^{\prime \prime }\left( t\right) \geq 0$ for all $t>0$.
\end{proposition}

\begin{proof}
From the definition (\ref{def Phi m ext}) of $\Phi _{m}$, and induction we
have $h(t)=h_{j}\left( t\right) =e^{\frac{1}{2}\left( \left( 2\ln t\right) ^{%
\frac{1}{m}}+j\right) ^{m}}$ for $t\geq e^{2^{m-1}}$, while, since for $%
0<t<e^{2^{m-1}}$, $h_{1}\left( t\right) =\sqrt{\Phi _{m}\left( t^{2}\right) }%
=\tau t$ with $\tau =\exp \left( \frac{1}{2}\left( 3^{m}-2^{m}\right)
\right) $. Letting $\gamma _{1}\left( t\right) =\tau t$ and $\gamma
_{2}\left( t\right) =$ $e^{\frac{1}{2}\left( \left( 2\ln t\right) ^{\frac{1}{%
m}}+1\right) ^{m}}$, we can write%
\begin{equation*}
h_{1}\left( t\right) =\left\{ 
\begin{array}{ll}
\gamma _{1}\left( t\right) & 0\leq t<e^{2^{m-1}} \\ 
\gamma _{2}\left( t\right) & e^{2^{m-1}}\leq t%
\end{array}%
\right. .
\end{equation*}%
Then, defining the intervals we have , $I_{0}=\left( 0,\tau
^{-(j-1)}e^{2^{m-1}}\right) $, $I_{k}=\left[ \tau ^{-(j-k)}e^{2^{m-1}},\tau
^{-(j-k-1)}e^{2^{m-1}}\right) $ for $k=1,\cdots ,j-1$, and $I_{j}=\left[
e^{2^{m-1}},\infty \right) $, we have the expression 
\begin{equation}
h\left( t\right) =h_{j}\left( t\right) =\left\{ 
\begin{array}{ll}
\gamma _{1}^{\left( j\right) }\left( t\right) & t\in I_{0}; \\ 
\gamma _{2}^{\left( k\right) }\left( \gamma _{1}^{\left( j-k\right) }\left(
t\right) \right) & t\in I_{k}\;k=1,\cdots ,j-1; \\ 
\gamma _{2}^{\left( j\right) }\left( t\right) . & t\in I_{j}.%
\end{array}%
\right. .  \label{hp}
\end{equation}%
Since $\gamma _{1}^{\left( j\right) }\left( t\right) =\tau ^{j}t$ is tis
clear that for all $j\geq 1$%
\begin{equation}
\Upsilon _{1}^{\left( j\right) }\left( t\right) :=\gamma _{1}^{\left(
j\right) }\left( t\right) \left( \gamma _{1}^{\left( j\right) }\left(
t\right) \right) ^{\prime \prime }+\left( \left( \gamma _{1}^{\left(
j\right) }\left( t\right) \right) ^{\prime }\right) ^{2}=\left( \left(
\gamma _{1}^{\left( j\right) }\left( t\right) \right) ^{\prime }\right)
^{2}\qquad \text{and\qquad }\frac{\left( \gamma _{1}^{\left( j\right)
}\left( t\right) \right) ^{\prime }t}{\gamma _{1}^{\left( j\right) }\left(
t\right) }=1  \label{hp0}
\end{equation}%
Now, for $j\geq 1$, and $t\geq e^{2^{m-1}}$%
\begin{equation}
\left( \gamma _{2}^{\left( j\right) }\left( t\right) \right) ^{\prime
}=\left( e^{\frac{1}{2}\left( \left( 2\ln t\right) ^{\frac{1}{m}}+j\right)
^{m}}\right) ^{\prime }=\frac{\gamma _{2}^{\left( j\right) }\left( t\right) 
}{t}\Omega _{j}^{\ast }\left( t\right)  \label{hp5}
\end{equation}%
with $\Omega _{j}^{\ast }\left( t\right) =\left( 1+\frac{j}{\left( 2\ln
t\right) ^{\frac{1}{m}}}\right) ^{m-1}$, and%
\begin{equation}
\left( \gamma _{2}^{\left( j\right) }\left( t\right) \right) ^{\prime \prime
}=\frac{\gamma _{2}^{\left( j\right) }\left( t\right) }{t^{2}}\left( \left(
\Omega _{j}^{\ast }\left( t\right) \right) ^{2}-\Omega _{j}^{\ast }\left(
t\right) +t\left( \Omega _{j}^{\ast }\left( t\right) \right) ^{\prime
}\right) .  \label{hp6}
\end{equation}%
So from (\ref{hp5}) and (\ref{hp0}) (since $\gamma _{1}^{\left( j-k\right)
}\left( t\right) \geq e^{2^{m-1}}$ whenever $t\in I_{k}$, $k=1,\cdots ,j$)
we obtain%
\begin{equation}
1\leq \frac{t\left( \gamma _{2}^{\left( j\right) }\left( t\right) \right)
^{\prime }}{\gamma _{2}^{\left( j\right) }\left( t\right) }=\Omega
_{j}^{\ast }\left( t\right) \leq \left( 1+\frac{j}{2}\right) ^{m-1}
\label{hp1}
\end{equation}%
and 
\begin{eqnarray}
\frac{\Upsilon _{2}^{\left( j\right) }\left( t\right) }{\left( \left( \gamma
_{2}^{\left( j\right) }\left( t\right) \right) ^{\prime }\right) ^{2}} &:&=%
\frac{\gamma _{2}^{\left( j\right) }\left( t\right) \left( \gamma
_{2}^{\left( j\right) }\left( t\right) \right) ^{\prime \prime }+\left(
\left( \gamma _{2}^{\left( j\right) }\left( t\right) \right) ^{\prime
}\right) ^{2}}{\left( \left( \gamma _{2}^{\left( j\right) }\left( t\right)
\right) ^{\prime }\right) ^{2}}  \notag \\
&=&\frac{\left( \Omega _{j}^{\ast }\left( t\right) \right) ^{2}-\Omega
_{j}^{\ast }\left( t\right) +t\left( \Omega _{j}^{\ast }\left( t\right)
\right) ^{\prime }}{\left( \Omega _{j}^{\ast }\left( t\right) \right) ^{2}}%
+1.  \label{hp2}
\end{eqnarray}%
Since $\left( \Omega _{j}^{\ast }\left( t\right) \right) ^{\prime }=-\frac{%
\frac{m-1}{m}j}{2^{\frac{1}{m}}\left( \ln t\right) ^{\frac{1}{m}+1}t}\left(
1+\frac{j}{\left( 2\ln t\right) ^{\frac{1}{m}}}\right) ^{m-2}$, using the
estimate $\left( 1+x\right) ^{m-1}-1\geq \left( 1+x\right) ^{-\left(
2-m\right) _{+}}\left( m-1\right) x\geq \left( 1+x\right) ^{-1}\left(
m-1\right) x$ when $m>1$, $x\geq 0$, we have that%
\begin{eqnarray*}
\Omega _{j}^{\ast }\left( t\right) -1+t\frac{\left( \Omega _{j}^{\ast
}\left( t\right) \right) ^{\prime }}{\Omega _{j}^{\ast }\left( t\right) }
&=&\Omega _{j}^{\ast }\left( t\right) -1-\frac{\frac{m-1}{m}j}{2^{\frac{1}{m}%
}\left( \ln t\right) ^{\frac{1}{m}+1}\left( 1+\frac{j}{\left( 2\ln t\right)
^{\frac{1}{m}}}\right) } \\
&\geq &\frac{1}{\left( 1+\tfrac{j}{\left( 2\ln t\right) ^{\frac{1}{m}}}%
\right) }\left( \tfrac{\left( m-1\right) j}{\left( 2\ln t\right) ^{\frac{1}{m%
}}}-\frac{\frac{m-1}{m}j}{2^{\frac{1}{m}}\left( \ln t\right) ^{\frac{1}{m}+1}%
}\right) \\
&=&\frac{\tfrac{\left( m-1\right) j}{\left( 2\ln t\right) ^{\frac{1}{m}}}}{%
\left( 1+\tfrac{j}{\left( 2\ln t\right) ^{\frac{1}{m}}}\right) }\left( 1-%
\frac{\frac{1}{m}}{\ln t}\right) \geq 0,\qquad \text{since }t\geq
e^{2^{m-1}}.
\end{eqnarray*}%
It follows that $\left( \gamma _{2}^{\left( j\right) }\left( t\right)
\right) ^{\prime \prime }\geq 0$, and 
\begin{equation}
0\leq \frac{\gamma _{2}^{\left( j\right) }\left( t\right) \left( \gamma
_{2}^{\left( j\right) }\left( t\right) \right) ^{\prime \prime }}{\left(
\left( \gamma _{2}^{\left( j\right) }\left( t\right) \right) ^{\prime
}\right) ^{2}}=\frac{\left( \Omega _{j}^{\ast }\left( t\right) \right)
^{2}-\Omega _{j}^{\ast }\left( t\right) +t\left( \Omega _{j}^{\ast }\left(
t\right) \right) ^{\prime }}{\left( \Omega _{j}^{\ast }\left( t\right)
\right) ^{2}}\leq 1,  \label{hp4}
\end{equation}%
substituting in (\ref{hp2}) yields%
\begin{equation}
1\leq \frac{\Upsilon _{2}^{\left( j\right) }\left( t\right) }{\left( \left(
\gamma _{2}^{\left( j\right) }\left( t\right) \right) ^{\prime }\right) ^{2}}%
\leq 2.  \label{hp3}
\end{equation}%
Then from from the expression (\ref{hp}) for $h_{j}$ and (\ref{hp0})-(\ref%
{hp1})-(\ref{hp3}) it follows that $h_{j}\left( t\right) $ satisfies the
estimates claimed in the proposition both for $t\in I_{0}$ and $t\in I_{j}$.
Now, note that when $t\in I_{k}$, $k=1,\cdots ,j-1$, we have $\gamma
_{1}^{\left( j-k\right) }\left( t\right) \in I_{j}$, and also 
\begin{eqnarray*}
h_{j}^{\prime }\left( t\right) &=&\left( \gamma _{2}^{\left( k\right)
}\right) ^{\prime }\left( \gamma _{1}^{\left( j-k\right) }\left( t\right)
\right) \cdot \left( \gamma _{1}^{\left( j-k\right) }\right) ^{\prime
}\left( t\right) \\
h_{j}^{\prime \prime }\left( t\right) &=&\left( \gamma _{2}^{\left( k\right)
}\right) ^{\prime \prime }\left( \gamma _{1}^{\left( j-k\right) }\left(
t\right) \right) \cdot \left( \left( \gamma _{1}^{\left( j-k\right) }\right)
^{\prime }\left( t\right) \right) ^{2}+\left( \gamma _{2}^{\left( k\right)
}\right) ^{\prime }\left( \gamma _{1}^{\left( j-k\right) }\left( t\right)
\right) \cdot \left( \gamma _{1}^{\left( j-k\right) }\right) ^{\prime \prime
}\left( t\right) \\
&=&\left( \gamma _{2}^{\left( k\right) }\right) ^{\prime \prime }\left(
\gamma _{1}^{\left( j-k\right) }\left( t\right) \right) \cdot \left( \left(
\gamma _{1}^{\left( j-k\right) }\right) ^{\prime }\left( t\right) \right)
^{2}
\end{eqnarray*}%
since $\left( \gamma _{1}^{\left( j-k\right) }\right) ^{\prime \prime
}\equiv 0$, then, in these intervals, by (\ref{hp0}) we have 
\begin{eqnarray*}
\frac{th_{j}^{\prime }\left( t\right) }{h_{j}\left( t\right) } &=&\frac{%
t~\left( \gamma _{2}^{\left( k\right) }\right) ^{\prime }\left( \gamma
_{1}^{\left( j-k\right) }\left( t\right) \right) \cdot \left( \gamma
_{1}^{\left( j-k\right) }\right) ^{\prime }\left( t\right) }{\gamma
_{2}^{\left( k\right) }\left( \gamma _{1}^{\left( j-k\right) }\left(
t\right) \right) }=\frac{\gamma _{1}^{\left( j-k\right) }\left( t\right)
\left( \gamma _{2}^{\left( k\right) }\right) ^{\prime }\left( \gamma
_{1}^{\left( j-k\right) }\left( t\right) \right) }{\gamma _{2}^{\left(
k\right) }\left( \gamma _{1}^{\left( j-k\right) }\left( t\right) \right) }%
\frac{t~\left( \gamma _{1}^{\left( j-k\right) }\right) ^{\prime }\left(
t\right) }{\gamma _{1}^{\left( j-k\right) }\left( t\right) } \\
&=&\frac{\gamma _{1}^{\left( j-k\right) }\left( t\right) \left( \gamma
_{2}^{\left( k\right) }\right) ^{\prime }\left( \gamma _{1}^{\left(
j-k\right) }\left( t\right) \right) }{\gamma _{2}^{\left( k\right) }\left(
\gamma _{1}^{\left( j-k\right) }\left( t\right) \right) }
\end{eqnarray*}%
so $1\leq \frac{th_{j}^{\prime }\left( t\right) }{h_{j}\left( t\right) }\leq
\left( 1+\frac{j}{2}\right) ^{m-1}\leq C_{m}j^{m-1}$ by (\ref{hp1}) for $%
t\in I_{k}$, $k=0,\cdots ,j$; this finishes the proof of the second set of
inequalities in the lemma. Also, for $t\in I_{k}$, $k=1,\cdots ,j-1$%
\begin{eqnarray*}
0 &\leq &\frac{h_{j}\left( t\right) h_{j}^{\prime \prime }\left( t\right) }{%
\left( h_{j}^{\prime }\left( t\right) \right) ^{2}}=\frac{\gamma
_{2}^{\left( k\right) }\left( \gamma _{1}^{\left( j-k\right) }\right) \cdot
\left( \gamma _{2}^{\left( k\right) }\right) ^{\prime \prime }\left( \gamma
_{1}^{\left( j-k\right) }\left( t\right) \right) \cdot \left( \left( \gamma
_{1}^{\left( j-k\right) }\right) ^{\prime }\left( t\right) \right) ^{2}}{%
\left( \left( \gamma _{2}^{\left( k\right) }\right) ^{\prime }\left( \gamma
_{1}^{\left( j-k\right) }\left( t\right) \right) \cdot \left( \gamma
_{1}^{\left( j-k\right) }\right) ^{\prime }\left( t\right) \right) ^{2}} \\
&=&\frac{\gamma _{2}^{\left( k\right) }\left( \gamma _{1}^{\left( j-k\right)
}\right) \cdot \left( \gamma _{2}^{\left( k\right) }\right) ^{\prime \prime
}\left( \gamma _{1}^{\left( j-k\right) }\right) }{\left( \left( \gamma
_{2}^{\left( k\right) }\right) ^{\prime }\left( \gamma _{1}^{\left(
j-k\right) }\right) \right) ^{2}}\leq 1
\end{eqnarray*}%
by (\ref{hp4}). Hence we also have $1\leq \frac{\Upsilon _{j}\left( t\right) 
}{\left( h_{j}^{\prime }\left( t\right) \right) ^{2}}\leq 2$ for $t\in I_{k}$%
, what finishes the proof of the first pair of inequalities.
\end{proof}

\begin{remark}
\label{remark-h0}Note that the identity, $h\left( t\right) =t$, trivially
satisfies the conclusions in the previous proposition.
\end{remark}

The following is a corollary of the proof of Proposition \ref{prop hinc},
which extends its conclusions with the Young function $\Phi _{m}$ replaced
by $\tilde{\Phi}_{m}$.

\begin{corollary}
\label{cor hinc}For any $m>1$, then for any integer $j\geq 1$ the function $%
\tilde{h}\left( t\right) =\tilde{h}_{j}\left( t\right) =\sqrt{\tilde{\Phi}%
_{m}^{\left( j\right) }\left( t^{2}\right) }$ with $\tilde{\Phi}_{m}\left(
t\right) $ defined in (\ref{Phi-tilde}) satisfies%
\begin{equation*}
\tilde{h}^{\prime }(t)^{2}\leq \tilde{\Upsilon}\left( t\right) \leq 3\tilde{h%
}^{\prime }(t)^{2}\qquad \text{and\qquad }1\leq \frac{t\tilde{h}^{\prime
}\left( t\right) }{\tilde{h}\left( t\right) }\leq C_{m}j^{m-1},
\end{equation*}%
where $\tilde{\Upsilon}\left( t\right) =\left( \frac{1}{2}\tilde{h}%
^{2}\left( t\right) \right) ^{\prime \prime }=\tilde{h}\left( t\right) 
\tilde{h}^{\prime \prime }\left( t\right) +\left( \tilde{h}^{\prime }\left(
t\right) \right) ^{2}$. Moreover, we have that $\tilde{h}^{\prime \prime
}\left( t\right) \geq 0$ for all $t>0$.
\end{corollary}

\begin{proof}
The proof is the same as for Proposition \ref{prop hinc}, with the
appropriate modification of the explicit formula for the compositions.
Indeed, for $t\in \left[ 0,\left( \frac{2E^{2}}{F}\right) ^{\frac{1}{2}}%
\right] :=\left[ 0,\tilde{a}\right] $ we have that $\tilde{h}_{1}\left(
t\right) =\sqrt{\tilde{\Phi}_{m}\left( t^{2}\right) }=$ $\tilde{\tau}t$,
with $\tilde{\tau}=\sqrt{\frac{1}{2}\frac{F}{E}}$. Using the definition (\ref%
{Phi-tilde}) of $\tilde{\Phi}_{m}$, we write%
\begin{eqnarray*}
\tilde{h}_{1}\left( t\right) &=&\sqrt{\tilde{\Phi}_{m}\left( t^{2}\right) }%
\equiv \left\{ 
\begin{array}{lll}
\sqrt{\Phi _{m}\left( t^{2}\right) } & \text{ if } & t\geq E^{\frac{1}{2}}
\\ 
\sqrt{\varrho _{m}\left( t^{2}\right) } & \text{ if } & \left( \frac{2E^{2}}{%
F}\right) ^{\frac{1}{2}}\leq t\leq E^{\frac{1}{2}} \\ 
\sqrt{\frac{1}{2}\frac{F}{E}}t & \text{ if } & 0\leq t\leq \left( \frac{%
2E^{2}}{F}\right) ^{\frac{1}{2}}%
\end{array}%
\right. \\
&:&=\left\{ 
\begin{array}{lll}
\gamma _{2}\left( t\right) & \text{ if } & t\geq E^{\frac{1}{2}} \\ 
\gamma _{1}\left( t\right) & \text{ if } & \left( \frac{2E^{2}}{F}\right) ^{%
\frac{1}{2}}\leq t\leq E^{\frac{1}{2}} \\ 
\gamma _{0}\left( t\right) & \text{ if } & 0\leq t\leq \left( \frac{2E^{2}}{F%
}\right) ^{\frac{1}{2}}%
\end{array}%
\right. .
\end{eqnarray*}%
Then, defining the intervals $\tilde{I}_{0}=\left( 0,\tilde{\tau}^{-(j-1)}%
\tilde{a}\right) $, $\tilde{I}_{k}=\left[ \tilde{\tau}^{-(j-k)}\tilde{a},%
\tilde{\tau}^{-(j-k-1)}\tilde{a}\right) $ for $k=1,\cdots ,j-1$, $\tilde{I}%
_{j}=\left[ \tilde{a},E^{\frac{1}{2}}\right) $, and $\tilde{I}_{j+1}=\left[
E^{\frac{1}{2}},\infty \right) $, we have that 
\begin{equation*}
\tilde{h}_{j}\left( t\right) \equiv \left\{ 
\begin{array}{lll}
\gamma _{2}^{\left( j\right) }\left( t\right) =h_{j}\left( t\right) =\sqrt{%
\Phi _{m}^{\left( j\right) }\left( t^{2}\right) } & \text{ if } & t\in 
\tilde{I}_{j+1} \\ 
\gamma _{2}^{\left( k-1\right) }\circ \gamma _{1}\left( t\right) \circ
\gamma _{0}^{\left( j-k\right) }\left( t\right) & \text{ if } & t\in \tilde{I%
}_{k},\quad k=1,\cdots ,j \\ 
\gamma _{0}^{\left( j\right) }\left( t\right) & \text{ if } & t\in \tilde{I}%
_{0}%
\end{array}%
\right.
\end{equation*}%
The proof when $t\in \tilde{I}_{0}$ or $t\in \tilde{I}_{j+1}$ is the same as
before (note that now $\gamma _{0}$ replaces $\gamma _{1}$ in the previous
proof), while if $t\in \tilde{I}_{k},$, $k=1,\cdots ,j$, 
\begin{eqnarray*}
\left( \tilde{h}_{j}\left( t\right) \right) ^{\prime } &=&\left( \gamma
_{2}^{\left( k-1\right) }\right) ^{\prime }\left( \gamma _{1}\left( t\right)
\left( \gamma _{0}^{\left( j-k\right) }\right) \right) \cdot \left( \gamma
_{1}\left( t\right) \right) ^{\prime }\left( \gamma _{0}^{\left( j-k\right)
}\right) \cdot \left( \gamma _{0}^{\left( j-k\right) }\right) ^{\prime } \\
\left( \tilde{h}_{j}\left( t\right) \right) ^{\prime \prime } &=&\left(
\gamma _{2}^{\left( k-1\right) }\right) ^{\prime \prime }\left( \gamma
_{1}\left( t\right) \left( \gamma _{0}^{\left( j-k\right) }\left( t\right)
\right) \right) \cdot \left( \left( \gamma _{1}\left( t\right) \right)
^{\prime }\left( \gamma _{0}^{\left( j-k\right) }\right) \cdot \left( \gamma
_{0}^{\left( j-k\right) }\right) ^{\prime }\right) ^{2} \\
&&+\left( \gamma _{2}^{\left( k-1\right) }\right) ^{\prime }\left( \gamma
_{1}\left( t\right) \left( \gamma _{0}^{\left( j-k\right) }\right) \right)
\cdot \left( \left( \gamma _{1}\right) ^{\prime \prime }\left( \gamma
_{0}^{\left( j-k\right) }\right) \right) \cdot \left( \left( \gamma
_{0}^{\left( j-k\right) }\right) ^{\prime }\right) ^{2}
\end{eqnarray*}%
where we used that $\gamma _{3}^{\prime \prime }\equiv 0$. Since by the
chain rule we have that for any smooth functions $a\left( t\right) ,b\left(
t\right) ,c\left( t\right) $%
\begin{eqnarray*}
\frac{a\left( b\left( c\left( t\right) \right) \right) \cdot \left( a\left(
b\left( c\left( t\right) \right) \right) \right) ^{\prime \prime }}{\left(
\left( a\left( b\left( c\left( t\right) \right) \right) \right) ^{\prime
}\right) ^{2}} &=&\frac{a\left( b\left( c\right) \right) \cdot a^{\prime
\prime }\left( b\left( c\right) \right) }{\left( a^{\prime }\left( b\left(
c\right) \right) \right) ^{2}}+\frac{a\left( b\left( c\right) \right) }{%
a^{\prime }\left( b\left( c\right) \right) \cdot b\left( c\right) }\frac{%
b\left( c\right) \cdot b^{\prime \prime }\left( c\right) }{\left( b^{\prime
}\left( c\right) ^{\prime }\right) ^{2}} \\
&&+\frac{a\left( b\left( c\right) \right) }{a^{\prime }\left( b\left(
c\right) \right) \cdot b\left( c\right) }\frac{b\left( c\right) }{b^{\prime
}\left( c\right) \cdot c}\frac{c\cdot c^{\prime \prime }}{\left( c^{\prime
}\right) ^{2}}
\end{eqnarray*}%
we have%
\begin{eqnarray*}
\frac{\tilde{h}_{j}\left( t\right) \left( \tilde{h}_{j}\left( t\right)
\right) ^{\prime \prime }}{\left( \left( \tilde{h}_{j}\left( t\right)
\right) ^{\prime }\right) ^{2}} &=&\frac{\gamma _{2}^{\left( k-1\right)
}\left( \gamma _{1}\left( t\right) \left( \gamma _{0}^{\left( j-k\right)
}\right) \right) \cdot \left( \gamma _{2}^{\left( k-1\right) }\right)
^{\prime \prime }\left( \gamma _{1}\left( t\right) \left( \gamma
_{0}^{\left( j-k\right) }\left( t\right) \right) \right) }{\left( \left(
\gamma _{2}^{\left( k-1\right) }\right) ^{\prime }\left( \gamma _{1}\left(
t\right) \left( \gamma _{0}^{\left( j-k\right) }\right) \right) \right) ^{2}}
\\
&&+\frac{\gamma _{2}^{\left( k-1\right) }\left( \gamma _{1}\left( t\right)
\left( \gamma _{0}^{\left( j-k\right) }\right) \right) }{\left( \gamma
_{2}^{\left( k-1\right) }\right) ^{\prime }\left( \gamma _{1}\left( t\right)
\left( \gamma _{0}^{\left( j-k\right) }\right) \right) \gamma _{1}\left(
t\right) \left( \gamma _{0}^{\left( j-k\right) }\right) }\frac{\gamma
_{1}\left( t\right) \left( \gamma _{0}^{\left( j-k\right) }\right) \left(
\gamma _{1}\right) ^{\prime \prime }\left( \gamma _{0}^{\left( j-k\right)
}\right) }{\left( \left( \gamma _{1}\left( t\right) \right) ^{\prime }\left(
\gamma _{0}^{\left( j-k\right) }\right) \right) ^{2}}
\end{eqnarray*}%
Then, from (\ref{hp1}) and (\ref{hp2}),%
\begin{equation*}
0\leq \frac{\tilde{h}_{j}\left( t\right) \left( \tilde{h}_{j}\left( t\right)
\right) ^{\prime \prime }}{\left( \left( \tilde{h}_{j}\left( t\right)
\right) ^{\prime }\right) ^{2}}\leq 1+\frac{\gamma _{1}\left( t\right)
\left( \gamma _{0}^{\left( j-k\right) }\right) \left( \gamma _{1}\right)
^{\prime \prime }\left( \gamma _{0}^{\left( j-k\right) }\right) }{\left(
\left( \gamma _{1}\left( t\right) \right) ^{\prime }\left( \gamma
_{0}^{\left( j-k\right) }\right) \right) ^{2}}
\end{equation*}%
From the definition of $\gamma _{1}\left( t\right) =\varrho _{m}\left(
t\right) $ it follows that $0\leq \frac{\varrho _{m}\left( t\right) \varrho
_{m}^{\prime \prime }\left( t\right) }{\left( \varrho _{m}^{\prime }\left(
t\right) \right) ^{2}}\leq 1$. Then $1\leq \frac{\tilde{\Upsilon}\left(
t\right) }{\tilde{h}^{\prime }(t)^{2}}=\frac{\tilde{h}_{j}\left( t\right)
\left( \tilde{h}_{j}\left( t\right) \right) ^{\prime \prime }}{\left( \left( 
\tilde{h}_{j}\left( t\right) \right) ^{\prime }\right) ^{2}}+1\leq 3$.
\end{proof}

We now consider $h_{\beta }\left( t\right) \equiv \sqrt{\Phi ^{\left(
j\right) }\left( t^{2\beta }\right) }\equiv \Gamma _{m}^{(j)}(t^{\beta })$.
We will show that this $h$ satisfies the hypotheses of Lemma \ref{reverse
Sobolev} for $\beta <0$ and $\beta \geq 1$.

\begin{proposition}
\label{prop hdec}The function $h_{j,\beta }\left( t\right) =h_{\beta
}=h_{j}\left( t^{\beta }\right) $, $\beta <0$ or $\beta \geq 1$, where $%
h\left( t\right) =h_{j}\left( t\right) =\sqrt{\Phi _{m}^{\left( j\right)
}\left( t^{2}\right) }$ is defined in (\ref{hinc}) for each $j\geq 1$,
satisfies $h_{\beta }^{\prime \prime }\left( t\right) \geq 0$ and 
\begin{equation*}
1\leq \frac{\Upsilon _{j,\beta }\left( t\right) }{\left( h_{j,\beta
}^{\prime }\left( t\right) \right) ^{2}}\leq 2+\frac{\left\vert \beta
-1\right\vert }{\left\vert \beta \right\vert }\qquad \text{and\qquad }%
\left\vert \beta \right\vert \leq \frac{t\left\vert h_{j,\beta }^{\prime
}\left( t\right) \right\vert }{h_{j,\beta }\left( t\right) }\leq
C_{m}\left\vert \beta \right\vert j^{m-1},
\end{equation*}%
where $\Upsilon _{\beta }\left( t\right) =\left( \frac{1}{2}h_{\beta
}^{2}\left( t\right) \right) ^{\prime \prime }=h_{\beta }\left( t\right)
h_{\beta }^{\prime \prime }\left( t\right) +\left( h_{\beta }^{\prime
}\left( t\right) \right) ^{2}$.\newline
Moreover, if $\tilde{h}_{j,\beta }\left( t\right) =\tilde{h}_{\beta }=\tilde{%
h}_{j}\left( t^{\beta }\right) $ with $\tilde{h}_{j}\left( t\right) =\sqrt{%
\tilde{\Phi}_{m}^{\left( j\right) }\left( t^{2}\right) }$ as in Corollary %
\ref{cor hinc}, then for $\tilde{\Upsilon}_{\beta }\left( t\right) =\left( 
\frac{1}{2}\tilde{h}_{\beta }^{2}\left( t\right) \right) ^{\prime \prime }=%
\tilde{h}_{\beta }\left( t\right) \tilde{h}_{\beta }^{\prime \prime }\left(
t\right) +\left( \tilde{h}_{\beta }^{\prime }\left( t\right) \right) ^{2}$ 
\begin{equation*}
1\leq \frac{\tilde{\Upsilon}_{j,\beta }\left( t\right) }{\left( \tilde{h}%
_{j,\beta }^{\prime }\left( t\right) \right) ^{2}}\leq 3+\frac{\left\vert
\beta -1\right\vert }{\left\vert \beta \right\vert }\qquad \text{and\qquad }%
\left\vert \beta \right\vert \leq \frac{t\left\vert \tilde{h}_{j,\beta
}^{\prime }\left( t\right) \right\vert }{\tilde{h}_{j,\beta }\left( t\right) 
}\leq C_{m}\left\vert \beta \right\vert j^{m-1},
\end{equation*}
\end{proposition}

\begin{proof}
Since for all $\beta \neq 0$ 
\begin{equation}
h_{\beta }^{\prime }\left( t\right) =\beta t^{\beta -1}h^{\prime }\left(
t^{\beta }\right) \qquad \text{and\qquad }h_{\beta }^{\prime \prime }\left(
t\right) =\beta \left( \beta -1\right) t^{\beta -2}h^{\prime }\left(
t^{\beta }\right) +\beta ^{2}t^{2\beta -2}h^{\prime \prime }\left( t^{\beta
}\right) .  \label{hd00}
\end{equation}%
The lower bound $h_{\beta }^{\prime \prime }\left( t\right) \geq 0$ follows
from the second identity and the facts that $h$ is an increasing convex
function, and $\beta \left( \beta -1\right) \geq 0$ when $\beta <0$ or $%
\beta \geq 1$. Now, by Proposition \ref{prop hinc} we have%
\begin{equation}
\left\vert \beta \right\vert \leq \left\vert \beta \right\vert \frac{%
t^{\beta }h^{\prime }\left( t^{\beta }\right) }{h\left( t^{\beta }\right) }=%
\frac{t\left\vert h_{\beta }^{\prime }\left( t\right) \right\vert }{h_{\beta
}\left( t\right) }\leq C_{m}\left\vert \beta \right\vert j^{m-1}.
\label{hd0}
\end{equation}%
Similarly, 
\begin{eqnarray}
\frac{\Upsilon _{\beta }\left( t\right) }{\left( h_{\beta }^{\prime }\left(
t\right) \right) ^{2}} &=&\frac{\left( \frac{1}{2}h_{\beta }(t)^{2}\right)
^{\prime \prime }}{h_{\beta }^{\prime }(t)^{2}}=\frac{h_{\beta }(t)h_{\beta
}^{\prime \prime }(t)+h_{\beta }^{\prime }(t)^{2}}{\left( h_{\beta }^{\prime
}\left( t\right) \right) ^{2}}  \notag \\
&=&\frac{h\left( t^{\beta }\right) \left( \beta \left( \beta -1\right)
t^{\beta -2}h^{\prime }\left( t^{\beta }\right) +\beta ^{2}t^{2\beta
-2}h^{\prime \prime }\left( t^{\beta }\right) \right) }{\left( \beta
t^{\beta -1}h^{\prime }\left( t^{\beta }\right) \right) ^{2}}+1  \notag \\
&=&\frac{\beta -1}{\beta }\frac{h\left( t^{\beta }\right) }{t^{\beta
}h^{\prime }\left( t^{\beta }\right) }+\frac{h\left( t^{\beta }\right)
h^{\prime \prime }\left( t^{\beta }\right) }{\left( h^{\prime }\left(
t^{\beta }\right) \right) ^{2}}+1  \notag \\
&=&\frac{\beta -1}{\beta }\frac{h\left( t^{\beta }\right) }{t^{\beta
}h^{\prime }\left( t^{\beta }\right) }+\frac{\Upsilon \left( t^{\beta
}\right) }{\left( h^{\prime }\left( t^{\beta }\right) \right) ^{2}}.
\label{hd1}
\end{eqnarray}%
By Proposition \ref{prop hinc} we have that $\frac{1}{C_{m}j^{m-1}}\leq 
\frac{h\left( t^{\beta }\right) }{t^{\beta }h^{\prime }\left( t^{\beta
}\right) }\leq 1$, and $1\leq \frac{\Upsilon \left( t^{\beta }\right) }{%
\left( h^{\prime }\left( t^{\beta }\right) \right) ^{2}}\leq 2$, so if $%
\beta <0$ or $\beta \geq 1$ we have 
\begin{equation*}
1\leq \frac{\beta -1}{\beta }\frac{1}{C_{m}j^{m-1}}+1\leq \frac{\Upsilon
_{\beta }\left( t\right) }{\left( h_{\beta }^{\prime }\left( t\right)
\right) ^{2}}\leq \frac{\beta -1}{\beta }+2,
\end{equation*}%
where we used that $\frac{\beta -1}{\beta }\geq 0$. The proof for $\tilde{h}%
_{j,\beta }$ is identical, using instead the estimates from Corollary \ref%
{cor hinc}.
\end{proof}

\subsection{The $L^{\infty }$ norm}

The following proposition establishes sufficient conditions for the iterated
integrals to converge to the supremum norm.

\begin{proposition}
\label{prop ineq3}Suppose that $\Xi $ is a nonnegative strictly increasing
function such that $\Xi \left( 0\right) =0$ and with the following property: 
\begin{equation}
\liminf_{j\rightarrow \infty }\frac{\Xi ^{(j)}\left( M\right) }{\Xi
^{(j)}\left( M_{1}\right) }=\infty \qquad \text{for all }M>M_{1}>0.
\label{LI-growth}
\end{equation}%
Let $D\Subset D_{1}$ be nonempty open bounded sets in $\mathbb{R}^{n}$, and
let $\left\{ D_{j}\right\} _{j=1}^{\infty }$ be a sequence of nested open
bounded sets satisfying 
\begin{equation*}
D_{1}\Supset D_{2}\Supset \cdots \Supset D_{j}\Supset D_{j+1}\Supset \cdots
\Supset D
\end{equation*}%
and such that $\overline{D}=\bigcap_{j=1}^{\infty }D_{j}$, Let $\omega $ be
a Borel measure in $D_{1}$, with $\omega \left( D_{1}\right) <\infty $, such
that $\mathfrak{m\ll \omega }$ where $\mathfrak{m}$ denotes Lebesgue's
measure. Then, if $f$ is $\omega $-measurable in $D_{1}$we have%
\begin{eqnarray*}
\left\Vert f\right\Vert _{L^{\infty }\left( D\right) } &\leq
&\liminf_{j\rightarrow \infty }\Xi ^{\left( -j\right) }\left(
\int_{D_{j}}\Xi ^{\left( j\right) }\left( \left\vert f\left( x\right)
\right\vert \right) ~d\omega \right) \qquad \text{and} \\
\lim_{j\rightarrow \infty }\left\Vert f\right\Vert _{L^{\infty }\left(
D_{j}\right) } &\geq &\limsup_{j\rightarrow \infty }\Xi ^{\left( -j\right)
}\left( \int_{D_{j}}\Xi ^{\left( j\right) }\left( \left\vert f\left(
x\right) \right\vert \right) ~d\omega \right) .
\end{eqnarray*}
\end{proposition}

\begin{proof}
Since $\Xi \left( 0\right) =0$ and $\Xi $ is strictly increasing, it is
invertible, and $\Xi ^{\left( j\right) }$, $\Xi ^{\left( -j\right) }$ are
nonnegative and strictly increasing for all $j\geq 1$. From the hypothesis
on $\Xi $ we have that for all $\delta \in (0,1)$, the inequality 
\begin{equation}
\delta \Xi ^{(j)}\left( M\right) \geq \Xi ^{(j)}\left( M_{1}\right)
\label{dM}
\end{equation}%
holds for each sufficiently large $j>N\left( M,M_{1},\delta \right) $, note
that if $\delta \geq 1$ the inequality trivially holds since $\Xi $ is
increasing. It follows that for all $M>M_{1}>0$ and $\delta >0$, there
exists $N\left( M,M_{1},\delta \right) $ such that (\ref{dM}) holds for all $%
j\geq N\left( M,M_{1},\delta \right) $. We have that (\ref{dM}) implies $\Xi
\left( \left[ 0,\infty \right) \right) =\left[ 0,\infty \right) $ so $\Xi
^{-1}$ is also defined on $\left[ 0,\infty \right) $. Since $\mathfrak{m\ll }%
\omega ,$ we have that $\omega \left( D_{j}\right) >0$ and in general $%
\omega \left( U\right) >0$ for all nonempty open sets $U$.

Since $\left\Vert f\right\Vert _{L^{\infty }\left( D_{j}\right) }$ is a
decreasing sequence bounded below by $\left\Vert f\right\Vert _{L^{\infty
}\left( D\right) }$, it follows that $F=\lim_{j\rightarrow \infty
}\left\Vert f\right\Vert _{L^{\infty }\left( D_{j}\right) }$ exists, and $%
F\geq \left\Vert f\right\Vert _{L^{\infty }\left( D\right) }$. Now, for each
fixed $k\geq 1$ and $j\geq k$ we have 
\begin{eqnarray*}
\Xi ^{\left( -j\right) }\left( \int_{D_{j}}\Xi ^{\left( j\right) }\left(
\left\vert f\left( x\right) \right\vert \right) ~d\omega \right) &\leq &\Xi
^{\left( -j\right) }\left( \int_{D_{j}}\Xi ^{\left( j\right) }\left(
\left\Vert f\right\Vert _{L^{\infty }\left( D_{k}\right) }\right) ~d\omega
\right) \\
&\leq &\Xi ^{\left( -j\right) }\left( \omega \left( D_{k}\right) \Xi
^{\left( j\right) }\left( \left\Vert f\right\Vert _{L^{\infty }\left(
D_{k}\right) }\right) \right) .
\end{eqnarray*}%
For $\varepsilon >0$ and $j\geq N_{k}=\max \left\{ k,N\left( \left\Vert
f\right\Vert _{L^{\infty }\left( D_{k}\right) }+\varepsilon ,\left\Vert
f\right\Vert _{L^{\infty }\left( D_{k}\right) },\omega \left( D_{k}\right)
\right) \right\} $, we have that \newline
$\omega \left( D_{k}\right) \Xi ^{\left( j\right) }\left( \left\Vert
f\right\Vert _{L^{\infty }\left( D_{k}\right) }\right) \leq \Xi ^{(j)}\left(
\left\Vert f\right\Vert _{L^{\infty }\left( D_{k}\right) }+\varepsilon
\right) $, so%
\begin{equation*}
\Xi ^{\left( -j\right) }\left( \int_{D_{j}}\Xi ^{\left( j\right) }\left(
\left\vert f\left( x\right) \right\vert \right) ~d\omega \right) \leq \Xi
^{\left( -j\right) }\left( \Xi ^{\left( j\right) }\left( \left\Vert
f\right\Vert _{L^{\infty }\left( D_{k}\right) }+\varepsilon \right) \right)
=\left\Vert f\right\Vert _{L^{\infty }\left( D_{k}\right) }+\varepsilon
\end{equation*}%
for all $j\geq N_{k}$. Since $\varepsilon >0$ is arbitrary, this proves that%
\begin{equation*}
\limsup_{j\rightarrow \infty }\Xi ^{\left( -j\right) }\left( \int_{D_{j}}\Xi
^{\left( j\right) }\left( \left\vert f\left( x\right) \right\vert \right)
~d\omega \right) \leq F=\lim_{j\rightarrow \infty }\left\Vert f\right\Vert
_{L^{\infty }\left( D_{j}\right) }.
\end{equation*}

On the other hand, for $0<2\varepsilon <\left\Vert f\right\Vert _{L^{\infty
}\left( D\right) }$ (assume $f$ is not trivially zero in $D$), define $%
\Delta _{\varepsilon }=\left\{ x\in D:\left\vert f\left( x\right)
\right\vert \geq \left\Vert f\right\Vert _{L^{\infty }\left( D\right)
}-\varepsilon \right\} $. Then we have that $0<\omega \left( \Delta
_{\varepsilon }\right) <\infty \ $(here we used that $\omega \left(
D_{1}\right) <\infty $) and 
\begin{equation*}
\int_{D_{j}}\Xi ^{\left( j\right) }\left( \left\vert f\left( x\right)
\right\vert \right) ~d\omega \geq \omega \left( \Delta _{\varepsilon
}\right) \Xi ^{\left( j\right) }\left( \left\Vert f\right\Vert _{L^{\infty
}\left( D\right) }-\varepsilon \right) .
\end{equation*}%
Hence, from (\ref{dM}), for $j\geq N\left( \left\Vert f\right\Vert
_{L^{\infty }\left( D\right) }-\varepsilon ,\left\Vert f\right\Vert
_{L^{\infty }\left( D\right) }-2\varepsilon ,\omega \left( \Delta
_{\varepsilon }\right) \right) $ it follows that%
\begin{eqnarray*}
\Xi ^{\left( -j\right) }\left( \int_{D_{j}}\Xi ^{\left( j\right) }\left(
\left\vert f\left( x\right) \right\vert \right) ~d\omega \right) &\geq &\Xi
^{\left( -j\right) }\left( \omega \left( \Delta _{\varepsilon }\right) \Xi
^{\left( j\right) }\left( \left\Vert f\right\Vert _{L^{\infty }\left(
D\right) }-\varepsilon \right) \right) \\
&\geq &\Xi ^{\left( -j\right) }\left( \Xi ^{\left( j\right) }\left(
\left\Vert f\right\Vert _{L^{\infty }\left( D\right) }-2\varepsilon \right)
\right) =\left\Vert f\right\Vert _{L^{\infty }\left( D\right) }-2\varepsilon
.
\end{eqnarray*}%
Letting $\varepsilon \rightarrow 0^{+}$, we conclude that $\left\Vert
f\right\Vert _{L^{\infty }\left( D\right) }\leq \liminf_{j\rightarrow \infty
}\Xi ^{\left( -j\right) }\left( \int_{D_{j}}\Xi ^{\left( j\right) }\left(
\left\vert f\left( x\right) \right\vert \right) ~d\mu _{j}\right) $. This
finishes the proof.
\end{proof}

\begin{remark}
Note that in the previous result we cannot in general guarantee that 
\begin{equation*}
\left\Vert f\right\Vert _{L^{\infty }\left( D\right) }=\lim_{j\rightarrow
\infty }\Xi ^{\left( -j\right) }\left( \int_{D_{j}}\Xi ^{\left( j\right)
}\left( \left\vert f\left( x\right) \right\vert \right) ~d\omega \right)
\end{equation*}%
unless we have $\left\Vert f\right\Vert _{L^{\infty }\left( D\right)
}=\lim_{j\rightarrow \infty }\left\Vert f\right\Vert _{L^{\infty }\left(
D_{j}\right) }$. This will be the case if, for example, $f$ is continuous.
\end{remark}

\begin{remark}
Proposition \ref{prop ineq3} also holds with $d\omega $ replaced by $d\mu
_{j}=\frac{d\omega }{\omega \left( D_{j}\right) }$ in each $D_{j}$, the
proof is the same.
\end{remark}

\begin{remark}
\label{remark-LI-PP}The Young functions $\Phi =\Phi _{m}$ defined on (\ref%
{def Phi m ext}) satisfies the hypotheses of Proposition \ref{prop ineq3}.
Indeed, it is clear that both $\Phi _{m}$ and $\Upsilon _{m}$ are
nonnegative, strictly increasing, and vanish at the origin. Given any $%
M>M_{1}>0$, there exists $N_{0}$ such that $\Phi ^{\left( N_{0}\right)
}\left( M_{1}\right) \geq E$, and $\Upsilon ^{\left( N_{0}\right) }\left(
M_{1}\right) \leq \frac{1}{E}$, so for all $N\geq 1$ we have%
\begin{equation*}
\frac{\Phi ^{N+N_{0}}\left( M\right) }{\Phi ^{N+N_{0}}\left( M_{1}\right) }%
=\exp \left( \left( a+N\right) ^{m}-\left( b+N\right) ^{m}\right) ,
\end{equation*}%
where $a=\left( \ln \Phi ^{N_{0}}\left( M\right) \right) ^{\frac{1}{m}%
}>\left( \ln \Phi ^{N_{0}}\left( M_{1}\right) \right) ^{\frac{1}{m}}=b$.
Since for $m>1$, we have 
\begin{equation*}
\lim_{N\rightarrow \infty }\left[ \left( a+N\right) ^{m}-\left( b+N\right)
^{m}\right] \geq \lim_{N\rightarrow \infty }(a-b)\cdot m(b+N)^{m-1}=\infty ,
\end{equation*}%
we see that the growth condition (\ref{LI-growth}) holds for $\Phi $. Note
that in terms of the associated Orlicz quasidistance (\ref{quasi-metric}) we
have that $\left\Vert f\right\Vert _{L^{\infty }\left( D\right) }\leq
\lim_{j\rightarrow \infty }\left\Vert f\right\Vert _{\mathcal{D}^{\Phi
^{\left( j\right) }}\left( D,\mu \right) }$.
\end{remark}

In \cite{CrRo23}, Cruz-Uribe and Rodney established a general result for
Orlicz norms with Young functions $B_{pq}\left( t\right) =t^{p}\left( \log
\left( e_{0}+t\right) \right) ^{q}$, $1\leq p<\infty $, $q>0$, $e_{0}=e-1$.
They showed that if $f$ is measurable in a general measure space $\left( X,%
\mathcal{M},\mu \right) $ then $\lim_{q\rightarrow \infty }\left\Vert
f\right\Vert _{B_{pq}}=\left\Vert f\right\Vert _{\infty }$, where $%
\left\Vert f\right\Vert _{B_{pq}}$ is the Orlicz norm of $f$ in $X$. Even
though the results seem of a similar type, Proposition \ref{prop ineq3}
neither contains nor it is contained in the theorem in \cite{CrRo23}, since
the integrals $\Xi ^{\left( -j\right) }\left( \int_{D_{j}}\Xi ^{\left(
j\right) }\left( \left\vert f\right\vert \right) ~d\omega \right) $ are not
in general the Orlicz norms associated with the Young functions $\Xi
^{\left( j\right) }$, but rather the quadi-distances $\left\Vert \cdot
\right\Vert _{D^{\Xi ^{\left( j\right) }}}$ defined in Section \ref%
{section-quasi}.

\section{The Moser Method - Abstract local boundedness and maximum principle 
\label{section bound}}

In this section we prove the abstract boundedness result under the presence
of an Orlicz-Sobolev inequality (\ref{OS ineq}) and a standard sequence of
Lipschitz cutoff functions (Definition \ref{def Lip cutoff}) for the Young
functions $\Phi _{m}$ given in (\ref{def Orlicz bumps}).

\subsection{Boundedness of subsolutions and supersolutions\label{section
local bound}}

\begin{theorem}
\label{L_infinity} Suppose $\sqrt{A\left( x\right) }$ is a bounded Lipschitz
continuous $n\times n$ real-valued nonnegative definite matrix in $\mathbb{R}%
^{n}$, and let $B=B\left( x,r\right) $, $0<r\leq 1$, be a $d_{A}$-metric
ball. Let $\tilde{A}$ be a symmetric $n\times n$ matrix defined in $B$ such
that the equivalence (\ref{struc_0}) holds for some $0<\lambda \leq \Lambda
<\infty $, i.e. $0\leq \lambda \,\xi ^{\mathrm{tr}}A\left( x\right) \xi
\leq \xi ^{\mathrm{tr}}\tilde{A}\left( x\right) \xi \leq \Lambda ~\xi ^{%
\mathrm{tr}}A\left( x\right) \xi .$\newline
Let $\Phi (t)=\Phi _{m}(t)$ be as in (\ref{def Phi m ext}) with $m>2$.
Suppose that there exists a superradius ${\varphi }$ so that the $\left( {%
\Phi }_{m},A,{\varphi }\right) $-Sobolev bump inequality (\ref{OS ineq})
holds in $B$, and that an $\left( A,d\right) $-standard sequence of
Lipschitz cutoff functions, as given in Definition \ref{def Lip cutoff},
exists. \newline
Let $\nu _{0}=1-\frac{\delta _{x}\left( r\right) }{r}$, where $\delta
_{x}\left( r\right) $ is the doubling increment of $B\left( x,r\right) $,
defined by (\ref{nondoub_order}). Then for all $\nu \in \left[ \nu
_{0},1\right) $ and $\beta \in \left[ 1,\infty \right) $ there exists a
constant $C\left( \varphi ,m,\lambda ,\Lambda ,r,\nu ,\beta \right) $ such
that if $u$ is a weak \textbf{subsolution} to the equation $L_{\tilde{A}}u=-%
\mathrm{div}_{\tilde{A}}\nabla _{\tilde{A}}u=\phi _{0}-\mathrm{div}_{A}\vec{\phi}%
_{1}$ in $B\left( x,r\right) $, with $A$-admissible right hand side $\left(
\phi _{0},\vec{\phi}_{1}\right) $ as prescribed in Definition \ref{def A
admiss}, then%
\begin{equation}
\left\Vert \left( u^{+}+\phi ^{\ast }\right) ^{\beta }\right\Vert
_{L^{\infty }\left( B\left( x,\nu r\right) \right) }\leq C\left( \varphi
,m,\lambda ,\Lambda ,r,\nu ,\beta \right) ~\left\Vert \left( u^{+}+\phi
^{\ast }\right) ^{\beta }\right\Vert _{L^{2}(B\left( x,r\right) ,d\mu
_{r})}\qquad \beta \geq 1,  \label{L-inf-est}
\end{equation}%
where $\phi ^{\ast }=\left\Vert \left( \phi ,\vec{\phi}_{1}\right)
\right\Vert _{\mathcal{X}\left( B\left( x,r\right) \right) }$ and .$d\mu
_{r}=\frac{dx}{\left\vert B\left( x,r\right) \right\vert }$. In fact, we can
choose%
\begin{equation*}
C\left( \varphi ,m,\lambda ,\Lambda ,r,\nu ,\beta \right) =\exp \left(
C_{m,\lambda ,\Lambda }\left( \left( \beta -1\right) ^{m}+\left( \ln \frac{%
\varphi \left( r\right) }{\left( 1-\nu \right) r}\right) ^{m}\right) \right)
.
\end{equation*}%
Furthermore, if $u$ is a weak \textbf{supersolution} to the equation $L_{%
\tilde{A}}u=\phi _{0}-\mathrm{div}_{A}\vec{\phi}_{1}$ in $B\left( x,r\right) $%
, then (\ref{L-inf-est}) holds with $u^{+}$ replaced by $u^{-}$. In
particular, if $u$ is a \textbf{solution} to $L_{\tilde{A}}u=\phi _{0}-\mathrm{%
div}_{\tilde{A}}\left( \vec{\phi}_{1}\right) $ in $B\left( x,r\right) $,
then $u$ is locally bounded in $B\left( x,r\right) $ and (\ref{L-inf-est})
holds for $\left\vert u\right\vert $ and all $\nu \in \left[ \nu
_{0},1\right) $.
\end{theorem}

\begin{proof}
Let us start by considering the\emph{\ standard} sequence of Lipschitz
cutoff functions $\left\{ \psi _{j}\right\} _{j=1}^{\infty }$ depending on $%
r $ as given in Definition \ref{def Lip cutoff}, along with the balls $%
B_{j}=B(x,r_{j})\supset \mathrm{supp}\psi _{j}$, so that $%
r=r_{1}>r_{2}>\dots >r_{j}>r_{j+1}>\dots r_{\infty }\equiv
\lim_{j\rightarrow \infty }r_{j}=\nu r$, and $\left\Vert \nabla _{A}\psi
_{j}\right\Vert _{\infty }\leq \dfrac{Cj^{2}}{\left( 1-\nu \right) r}$ with $%
\nabla _{A}$ as in (\ref{def grad A}) and $1-\frac{\delta _{x}\left(
r\right) }{r}=\nu _{0}\leq \nu <1$.

Note that a priori we do not know whether $\left\vert u\right\vert +\phi
^{\ast }\in L^{2\beta }\left( B\right) $ when $\beta >1$, however, the proof
will proceed with the assumption that $\left\vert u\right\vert +\phi ^{\ast
}\in L^{2\beta }\left( B\right) $ for all $\beta $, and then, a posteriori,
the case $\beta =1$ implies that $u^{\pm }+\phi ^{\ast }\in L^{2\beta
}\left( B\left( x,\nu r\right) \right) $ for all $0<\nu <1$, $\phi ^{\ast
}\geq 0$, $\beta \geq 1$. Let $u$ be a subsolution or supersolution of $L_{%
\tilde{A}}u=\phi _{0}-\mathrm{div}_{A}\vec{\phi}_{1}$ in $B\left( x,r\right) $%
. Then we have that%
\begin{equation}
\text{if\quad }\tilde{a}=\frac{e^{\frac{2^{m-1}}{\beta }}}{\left\Vert u^{\pm
}+\phi ^{\ast }\right\Vert _{L^{2\beta }(d\mu _{r})}}\quad \text{then\quad }%
\tilde{u}=\tilde{a}u  \label{utilde}
\end{equation}%
is a subsolution or supersolution (respectively)\ of $L_{\tilde{A}}\tilde{u}=%
\tilde{\phi}_{0}-\mathrm{div}_{A}\left( \overrightarrow{\tilde{\phi}}%
_{1}\right) $ in $B(x,r)$ with $\tilde{\phi}_{0}=\tilde{a}\phi _{0}$, $%
\overrightarrow{\tilde{\phi}}_{1}=\tilde{a}\vec{\phi}_{1}$. Moreover, $%
\tilde{\phi}^{\ast }:=\left\Vert \left( \tilde{\phi}_{0},\overrightarrow{%
\tilde{\phi}}_{1}\right) \right\Vert _{X\left( B\right) }=\frac{\phi ^{\ast
}e^{\frac{2^{m-1}}{\beta }}}{\left\Vert u^{\pm }+\phi ^{\ast }\right\Vert
_{L^{2\beta }(d\mu _{r})}}\leq e^{\frac{2^{m-1}}{\beta }}$, and%
\begin{equation}
\left\Vert \left( \tilde{u}^{\pm }+\tilde{\phi}^{\ast }\right) ^{\beta
}\right\Vert _{L^{2}(d\mu _{r})}^{\frac{1}{\beta }}=\left\Vert \tilde{u}%
^{\pm }+\tilde{\phi}^{\ast }\right\Vert _{L^{2\beta }(d\mu
_{r})}^{2}=\left\Vert \frac{u^{\pm }+\phi ^{\ast }}{\left\Vert u^{\pm }+\phi
^{\ast }\right\Vert _{L^{2\beta }(d\mu _{r})}}\right\Vert _{L^{2\beta }(d\mu
_{r})}^{2}e^{\frac{2^{m}}{\beta }}=e^{\frac{2^{m}}{\beta }}.  \label{utb}
\end{equation}

For simplicity, in what follows we write $v=\tilde{u}^{\pm }+\tilde{\phi}%
^{\ast }$, explicitly, 
\begin{equation}
v=\left\{ 
\begin{array}{lll}
\tilde{u}^{+}+\tilde{\phi}^{\ast } &  & \text{if }L_{\tilde{A}}\tilde{u}\leq 
\tilde{\phi}_{0}-\mathrm{div}_{A}\left( \overrightarrow{\tilde{\phi}}%
_{1}\right) \\ 
\tilde{u}^{-}+\tilde{\phi}^{\ast } &  & \text{if }L_{\tilde{A}}\tilde{u}\geq 
\tilde{\phi}_{0}-\mathrm{div}_{A}\left( \overrightarrow{\tilde{\phi}}%
_{1}\right)%
\end{array}%
\right.  \label{choices}
\end{equation}%
By Proposition \ref{prop hdec} we have that $h\left( t\right) =h_{j,\beta
}\left( t\right) =\sqrt{\Phi _{m}^{(j-1)}\left( t^{2\beta }\right) },$ $%
j\geq 1$, (where $\Phi ^{\left( 0\right) }\left( t\right) =t$, see Remark %
\ref{remark-h0}) satisfies the hypotheses of Lemma \ref{reverse Sobolev}
with constant $C_{h_{j,\beta }}=C_{m}\left\vert \beta \right\vert j^{m-1}$,
namely, $\left\vert h_{j,\beta }^{\prime }\left( t\right) \right\vert \leq
C_{m}\left\vert \beta \right\vert j^{m-1}\frac{h_{j,\beta }\left( t\right) }{%
t}$. We apply Lemma \ref{reverse Sobolev} to $h\left( t\right) =h_{j,\beta
}\left( t\right) $, with $\psi =\psi _{j}$, $d\mu _{j}\equiv \frac{dx}{%
\left\vert B_{j}\right\vert }$, we obtain%
\begin{equation}
\int_{B_{j}}\psi _{j}^{2}\left\vert \nabla _{A}\left[ h_{j}(v)\right]
\right\vert ^{2}d\mu _{j}\leq C_{m,\lambda ,\Lambda }^{2}\beta
^{2}j^{2\left( m-1\right) }\int_{B_{j}}\left( h(v)\right) ^{2}\left( |\nabla
_{A}\psi |^{2}+\psi ^{2}\right) ~d\mu _{j},  \label{main-cacc}
\end{equation}%
where we used the estimates in Proposition \ref{prop hdec}, namely, $%
\left\vert h_{j,\beta }^{\prime }\left( t\right) \right\vert \leq
C_{m}\left\vert \beta \right\vert j^{m-1}\frac{h_{j,\beta }\left( t\right) }{%
t}$. It follows that 
\begin{eqnarray}
\left\Vert \nabla _{A}[\psi _{j}h(v)]\right\Vert _{L^{2}(\mu _{j})}^{2}
&\leq &2\left\Vert \psi _{j}\nabla _{A}h\left( v\right) \right\Vert
_{L^{2}\left( \mu _{j}\right) }^{2}+2\left\Vert \left\vert \nabla _{A}\psi
_{j}\right\vert h\left( v\right) \right\Vert _{L^{2}\left( \mu _{j}\right)
}^{2}  \notag \\
&\leq &C_{m,\lambda ,\Lambda }^{2}\beta ^{2}j^{2\left( m-1\right)
}\int_{B_{j}}h(v)^{2}\left( |\nabla _{A}\psi _{j}|^{2}+\psi _{j}^{2}\right)
~d\mu _{j}  \notag \\
&&+2\left\Vert \left\vert \nabla _{A}\psi _{j}\right\vert h\left( v\right)
\right\Vert _{L^{2}\left( \mu _{j}\right) }^{2}  \notag \\
&\leq &C_{m,\lambda ,\Lambda }^{2}\left( \beta +1\right) ^{2}j^{2\left(
m-1\right) }\left\Vert \nabla _{A}\psi _{j}\right\Vert _{L^{\infty
}}^{2}\left\Vert h\left( v\right) \right\Vert _{L^{2}\left( B_{j},\mu
_{j}\right) }^{2}\   \notag \\
&\leq &C_{m,\lambda ,\Lambda }^{2}\frac{\left( \beta +1\right) ^{2}}{\left(
1-\nu \right) ^{2}r^{2}}j^{2\left( m+1\right) }\left\Vert h\left( v\right)
\right\Vert _{L^{2}\left( B_{j},\mu _{j}\right) }^{2},  \label{msr00}
\end{eqnarray}%
where we used the inequalities $\left\Vert \psi _{j}\right\Vert _{\infty
}\leq 1\leq \left\Vert \nabla _{A}\psi _{j}\right\Vert _{L^{\infty }}\leq
Cj^{2}/\left( \left( 1-\nu \right) r\right) $, and the fact $r_{j}\leq r\leq
1$.

Taking $w=\psi _{j}^{2}h(v)^{2}$ in the Orlicz-Sobolev inequality, and since 
$\frac{\left\vert B\left( x,r_{j}\right) \right\vert }{\left\vert B\left(
x,r_{j+1}\right) \right\vert }\leq 2$ by the choice of the sequence of
radii, yields%
\begin{eqnarray*}
&&\Phi ^{\left( -1\right) }\left( \int_{B_{j+1}}\Phi (h(v)^{2})d\mu
_{j+1}\right) \leq \Phi ^{\left( -1\right) }\left( \int_{B_{j}}2\Phi (\psi
_{j}^{2}h(v)^{2})d\mu _{j}\right) \\
&\leq &\Phi ^{\left( -1\right) }\left( \int_{B_{j}}\Phi (2\psi
_{j}^{2}h(v)^{2})d\mu _{j}\right) \leq C\varphi \left( r_{j}\right)
\left\Vert \nabla _{A}\left( \left( \psi _{j}h\left( v\right) \right)
^{2}\right) \right\Vert _{L^{1}\left( B_{j},\mu _{j}\right) } \\
&\leq &2C\varphi \left( r_{j}\right) \left\Vert \nabla _{A}\left( \psi
_{j}h(v)\right) \right\Vert _{L^{2}\left( B_{j},\mu _{j}\right) }\left\Vert
\psi _{j}h(v)\right\Vert _{L^{2}\left( B_{j},\mu _{j}\right) } \\
&\leq &C_{m,\lambda ,\Lambda }\left( \beta +1\right) \frac{\varphi \left(
r_{j}\right) }{\left( 1-\nu \right) r}j^{m+1}\left\Vert h\left( v\right)
\right\Vert _{L^{2}\left( B_{j},\mu _{j}\right) }^{2}
\end{eqnarray*}%
where we applied (\ref{msr00}). Recalling the definition of $h(v)=\sqrt{\Phi
^{(j)}\left( t^{2\beta }\right) }$ with $\Phi =\Phi _{m}$, this is%
\begin{eqnarray*}
\int_{B_{j+1}}\Phi ^{\left( j\right) }\left( v^{2\beta }\right) ~d\mu _{j+1}
&=&\int_{B_{j+1}}\Phi (h(v)^{2})~d\mu _{j+1} \\
&\leq &\Phi \left( C_{m,\lambda ,\Lambda }\left( \beta +1\right) \frac{%
\varphi \left( r\right) }{\left( 1-\nu \right) r}j^{m+1}\int_{B_{j}}\Phi
^{\left( j-1\right) }\left( v^{2\beta }\right) d\mu _{j}\right) .
\end{eqnarray*}%
Thus, setting 
\begin{equation}
K=K_{\mathrm{standard}}(\varphi ,r)=C_{m,\lambda ,\Lambda }\left( \beta
+1\right) \frac{\varphi \left( r\right) }{\left( 1-\nu \right) r}>1,
\label{def K standard}
\end{equation}%
we have that 
\begin{equation}
\int_{B\left( x,r_{j+1}\right) }\Phi ^{\left( j\right) }\left( v^{2\beta
}\right) ~d\mu _{j+1}\leq \Phi \left( Kj^{m+1}\int_{B\left( x,r_{j}\right)
}\Phi ^{\left( j-1\right) }\left( v^{2\beta }\right) ~d\mu _{j}\right)
=b_{j+1}.  \label{iteration inequality}
\end{equation}%
Now define a sequence by 
\begin{equation}
b_{1}=\int_{B_{1}}\left\vert v\right\vert ^{2\beta }d\mu _{1},\ \ \ \ \
b_{j+1}=\Phi \left( Kj^{m+1}b_{j}\right) .  \label{def Bn}
\end{equation}%
The inequality \ref{iteration inequality} and a basic induction shows that 
\begin{equation}
\int_{B_{j}}\Phi ^{\left( j-1\right) }\left( v^{2\beta }\right) ~d\mu
_{j}\leq b_{j}.  \label{upper bound of composition}
\end{equation}

Now we apply Lemma \ref{ineq2} with $b_{1}=\int_{B\left( x,r_{1}\right)
}\left\vert v\right\vert ^{2\beta }d\mu _{r_{1}}$, $b_{j+1}=\Phi \left(
Kj^{\gamma }b_{j}\right) $, and $\gamma =m+1$, then there exists a positive
number $C^{\ast }=C^{\ast }\left( b_{1},K,m\right) $ such that the
inequality $\Phi ^{(j)}\left( C^{\ast }\right) \geq b_{j+1}$ holds for each
positive number $j$. Moreover, since from (\ref{utb}) we have that $%
b_{1}=\left\Vert v\right\Vert _{L^{2\beta }(d\mu _{r})}^{\beta
}=e^{2^{m-1}}\left( \ln b_{1}\right) ^{\frac{1}{m}}=2$, we can take 
\begin{equation*}
\exp \left( \left( C_{m}\ln K\right) ^{m}\right) \leq \exp \left(
C_{m,\lambda ,\Lambda }\left( \left( \beta -1\right) ^{m}+\left( \ln \frac{%
\varphi \left( r\right) }{\left( 1-\nu \right) r}\right) ^{m}\right) \right)
\equiv C^{\ast }.
\end{equation*}%
It follows that%
\begin{equation*}
\Phi ^{-\left( j\right) }\left( \int_{B\left( x,r_{j+1}\right) }\Phi
^{\left( j\right) }\left( v^{2\beta }\right) ~d\mu _{j+1}\right) \leq \Phi
^{-\left( j\right) }\left( b_{j+1}\right) \leq C^{\ast }.
\end{equation*}%
On the other hand, by Proposition \ref{prop ineq3} (and Remark \ref%
{remark-LI-PP}) we have that%
\begin{equation*}
\left\Vert v^{2\beta }\right\Vert _{L^{\infty }\left( B_{\infty }\right)
}\leq \liminf_{j\rightarrow \infty }\Phi ^{-\left( j+1\right) }\left(
\int_{B_{j+1}}\Phi ^{\left( j+1\right) }\left( v^{2}\right) ~d\mu
_{j+1}\right) ,
\end{equation*}%
hence%
\begin{eqnarray*}
\left\Vert v^{2\beta }\right\Vert _{L^{\infty }\left( B\left( x,\nu r\right)
\right) } &=&\left\Vert \left( \tilde{u}^{+}+\tilde{\phi}\right) ^{2\beta
}\right\Vert _{L^{\infty }\left( B\left( x,\nu r\right) \right) }=\left\Vert 
\tilde{u}^{+}+\tilde{\phi}\right\Vert _{L^{\infty }\left( B\left( x,\nu
r\right) \right) }^{2\beta } \\
&=&\left\Vert \frac{u+\phi ^{\ast }}{\left\Vert u^{+}+\phi ^{\ast
}\right\Vert _{L^{2\beta }(d\mu _{r})}}\right\Vert _{L^{\infty }\left(
B\left( x,\nu r\right) \right) }^{2\beta }e^{2^{m}}=\frac{\left\Vert \left(
u+\phi ^{\ast }\right) ^{\beta }\right\Vert _{L^{\infty }\left( B\left(
x,\nu r\right) \right) }^{2}}{\left\Vert \left( u^{+}+\phi ^{\ast }\right)
^{\beta }\right\Vert _{L^{2}(B\left( x,r\right) ,d\mu _{r})}^{2}}e^{2^{m}} \\
&\leq &\exp \left( C_{m,\lambda ,\Lambda }\left( \left( \beta -1\right)
^{m}+\left( \ln \frac{\varphi \left( r\right) }{\left( 1-\nu \right) r}%
\right) ^{m}\right) \right) :=\left( C\left( \varphi ,m,\lambda ,\Lambda
,r,\nu ,\beta \right) \right) ^{2}.
\end{eqnarray*}%
Recalling now that we wrote $u$ for $\tilde{u}$ defined in (\ref{utilde}),
and by the choice of $v$ in (\ref{choices}), this yields%
\begin{equation*}
\left\Vert \left( u^{+}+\phi ^{\ast }\right) ^{\beta }\right\Vert
_{L^{\infty }\left( B\left( x,\nu r\right) \right) }\leq C\left( \varphi
,m,\lambda ,\Lambda ,r,\nu ,\beta \right) ~\left\Vert u^{+}+\phi ^{\ast
}\right\Vert _{L^{2\beta }(B\left( x,r\right) ,d\mu _{r})}^{\beta }
\end{equation*}%
for all $\beta \geq 1$ when $L_{\tilde{A}}u\leq \phi _{0}-\mathrm{div}%
_{A}\left( \overrightarrow{\phi }_{1}\right) $, while we obtain%
\begin{equation*}
\left\Vert \left( u^{-}+\phi ^{\ast }\right) ^{\beta }\right\Vert
_{L^{\infty }\left( B\left( x,\nu r\right) \right) }\leq C\left( \varphi
,m,\lambda ,\Lambda ,r,\nu ,\beta \right) ~\left\Vert u^{-}+\phi ^{\ast
}\right\Vert _{L^{2\beta }(B\left( x,r\right) ,d\mu _{r})}^{\beta }
\end{equation*}%
for all $\beta \geq 1$ when $L_{\tilde{A}}u\geq \phi _{0}-\mathrm{div}%
_{A}\left( \overrightarrow{\phi }_{1}\right) $.
\end{proof}

In the previous theorem we obtain abstract local boundedness of weak
solutions of $Lu=\phi _{0}-\mathrm{div}_{A}\vec{\phi}_{1}$, when the right
hand side only had the first term this was obtained in \cite{KoRiSaSh19}. In
order to obtain continuity, we need $L^{\infty }$ bounds for powers of
solutions $u^{\beta }$ for $\beta $ in a neighbourhood of $\beta =0$. When $%
\beta <0$ this can be done with a slight modification of the previous
argument via the application of a different Caccioppoli estimate (Lemma \ref%
{Cacc-positive}). Note that we only consider nonnegative weak
supersolutions, as this suffices for our applications.

\begin{theorem}
\label{L_infinity-beta} Under the hypotheses of Theorem \ref{L_infinity},
for all $\nu \in \left[ \nu _{0},1\right) $ and $\beta <0$ there exists a
constant $C\left( \varphi ,m,\lambda ,\Lambda ,r,\nu \right) $ such that if $%
u$ is a \textbf{nonnegative} weak \textbf{supersolution} to the equation $L_{%
\tilde{A}}u=\phi _{0}-\mathrm{div}_{A}\vec{\phi}_{1}$ in $B\left( 0,r\right) $%
, then 
\begin{equation}
\left\Vert \left( u+\phi ^{\ast }\right) ^{\beta }\right\Vert _{L^{\infty
}\left( B\left( 0,\nu r\right) \right) }\leq C\left( \varphi ,m,\lambda
,\Lambda ,r,\nu ,\beta \right) ~\left\Vert \left( u+\phi ^{\ast }\right)
^{\beta }\right\Vert _{L^{2}(d\mu _{r})}\qquad \beta <0
\label{L-inf-beta-est}
\end{equation}%
In fact, we can choose%
\begin{equation*}
C\left( \varphi ,m,\lambda ,\Lambda ,r,\nu \right) =\exp \left( C_{m,\lambda
,\Lambda }\left( \left( \left\vert \beta \right\vert +1\right) ^{m}+\left(
\ln \frac{\varphi \left( r\right) }{\left( 1-\nu \right) r}\right)
^{m}\right) \right) .
\end{equation*}
\end{theorem}

\begin{proof}
We proceed as in the proof of Theorem \ref{L_infinity}, we consider a\emph{\
standard} sequence of Lipschitz cutoff functions $\left\{ \psi _{j}\right\}
_{j=1}^{\infty }$ depending on $r$ as given in Definition \ref{def Lip
cutoff}, along with the balls $B_{j}=B(0,r_{j})\supset \mathrm{supp}\psi
_{j}$, so that $r=r_{1}>\dots >r_{j}\searrow r_{\infty }\equiv
\lim_{j\rightarrow \infty }r_{j}=\nu r$, and $\left\Vert \nabla _{A}\psi
_{j}\right\Vert _{\infty }\leq \dfrac{Cj^{2}}{\left( 1-\nu \right) r}$ with $%
1-\frac{\delta _{0}\left( r\right) }{r}=\nu _{0}\leq \nu <1$.

Let $u$ be a nonnegative supersolution of $L_{\tilde{A}}u=-\mathrm{div}_{%
\tilde{A}}\nabla _{\tilde{A}}u=\phi _{0}-\mathrm{div}_{A}\vec{\phi}_{1}$ in $%
B\left( 0,r\right) $, then we have that $\left( u+\phi ^{\ast }\right)
^{\beta }\ $is locally bounded for all $\beta <0$, $\phi ^{\ast }>0$. If $%
\phi ^{\ast }=0$ we replace it by a small positive $\varepsilon $ and let $%
\varepsilon \rightarrow 0$ at the end of the argument. As in the previous
proof, we have that%
\begin{equation*}
\text{if\quad }\tilde{a}=\frac{e^{\frac{2^{m-1}}{\beta }}}{\left\Vert u+\phi
^{\ast }\right\Vert _{L^{2\beta }(d\mu _{r})}}\quad \text{then\quad }\tilde{u%
}=\tilde{a}u
\end{equation*}%
is a supersolution\ of $L\tilde{u}=\tilde{\phi}_{0}-\mathrm{div}_{A}\left( 
\overrightarrow{\tilde{\phi}}_{1}\right) $ in $B(0,r)$ with $\tilde{\phi}%
_{0}=\tilde{a}\phi _{0}$, $\overrightarrow{\tilde{\phi}}_{1}=\tilde{a}\vec{%
\phi}_{1}$, $\tilde{\phi}^{\ast }:=\left\Vert \left( \tilde{\phi}_{0},%
\overrightarrow{\tilde{\phi}}_{1}\right) \right\Vert _{X\left( B\right) }=%
\frac{\phi ^{\ast }e^{\frac{2^{m-1}}{\beta }}}{\left\Vert u+\phi ^{\ast
}\right\Vert _{L^{2\beta }(d\mu _{r})}}\leq e^{\frac{2^{m-1}}{\beta }}$, and%
\begin{equation*}
\left\Vert \left( \tilde{u}+\tilde{\phi}^{\ast }\right) ^{\beta }\right\Vert
_{L^{2}(d\mu _{r})}^{\frac{1}{\beta }}=\left\Vert \tilde{u}+\tilde{\phi}%
^{\ast }\right\Vert _{L^{2\beta }(d\mu _{r})}^{2}=\left\Vert \frac{u+\phi
^{\ast }}{\left\Vert u+\phi ^{\ast }\right\Vert _{L^{2\beta }(d\mu _{r})}}%
\right\Vert _{L^{2\beta }(d\mu _{r})}^{2}e^{\frac{2^{m}}{\beta }}=e^{\frac{%
2^{m}}{\beta }}.
\end{equation*}

By Proposition \ref{prop hdec} we have that $h\left( t\right) =h_{j,\beta
}\left( t\right) =\sqrt{\tilde{\Phi}_{m}^{(j-1)}\left( t^{2\beta }\right) },$
$j\geq 1$, (where $\tilde{\Phi}^{\left( 0\right) }\left( t\right) =t$, see
Remark \ref{remark-h0}) satisfies the hypotheses of Lemma \ref{Cacc-positive}%
. Explicitly, for $\Upsilon \left( t\right) =\Upsilon _{j,\beta }\left(
t\right) =h(t)h^{\prime \prime }(t)+h^{\prime }(t)^{2}>0$, we have%
\begin{equation*}
1\leq \frac{\Upsilon _{j,\beta }\left( t\right) }{\left( h_{j,\beta
}^{\prime }\left( t\right) \right) ^{2}}\leq 3+\frac{\left\vert \beta
-1\right\vert }{\left\vert \beta \right\vert }\qquad \text{and\qquad }%
\left\vert h_{j,\beta }^{\prime }\right\vert \leq C_{m}\left\vert \beta
\right\vert j^{m-1}\frac{h_{j,\beta }\left( t\right) }{t},
\end{equation*}%
and $h^{\prime }\left( t\right) <0$. Notice that here we are using the
modified Young function $\tilde{\Phi}_{m}$ (\ref{Phi-tilde}) so we may apply
Lemma \ref{Cacc-positive} with $c_{1}=1$, $C_{1}=C_{m}N^{m-1}+\frac{%
\left\vert \beta -1\right\vert }{\left\vert \beta \right\vert }$, $%
C_{2}=C_{m}\left\vert \beta \right\vert j^{m-1}$, and 
\begin{equation*}
\frac{C_{1}^{2}C_{2}^{2}}{c_{1}^{2}}=\left( 3+\frac{\left\vert \beta
-1\right\vert }{\left\vert \beta \right\vert }\right) ^{2}\left(
C_{m,\lambda ,\Lambda }\left\vert \beta \right\vert j^{m-1}\right) ^{2}\leq
C_{m,\lambda ,\Lambda }\left\vert \beta \right\vert ^{2}j^{2\left(
m-1\right) },
\end{equation*}%
to obtain for $v=\tilde{u}+\tilde{\phi}^{\ast }$, $h=h_{j,\beta }$%
\begin{equation*}
\int_{B_{j}}\psi _{j}^{2}\left\vert \nabla _{A}h\left( v\right) \right\vert
^{2}d\mu _{j}\leq C_{m,\lambda ,\Lambda }\left\vert \beta \right\vert
^{2}j^{2\left( m-1\right) }\int_{B_{j}}h\left( v\right) ^{2}\left(
\left\vert \nabla _{A}\psi _{j}\right\vert ^{2}+\psi _{j}^{2}\right) ~d\mu
_{j}.
\end{equation*}%
This is a similar estimate to (\ref{main-cacc}) in the previous proof of
Theorem \ref{L_infinity}. Recall that since then $\left( {\Phi }_{m},{%
\varphi }\right) $-Sobolev bump inequality (\ref{OS ineq}) holds in $B$,
then for some $C_{m}\geq 1$ we have that from (\ref{same-OS}) then $\left( {%
\tilde{\Phi}}_{m},C_{m}{\varphi }\right) $-Sobolev bump inequality holds in $%
B$. The proof proceeds now identically as before, to obtain (\ref%
{L-inf-beta-est}) with the given constants.
\end{proof}

\subsection{Abstract maximum principle\label{section maxp}}

We can now obtain the analogous weak form of the maximum principle.

\begin{theorem}
\label{L_infinity max}Let $\Omega $ be a bounded open subset of $\mathbb{R}%
^{n}$. Assume that $\Phi (t)=\Phi _{m}(t)$ with $m>2$ satisfies the (global)
Sobolev bump inequality (\ref{OS global}) for $\Omega $. Let $u$ be a weak
subsolution to the equation (\ref{eq-linear}), i.e. $L_{\tilde{A}}u=\phi
_{0}-\mathrm{div}_{\tilde{A}}\vec{\phi}_{1}$ in $\Omega $ with $A$-admissible
pair $\left( \phi _{0},\vec{\phi}_{1}\right) $ and $A\approx \tilde{A}$ in
the sense that the equivalences (\ref{struc_0}) for some $0<\lambda \leq
\Lambda <\infty $. Suppose that $u$ is nonpositive on the boundary $\partial
\Omega $ in the sense that $u^{+}\in W_{A,0}^{1,2}\left( \Omega \right) $,
and suppose that $\left\Vert \left( \phi _{0},\vec{\phi}_{1}\right)
\right\Vert _{\mathcal{X}\left( \Omega \right) }<\infty $. Then 
\begin{equation*}
\mathrm{esssup}_{x\in \Omega }u\left( x\right) \leq C\left( n,m,\lambda
,\Lambda ,\Phi ,\Omega \right) \left( \left\Vert u\right\Vert _{L^{2}(\Omega
)}+\left\Vert \left( \phi _{0},\vec{\phi}_{1}\right) \right\Vert _{\mathcal{X%
}\left( \Omega \right) }\right) .
\end{equation*}
\end{theorem}

\begin{proof}
An examination of all of the arguments used to prove Theorem \ref{L_infinity}
shows that the only property we need of the cutoff functions $\psi _{j}$ is
that certain Sobolev and Caccioppoli inequalities hold for the functions $%
\psi _{j}h\left( u^{+}\right) $. But under the hypothesis $u^{+}\in
W_{A,0}^{1,2}\left( \Omega \right) $, we can simply take $\psi _{j}\equiv 1$
and all of our balls $B$ to be equal to the whole set $\Omega $, since then
our weak subsolution $u^{+}$ already is such that $h\left( u^{+}\right) $
satisfies the appropriate Sobolev and Caccioppoli inequalities.

Indeed, proceeding as in the proof of Lemma \ref{reverse Sobolev}, we can
take $\psi ^{2}\omega \left( u^{+}\left( x\right) +\phi ^{\ast }\right)
=\omega \left( u^{+}\left( x\right) +\phi ^{\ast }\right) $ as a test
function, and we obtain%
\begin{equation*}
\int_{\Omega }\left\vert \nabla _{A}\left[ h\left( u^{+}+\phi ^{\ast
}\right) \right] \right\vert ^{2}dx\leq C_{\lambda ,\Lambda
}C_{h}^{2}\int_{\Omega }h\left( u^{+}+\phi ^{\ast }\right) ^{2}
\end{equation*}%
where the constant $C_{h}$ is such that $h^{\prime }\left( t\right) \leq
C_{h}\frac{h\left( t\right) }{t}$, and $\phi ^{\ast }=\left\Vert \left( \phi
_{0},\vec{\phi}_{1}\right) \right\Vert _{\mathcal{X}\left( \Omega \right) }$%
. Taking $h\left( t\right) =h_{n}\left( t\right) =\sqrt{\Phi _{m}^{\left(
n\right) }\left( t^{2}\right) }$, from Proposition \ref{prop hinc} we have
that $h_{n}^{\prime }\left( t\right) \leq C_{m}n^{m-1}\frac{h_{n}\left(
t\right) }{t}$ so $C_{h}=C_{m}n^{m-1}$ in the above inequality.

Thus we have the following pair of inequalities:

(\textbf{1}) Orlicz-Sobolev type inequality with $\Phi $ bump (\ref{OS
global}), for some constant $C_{S}=C_{S}\left( n,A,\Phi ,\Omega \right) :$ 
\begin{equation*}
\Phi ^{\left( -1\right) }\left( \int_{\Omega }\Phi (w)dx\right) \leq
C_{S}\int_{\Omega }|\nabla _{A}\left( w\right) |dx,\ \ \ \ \ w\in \mathrm{%
Lip}_{\mathrm{c}}\left( \Omega \right) .
\end{equation*}

(\textbf{2}) Caccioppoli inequality for subsolutions $u$ that are
nonpositive on $\partial \Omega $: 
\begin{equation*}
\left\Vert \nabla _{\tilde{A}}h_{n}\left( u^{+}+\phi ^{\ast }\right)
\right\Vert _{L^{2}(\Omega )}\leq C_{\lambda ,\Lambda }n^{2\left( m-1\right)
}\left\Vert h_{n}\left( u\right) \right\Vert _{L^{2}(\Omega )}\ .
\end{equation*}

Taking $w=h_{n}^{2}(u^{+}+\phi ^{\ast })=\Phi ^{\left( n\right) }\left(
\left( u^{+}+\phi ^{\ast }\right) ^{2}\right) $ and combining the two
inequalities together gives 
\begin{eqnarray*}
&&\Phi ^{\left( -1\right) }\left( \int_{\Omega }\Phi (\Phi ^{\left( n\right)
}\left( \left( u^{+}+\phi ^{\ast }\right) ^{2}\right) )dx\right) \\
&\leq &C_{S}\int_{\Omega }|\nabla _{A}\left( h_{n}(u^{+}+\phi ^{\ast
})^{2}\right) |dx=2C_{S}\left\{ \int_{\Omega }\left\vert h(u^{+}+\phi ^{\ast
})\right\vert \ \left\vert \nabla _{A}h(u^{+}+\phi ^{\ast })\right\vert
dx\right\} \\
&\leq &2C_{S}\left\Vert h_{n}\left( u^{+}+\phi ^{\ast }\right) \right\Vert
_{L^{2}(\Omega )}\left\Vert \nabla _{A}h_{n}\left( u^{+}+\phi ^{\ast
}\right) \right\Vert _{L^{2}(\Omega )}\leq C_{\lambda ,\Lambda
}C_{S}n^{2\left( m-1\right) }\ \left\Vert h_{n}\left( u^{+}+\phi ^{\ast
}\right) \right\Vert _{L^{2}(\Omega )}^{2}.
\end{eqnarray*}%
Recalling the definition of $h(u)=\sqrt{\Phi ^{(n)}\left( t^{2}\right) }$
with $\Phi =\Phi _{m}$ we get, 
\begin{equation}
\int_{\Omega }\Phi ^{\left( n+1\right) }\left( \left( u^{+}+\phi ^{\ast
}\right) ^{2}\right) dx\leq \Phi \left( C_{\lambda ,\Lambda }C_{S}n^{2\left(
m-1\right) }\int_{\Omega }\Phi ^{\left( n\right) }\left( u^{+}+\phi ^{\ast
}\right) dx\right) .  \label{iteration inequality max}
\end{equation}%
Now we proceed exactly as in the proof of Theorem \ref{L_infinity} above to
complete the proof.
\end{proof}

At this point we will prove Theorem \ref{th-abs-max}, the maximum principle
for subsolutions of 
\begin{equation}
Lu\equiv \nabla ^{\mathrm{tr}}\mathcal{A}(x,u(x))\nabla u=\phi _{0}-\mathrm{%
div}_{A}\vec{\phi}_{1}  \label{eq_temp'}
\end{equation}%
with admissible pair $\left( \phi _{0},\vec{\phi}_{1}\right) $. We wish to
replace the right hand side in Theorem \ref{L_infinity max} above by $%
\sup_{\partial \Omega }u+C\left\Vert \left( \phi _{0},\vec{\phi}_{1}\right)
\right\Vert _{\mathcal{X}\left( \Omega \right) }$.

\begin{proof}[Proof of Theorem \protect\ref{th-abs-max}]
First of all, replacing $u$ by $u-\mathrm{esssup}_{\partial \Omega }u$, it
suffices to consider the case $\mathrm{esssup}_{\partial \Omega }u\leq 0$
(if $\mathrm{esssup}_{\partial \Omega }u=\infty $ there is nothing to
prove). note that under this assumption we have $\mathrm{esssup}_{x\in
\Omega }u\left( x\right) =\mathrm{esssup}_{x\in \Omega }u^{+}\left(
x\right) $ and $\mathrm{esssup}_{x\in \partial \Omega }u\left( x\right) =0$%
. Taking $\tilde{A}\left( x\right) =\mathcal{A}\left( x,u\right) $ in
Theorem \ref{L_infinity max} we then have the estimate%
\begin{equation}
\mathrm{esssup}_{x\in \Omega }u^{+}\left( x\right) \leq C\left( n,m,\lambda
,\Lambda ,\Phi ,\Omega \right) \left( \left\Vert u^{+}\right\Vert
_{L^{2}(\Omega )}+\left\Vert \left( \phi _{0},\vec{\phi}_{1}\right)
\right\Vert _{\mathcal{X}\left( \Omega \right) }\right) .  \label{maxp00}
\end{equation}%
Now, since $u\leq 0$ on $\partial \Omega $, it follows that $u^{+}\in
W_{A,0}^{1,2}\left( \Omega \right) $ and therefore it can be used as a
nonnegative test function. Since $u$ is a subsolution to $L_{\tilde{A}%
}u=\phi _{0}-\mathrm{div}_{A}\vec{\phi}_{1}$ with $\left( \phi _{0},\vec{\phi}%
_{1}\right) $ admissible it follows that%
\begin{eqnarray*}
\int_{\Omega }\left\vert \nabla _{A}u^{+}\right\vert ^{2}d\mu 
&=&\int_{\Omega }\nabla _{A}u^{+}\cdot \nabla _{A}u~d\mu \leq \int \left(
u^{+}\phi _{0}+\nabla _{A}u^{+}\vec{\phi}_{1}\right) ~d\mu  \\
&\leq &\left\Vert \left( \phi _{0},\vec{\phi}_{1}\right) \right\Vert _{%
\mathcal{X}\left( \Omega \right) }\int \left\vert \nabla
_{A}u^{+}\right\vert ~d\mu .
\end{eqnarray*}%
Using H\"{o}lder's inequality we obtain%
\begin{equation}
\left\Vert \nabla _{A}u^{+}\right\Vert _{L^{2}\left( \Omega ,\mu \right)
}\leq \mu \left( \Omega \right) ^{\frac{1}{2}}\left\Vert \left( \phi _{0},%
\vec{\phi}_{1}\right) \right\Vert _{\mathcal{X}\left( \Omega ,\mu \right) }.
\label{maxp01}
\end{equation}%
Now, since $\Phi $ is convex, by Jensen's inequality we have%
\begin{equation*}
\int_{\Omega }\left( u^{+}\right) ^{2}\frac{d\mu }{\mu \left( \Omega \right) 
}=\Phi ^{-1}\Phi \left( \int_{\Omega }\left( u^{+}\right) ^{2}\frac{d\mu }{%
\mu \left( \Omega \right) }\right) \leq \Phi ^{-1}\left( \int_{\Omega }\Phi
\left( \left( u^{+}\right) ^{2}\right) \frac{d\mu }{\mu \left( \Omega
\right) }\right) .
\end{equation*}%
Since $\Phi $ is submultiplicative then $\Phi ^{-1}$ is supermultiplicative,
so%
\begin{eqnarray*}
\Phi ^{-1}\left( \int_{\Omega }\Phi \left( \left( u^{+}\right) ^{2}\right) 
\frac{d\mu }{\mu \left( \Omega \right) }\right)  &=&\frac{\Phi ^{-1}\left(
\mu \left( \Omega \right) \right) }{\Phi ^{-1}\left( \mu \left( \Omega
\right) \right) }\Phi ^{-1}\left( \int_{\Omega }\Phi \left( \left(
u^{+}\right) ^{2}\right) \frac{d\mu }{\mu \left( \Omega \right) }\right)  \\
&\leq &\frac{1}{\Phi ^{-1}\left( \mu \left( \Omega \right) \right) }\Phi
^{-1}\left( \int_{\Omega }\Phi \left( \left( u^{+}\right) ^{2}\right) d\mu
\right) .
\end{eqnarray*}%
Putting these inequalities together and applying the global Orlicz-Sobolev
inequality, it follows that%
\begin{eqnarray*}
\frac{\Phi ^{-1}\left( \mu \left( \Omega \right) \right) }{\mu \left( \Omega
\right) }\left\Vert u^{+}\right\Vert _{L^{2}\left( \Omega \right) }^{2}
&\leq &\Phi ^{-1}\left( \int_{\Omega }\Phi \left( \left( u^{+}\right)
^{2}\right) d\mu \right) \leq C_{S}\left\Vert \nabla _{A}\left( u^{+}\right)
^{2}\right\Vert _{L^{1}\left( \Omega ,\mu \right) } \\
&\leq &2C_{S}\left\Vert u^{+}\right\Vert _{L^{2}\left( \Omega ,\mu \right)
}\left\Vert \nabla _{A}u^{+}\right\Vert _{L^{2}\left( \Omega ,\mu \right) },
\end{eqnarray*}
so from (\ref{maxp01}) we then obtain%
\begin{equation*}
\left\Vert u^{+}\right\Vert _{L^{2}\left( \Omega \right) }\leq 2C_{S}\frac{%
\mu \left( \Omega \right) ^{\frac{3}{2}}}{\Phi ^{-1}\left( \mu \left( \Omega
\right) \right) }\left\Vert \left( \phi _{0},\vec{\phi}_{1}\right)
\right\Vert _{\mathcal{X}\left( \Omega ,\mu \right) }.
\end{equation*}%
Plugging this back into (\ref{maxp00}) yields the desired inequality.
\end{proof}

\subsection{Proof of Recurrence Inequalities\label{Sec iteration lemmas}}

Now we provide the proof of the recurrence estimate used in Sections \ref%
{section local bound} and \ref{section maxp} to prove boundedness of
solutions.

\begin{lemma}
\label{ineq2}Let $m>2$, $K>1$ and $\gamma >0$. Consider the sequence defined
by 
\begin{equation*}
b_{1}\geq e^{2^{m}},\ \ \ \ b_{n+1}=\Phi (Kn^{\gamma }b_{n}).
\end{equation*}%
Then there exists a positive number $C^{\ast }=C^{\ast }\left(
m,b_{1},K,\gamma \right) $ , such that the inequality $\Phi ^{(n-1)}(C^{\ast
})\geq b_{n}$ holds for each positive integer $n$. In fact, we can choose%
\begin{equation*}
C^{\ast }=\exp \left( \left( \ln b_{1}\right) ^{\frac{1}{m}}+C_{m}\left(
\gamma +\ln K\right) \right) ^{m}=\Phi ^{\left( C_{m}\left( \gamma +\ln
K\right) \right) }\left( b_{1}\right)
\end{equation*}%
where $C_{m}$ only depends on $m$. Now we prove the growth estimate which
allowed the Moser iteration to yield the boundedness theorem.
\end{lemma}

\begin{proof}
Let $m>2$, $K>1$, $\gamma >0$, and 
\begin{equation*}
b_{1}=\int_{B\left( 0,r_{1}\right) }\left\vert u\right\vert ^{2}d\mu
_{r_{1}}\geq e^{2^{m}},\ \ \ \ b_{n+1}=\Phi (Kn^{\gamma }b_{n}).
\end{equation*}%
We want to estimate $\Phi ^{\left( -j\right) }\left( b_{j+1}\right) $. Let
us define another sequence by 
\begin{equation*}
\beta _{1}=C^{\ast },\ \ \ \ \ \beta _{n+1}=\Phi (\beta _{n}),\quad n\geq 0
\end{equation*}%
Thus we are trying to find a number $C^{\ast }$ such that $\beta _{n}=\Phi
^{\left( n-1\right) }\left( \beta _{1}\right) \geq b_{n}$ holds for all $%
n\geq 0$. Next we define the two related sequences: 
\begin{equation*}
\alpha _{n}=\left( \ln \beta _{n}\right) ^{1/m},\qquad \text{and\qquad }%
\beta _{n}=\left( \ln b_{n}\right) ^{1/m}.
\end{equation*}%
The sequence $\{\alpha _{n}\}$ satisfies $\alpha _{1}=(\ln C^{\ast })^{1/m}$
and 
\begin{equation*}
\alpha _{n+1}=\left( \ln \beta _{n+1}\right) ^{1/m}=\left( \ln \Phi \left(
\beta _{n}\right) \right) ^{1/m}=(\ln \beta _{n})^{1/m}+1=\alpha _{n}+1
\end{equation*}%
for all $n\geq 1$. As for the other sequence, it is clear that $\beta
_{1}=\left( \ln b_{1}\right) ^{1/m}>2$, but the recurrence relation for $%
b_{n}$ is a bit more complicated, we have:%
\begin{eqnarray*}
\beta _{n+1} &=&\left( \ln b_{n+1}\right) ^{1/m}=\left( \ln \Phi \left(
Kn^{\gamma }b_{n}\right) \right) ^{1/m}=\left( \ln \left( Kn^{\gamma
}b_{n}\right) \right) ^{1/m}+1 \\
&=&\left( \beta _{n}^{m}+\ln \left( Kn^{\gamma }\right) \right) ^{1/m}+1.
\end{eqnarray*}%
This is clear that $\beta _{n+1}>\beta _{n}+1$ thus we have a rough lower
bound 
\begin{equation}
\beta _{n+1}\geq n+b_{1}.  \label{betalow}
\end{equation}%
Since the function $g(x)=x^{1/m}$ is concave, we have 
\begin{equation*}
\beta _{n+1}=\left( \beta _{n}^{m}+\ln \left( Kn^{\gamma }\right) \right)
^{1/m}+1=\beta _{n}\left\{ 1+\frac{\ln \left( Kn^{\gamma }\right) }{\beta
_{n}^{m}}\right\} ^{1/m}+1\leq \beta _{n}+\frac{\ln \left( Kn^{\gamma
}\right) }{m\cdot \beta _{n}^{m-1}}+1
\end{equation*}%
Thus 
\begin{equation}
\beta _{n+1}\leq b_{1}+n+\frac{1}{m}\sum_{j=1}^{n}\frac{\ln \left(
Kj^{\gamma }\right) }{\beta _{j}^{m-1}}\quad \Longrightarrow \quad \alpha
_{n}-\beta _{n}\geq \alpha _{1}-b_{1}-\frac{1}{m}\sum_{j=1}^{n}\frac{\ln
\left( Kj^{\gamma }\right) }{\beta _{j}^{m-1}}.  \label{betalpha}
\end{equation}%
Because $m>2$, by (\ref{betalow}) we have 
\begin{equation*}
\sum_{j=1}^{n}\frac{\ln \left( Kj^{\gamma }\right) }{\beta _{j}^{m-1}}%
<\sum_{j=1}^{\infty }\frac{\ln \left( Kj^{\gamma }\right) }{(\beta
_{1}+j-1)^{m-1}}\leq \frac{C_{m}}{\beta _{1}^{m-2}}\left( \gamma +\ln
K\right) \leq C_{m}\left( \gamma +\ln K\right) <\infty ,
\end{equation*}%
where we used that $\beta _{1}=\left( \ln b_{1}\right) ^{\frac{1}{m}}\geq 2$%
. Therefore, choosing $\alpha _{1}=\beta _{1}+C_{m}\left( \gamma +\ln
K\right) $, (\ref{betalpha}) guarantees $\alpha _{n}>\beta _{n}$ for all $%
n\geq 1$, and so%
\begin{equation*}
\Phi ^{\left( n-1\right) }\left( C^{\ast }\right) =\Phi ^{\left( n-1\right)
}\left( a_{1}\right) >b_{n},
\end{equation*}%
where $C^{\ast }=C^{\ast }(b_{1},K,\gamma )$ is 
\begin{eqnarray*}
C^{\ast } &=&\exp \left( \alpha _{1}^{m}\right) =\exp \left( \beta
_{1}+C_{m}\left( \gamma +\ln K\right) \right) ^{m} \\
&=&\exp \left( \left( \ln b_{1}\right) ^{\frac{1}{m}}+C_{m}\left( \gamma
+\ln K\right) \right) ^{m}=\Phi ^{\left( C_{m}\left( \gamma +\ln K\right)
\right) }\left( b_{1}\right) .
\end{eqnarray*}
\end{proof}

\begin{remark}
\label{Moser fails}Lemma \ref{ineq2} fails for $m\leq 2$ even with $\gamma
=0 $ and $K>e$. Indeed, then from the calculations above we have%
\begin{eqnarray*}
\beta _{n+1} &=&\beta _{n}\left( 1+\frac{\ln \left( Kn^{\gamma }\right) }{%
\beta _{n}^{m}}\right) ^{1/m}+1 \\
&\geq &\beta _{n}+\frac{\ln \left( Kn^{\gamma }\right) }{m\beta _{n}^{m-1}}%
+1\geq \beta _{n}+\frac{\ln K}{m\beta _{n}^{m-1}}+1
\end{eqnarray*}%
which when iterated gives%
\begin{equation*}
\beta _{n+1}\geq \beta _{1}+n+\sum_{j=1}^{n}\frac{\ln K}{m\beta _{j}^{m-1}}%
\geq \beta _{1}+n+\frac{\ln K}{2}\sum_{j=1}^{n}\frac{1}{\beta _{j}}.
\end{equation*}%
So if there is a positive constant $A$ such that $\beta _{n+1}\leq n+A$ for $%
n$ large, then we would have%
\begin{equation*}
\beta _{n+1}\geq \beta _{1}+n+\frac{\ln K}{2}c\ln n
\end{equation*}%
for some positive constant $c$, which is a contradiction to our assumption.
Thus $\beta _{n+1}\leq \alpha _{0}+n$ for all $n\geq 1$ is impossible. That
is, we have 
\begin{equation*}
\Phi ^{\left( -n\right) }\left( b_{n}\right) =e^{\left[ \left( \ln
b_{n}\right) ^{\frac{1}{m}}-n\right] ^{m}}=e^{\left[ \beta _{n}-n\right]
^{m}}\geq e^{\left[ \beta _{1}+\frac{\ln K}{2}c\ln n\right] ^{m}}\nearrow
\infty
\end{equation*}%
as $n\rightarrow \infty $, so Lemma \ref{ineq2} does not hold.
\end{remark}

\section{The geometric setting\label{section geom-setting}}

In order to obtain \emph{geometric} applications, we will take the metric $d$
in Theorem \ref{th-abs-bound} to be the Carnot-Carath\'{e}odory metric
associated with the vector field $\nabla _{A}$ for appropriate matrices $A$,
and we will show that the hypotheses of our abstract theorems hold in this
geometry. For this we need to introduce a family of infinitely degenerate
geometries that are simple enough so that we can compute the balls
explicitly, prove the required Orlicz-Sobolev bump inequality, and define an
appropriate accumulating sequence of Lipschitz cutoff functions. We will
work solely in the plane and consider linear operators of the form%
\begin{equation*}
Lu\left( x,y\right) \equiv \nabla ^{\mathrm{tr}}A\left( x,y\right) \nabla
u\left( x,y\right) ,\ \ \ \ \ \left( x,y\right) \in \Omega ,
\end{equation*}%
where $\Omega \subset \mathbb{R}^{2}$ is a planar domain, and where the $%
2\times 2$ matrix is 
\begin{equation*}
A\left( x,y\right) =\left[ 
\begin{array}{cc}
1 & 0 \\ 
0 & f\left( x\right) ^{2}%
\end{array}%
\right] ,
\end{equation*}%
where $f\left( x\right) =e^{-F\left( x\right) }$ is even and there is $R>0$
such that $F$ satisfies the following five structure conditions for some
constants $C\geq 1$ and $\varepsilon >0$:

\begin{definition}
\textbf{Structural conditions}\label{structure conditions}

\begin{enumerate}
\item $\lim_{x\rightarrow 0^{+}}F\left( x\right) =+\infty $;

\item $F^{\prime }\left( x\right) <0$ and $F^{\prime \prime }\left( x\right)
>0$ for all $x\in (0,R)$;

\item $\frac{1}{C}\left\vert F^{\prime }\left( r\right) \right\vert \leq
\left\vert F^{\prime }\left( x\right) \right\vert \leq C\left\vert F^{\prime
}\left( r\right) \right\vert $ for $\frac{1}{2}r<x<2r<R$;

\item $\frac{1}{-xF^{\prime }\left( x\right) }$ is increasing in the
interval $\left( 0,R\right) $ and satisfies $\frac{1}{-xF^{\prime }\left(
x\right) }\leq \frac{1}{\varepsilon }\,$for $x\in (0,R)$;

\item $\frac{F^{\prime \prime }\left( x\right) }{-F^{\prime }\left( x\right) 
}\approx \frac{1}{x}$ for $x\in (0,R)$.%
\setcounter{enumi}{0}\RESUME%
\end{enumerate}
\end{definition}

\begin{remark}
We make no smoothness assumption on $f$ other than the existence of the
second derivative $f^{\prime \prime }$ on the open interval $(0,R)$. Note
also that at one extreme, $f$ can be of finite type, namely $f\left(
x\right) =x^{\alpha }$ for any $\alpha >0$, and at the other extreme, $f$
can be of strongly degenerate type, namely $f\left( x\right) =e^{-\frac{1}{
x^{\alpha }}}$ for any $\alpha >0$. Assumption (1) rules out the elliptic
case $f\left( 0\right) >0$.
\end{remark}

Under the general structural conditions \ref{structure conditions} we will
find further sufficient conditions on $F$ so that the $\left( \Phi ,A,{%
\varphi }\right) $-Orlicz-Sobolev bump inequality (\ref{OS ineq}) holds for
a particular $\Phi $ in this geometry, where the superradius ${\varphi }$
will depend on $F$ (see Proposition \ref{sob}). In \cite[Section 8.2]%
{KoRiSaSh19} we showed that these geometries support both the $\left(
1,1\right) $-Poincar\'{e} and the $\left( 1,1\right) $-Sobolev inequalities.

In particular, we consider specific functions $F$ satisfying the structural
conditions \ref{structure conditions}, namely, the geometries $F_{k,\sigma }$
defined by%
\begin{equation*}
F_{k,\sigma }\left( r\right) =\left( \ln \frac{1}{r}\right) \left( \ln
^{\left( k\right) }\frac{1}{r}\right) ^{\sigma },\qquad k\in \mathbb{N}%
,\quad \sigma >0.
\end{equation*}%
Note that $f_{k,\sigma }=e^{-F_{k,\sigma }\left( r\right) }=e^{-\left( \ln 
\frac{1}{r}\right) \left( \ln ^{\left( k\right) }\frac{1}{r}\right) ^{\sigma
}}$ vanishes to infinite order at $r=0$, and that $f_{k,\sigma }$ vanishes
to a faster order than $f_{k^{\prime },\sigma ^{\prime }}$ if either $%
k<k^{\prime }$ or \ if $k=k^{\prime }$ and $\sigma >\sigma ^{\prime }$.

To see that in the geometries $F_{k,\sigma }$ there exists a standard
sequence of Lipschitz cutoff functions in $B=B\left( x,r\right) $, as given
in Definition \ref{def Lip cutoff}, we will prove the following general
lemma for the Carnot-Carath\'{e}odory metric induced by a continuous
nonnegative semidefinite quadratic form.

\begin{lemma}
\label{spec_cutoff_lemma} Let $\xi ^{t}A(x)\xi $ be a continuous nonnegative
semidefinite quadratic form. Suppose that the subunit metric $d$ associated
to $A(x)$ is topologically equivalent to the Euclidean metric $d_{E}$ in the
sense that for all $B\left( x,r\right) \subset \Omega $ there exist
Euclidean balls $B_{E}\left( x,r_{E}\left( x,r\right) \right) $ and $%
B_{E}\left( x,R_{E}\left( x,r\right) \right) $ such that%
\begin{equation}
B_{E}\left( x,r_{E}\left( x,r\right) \right) \subseteq B\left( x,r\right)
\subseteq B_{E}\left( x,R_{E}\left( x,r\right) \right) .
\label{contain-upper}
\end{equation}%
Then for each ball $B\left( x,R\right) \subset \Omega $ and $0<r<R$ there
exists a cutoff function $\phi _{r,R}\in \mathrm{Lip}(\Omega )$ satisfying 
\begin{equation}
\begin{cases}
\mathrm{supp}\left( \phi _{r,R}\right) & \subseteq B\left( x,R\right) , \\ 
\left\{ x:\phi _{r,R}(x)=1\right\} & \supseteq B\left( x,r\right) , \\ 
\left\Vert \nabla _{A}\phi _{r,R}\right\Vert _{L^{\infty }\left( B\left(
x,R\right) \right) } & \leq \dfrac{C_{n}}{R-r}.%
\end{cases}
\label{cutoff}
\end{equation}
\end{lemma}

\begin{proof}
For any $\varepsilon \geq 0$ let $A^{\varepsilon }(x,\xi )=\xi ^{t}A(x)\xi
+\varepsilon ^{2}|\xi |^{2}$. It has been shown in \cite[Lemma 65]{SaWh4}
that under the hypothesis of Lemma \ref{spec_cutoff_lemma} the subunit
metric $d^{\varepsilon }(x,y)$ associated to $A^{\varepsilon }$ satisfies 
\begin{equation*}
\left\vert \nabla _{A}d^{\varepsilon }\left( x,y\right) \right\vert \leq 
\sqrt{n},\quad x,y\in \Omega
\end{equation*}%
uniformly in $\varepsilon >0$. Moreover, $d^{\varepsilon }(\cdot ,y)\nearrow
d(\cdot ,y)$, the convergence is monotone and $d$ is continuous (in the
Euclidean distance), therefore, $d^{\varepsilon }(\cdot ,y)\rightarrow
d(\cdot ,y)$ uniformly on compact subsets of $\Omega $.

Define $g(t)$ to vanish for $t\geq R-\frac{R-r}{4}$, to equal $1$ for $t\leq
r$ and to be linear on the interval $[r,R-\frac{R-r}{4}]$. Let $\phi
_{r,R}(x)=g(d^{\varepsilon ^{\ast }}(x,y))$, with $\varepsilon ^{\ast }$ to
be chosen later. Since $d^{\varepsilon ^{\ast }}\left( x,y\right) \leq
d\left( x,y\right) $ we have 
\begin{equation*}
\phi _{r,R}(x)=1\quad \text{when}\quad d(x,y)\leq r.
\end{equation*}%
And since $\phi _{r,R}(x)=0$ when $d^{\varepsilon ^{\ast }}(x,y)\geq R-\frac{%
R-r}{4}$, by choosing $\varepsilon ^{\ast }$ small enough, we obtain that $%
\phi _{r,R}(x)=0$ when $d(x,y)\geq R$. This shows that $\mathrm{supp}\left(
\phi _{r,R}\right) \subseteq B\left( x,R\right) $ and $\left\{ x:\phi
_{r,R}\left( x\right) =1\right\} \supseteq B\left( x,r\right) $. Next, 
\begin{equation}
\left\vert \nabla _{A}\phi _{r,R}(x)\right\vert \leq \left\Vert g^{\prime
}\right\Vert _{\infty }\left\vert \nabla _{A}d^{\varepsilon ^{\ast
}}\right\vert \leq \frac{4}{3}\frac{1}{R-r}\sqrt{n}=\frac{C_{n}}{R-r}.
\label{cut_1}
\end{equation}%
This completes the proof.
\end{proof}

\begin{remark}
Note that the condition that $A(x)$ is continuous cannot be easily omitted.
In \cite{Zhong} the author constructs an example of a discontinuous solution
to a degenerate linear elliptic equation (see Theorem 1.3 and Conjecture 6).
However, the matrix $Q$ in that case is discontinuous and this requirement
seems to be essential for the construction.
\end{remark}

\subsection{Geometric Orlicz-Sobolev inequality\label{Sec Orlicz}}

In this section we use subrepresentation inequalities proved in \cite%
{KoRiSaSh19} to prove the relevant Sobolev and Poincar\'{e} inequalities.
More precisely, we will use \cite[Lemma 58]{KoRiSaSh19}, which says that for
every Lipschitz function $w$ there holds 
\begin{equation}
\left\vert w\left( x\right) -\mathbb{E}_{x,r_{1}}w\right\vert \leq
C\int_{\Gamma \left( x,r\right) }\left\vert \nabla _{A}w\left( y\right)
\right\vert \frac{\widehat{d}\left( x,y\right) }{\left\vert B\left(
x,d\left( x,y\right) \right) \right\vert }dy,  \label{subrepresentation}
\end{equation}%
where 
\begin{equation}
\widehat{d}\left( x,y\right) \equiv \min \left\{ d\left( x,y\right) ,\frac{1%
}{\left\vert F^{\prime }\left( x_{1}+d\left( x,y\right) \right) \right\vert }%
\right\} .  \label{def d hat'}
\end{equation}%
Here $\Gamma \left( x,r\right) $ is a cusp-like region defined as 
\begin{equation*}
\Gamma \left( x,r\right) =\bigcup\limits_{k=1}^{\infty }\mathrm{co}\left[
E\left( x,r_{k}\right) \cup E\left( x,r_{k+1}\right) \right] ,
\end{equation*}%
where\ the sets $E\left( x,r_{k}\right) $ are curvilinear trapezoidal sets
on which the function $f$ does not change much, and which satisfy 
\begin{equation}
\left\vert E\left( x,r_{k}\right) \right\vert \approx \left\vert E\left(
x,r_{k}\right) \bigcap B\left( x,r_{k}\right) \right\vert \approx \left\vert
B\left( x,r_{k}\right) \right\vert \text{ for all }k\geq 1.
\label{claim that}
\end{equation}%
Finally, we use the following notation for averages 
\begin{equation*}
\mathbb{E}_{x,r_{1}}w\equiv \frac{1}{\left\vert E(x,r_{1})\right\vert }\int
\int_{E(x,r_{1})}w.
\end{equation*}

In our setting of infinitely degenerate metrics in the plane, the metrics we
consider are elliptic away from the $x_{2}$ axis, and are invariant under
vertical translations. As a consequence, we need only consider Sobolev
inequalities for the metric balls $B\left( 0,r\right) $ centered at the
origin. So from now on we consider $X=\mathbb{R}^{2}$ and the metric balls $%
B\left( 0,r\right) $ associated to one of the geometries $F$ considered in 
\cite[Part 2]{KoRiSaSh19}.

First we recall that the optimal form of the degenerate Orlicz-Sobolev \emph{%
norm} inequality for balls is%
\begin{equation*}
\left\Vert w\right\Vert _{L^{\Theta }\left( \mu _{r_{0}}\right) }\leq
Cr_{0}\left\Vert \nabla _{A}w\right\Vert _{L^{\Omega }\left( \mu
_{r_{0}}\right) },
\end{equation*}%
where $d\mu _{r_{0}}\left( x\right) =\frac{dx}{\left\vert B\left(
0,r_{0}\right) \right\vert }$, the balls $B\left( 0,r_{0}\right) $ are
control balls for a metric $A$, and the Young function $\Theta $ is a `bump
up' of the Young function $\Omega $. We will instead obtain the
nonhomogeneous form of this inequality where $L^{\Omega }\left( \mu
_{r_{0}}\right) =L^{1}\left( \mu _{r_{0}}\right) $ is the usual Lebesgue
space, and the factor $r_{0}$ on the right hand side is replaced by a
suitable superradius $\varphi \left( r_{0}\right) $, namely%
\begin{equation}
\Phi ^{\left( -1\right) }\left( \int_{B\left( 0,r_{0}\right) }\Phi \left(
w\right) ~d\mu _{r_{0}}\right) \leq C\varphi \left( r_{0}\right) \
\left\Vert \nabla _{A}w\right\Vert _{L^{1}\left( \mu _{r_{0}}\right) },\ \ \
\ \ w\in \mathrm{Lip}_{\mathrm{c}}\left( X\right) ,
\label{Phi bump'}
\end{equation}%
which we refer to as the $\left( \Phi ,A,\varphi \right) $\emph{-Sobolev
Orlicz bump inequality}. In fact, consider the positive operator $T_{B\left(
0,r_{0}\right) }:L^{1}\left( \mu _{r_{0}}\right) \rightarrow L^{\Phi }\left(
\mu _{r_{0}}\right) $ defined by 
\begin{equation*}
T_{B\left( 0,r_{0}\right) }g(x)\equiv \int_{B(0,r_{0})}K_{B\left(
0,r_{0}\right) }\left( x,y\right) g(y)dy
\end{equation*}%
with kernel $K_{B\left( 0,r_{0}\right) }$ defined as 
\begin{equation}
K_{B\left( 0,r_{0}\right) }(x,y)=\frac{\widehat{d}\left( x,y\right) }{%
\left\vert B\left( x,d\left( x,y\right) \right) \right\vert }\mathbf{1}%
_{\Gamma \left( x,r_{0}\right) }\left( y\right) .  \label{kernel_est}
\end{equation}%
We will obtain the following stronger inequality,%
\begin{equation}
\Phi ^{\left( -1\right) }\left( \int_{B\left( 0,r_{0}\right) }\Phi \left(
T_{B\left( 0,r_{0}\right) }g\right) d\mu _{r_{0}}\right) \leq C\varphi
\left( r_{0}\right) \ \left\Vert g\right\Vert _{L^{1}\left( \mu
_{r_{0}}\right) }\ ,  \label{Phi bump}
\end{equation}%
which we refer to as the \emph{strong} $\left( \Phi ,A,\varphi \right) $%
-Sobolev Orlicz bump inequality, and which is stronger by the
subrepresentation inequality $w\lesssim T_{B\left( 0,r_{0}\right) }\nabla
_{A}w$ on $B\left( 0,r_{0}\right) $. But this inequality cannot in general
be reversed. When we wish to emphasize that we are working with (\ref{Phi
bump'}), we will often call it the \emph{standard} $\left( \Phi ,A,\varphi
\right) $-Sobolev Orlicz bump inequality.

Recall the operator $T_{B\left( 0,r_{0}\right) }:L^{1}\left( \mu
_{r_{0}}\right) \rightarrow L^{\Phi }\left( \mu _{r_{0}}\right) $ defined by 
\begin{equation*}
T_{B\left( 0,r_{0}\right) }g(x)\equiv \int_{B(0,r_{0})}K_{B\left(
0,r_{0}\right) }\left( x,y\right) g(y)dy
\end{equation*}%
with kernel $K$ defined as in (\ref{kernel_est}). We begin by proving that
the bound (\ref{Phi bump}) holds if the following endpoint inequality holds:%
\begin{equation}
\Phi ^{-1}\left( \sup_{y\in B}\int_{B}\Phi \left( K(x,y)|B|\alpha \right)
d\mu (x)\right) \leq C\alpha \varphi \left( r\right) \ .  \label{endpoint'}
\end{equation}%
for all $\alpha >0$. Indeed, if (\ref{endpoint'}) holds, then with $%
g=\left\vert \nabla _{A}w\right\vert $ and $\alpha =\left\Vert g\right\Vert
_{L^{1}}=\left\Vert \nabla _{A}w\right\Vert _{L^{1}}$, we have using first
the subrepresentation inequality, and then Jensen's inequality applied to
the convex function $\Phi $,%
\begin{eqnarray*}
\int_{B}\Phi (w)d\mu (x) &\lesssim &\int_{B}\Phi \left( \int_{B}K(x,y)\ |B|\
||g||_{L^{1}(\mu )}\frac{g\left( y\right) d\mu \left( y\right) }{%
||g||_{L^{1}(\mu )}}\right) d\mu (x) \\
&\leq &\int_{B}\int_{B}\Phi \left( K(x,y)\ |B|\ ||g||_{L^{1}(\mu )}\right) 
\frac{g\left( y\right) d\mu \left( y\right) }{||g||_{L^{1}(\mu )}}d\mu (x) \\
&\leq &\int_{B}\left\{ \sup_{y\in B}\int_{B}\Phi \left( K(x,y)\ |B|\
||g||_{L^{1}(\mu )}\right) d\mu (x)\right\} \frac{g\left( y\right) d\mu
\left( y\right) }{||g||_{L^{1}(\mu )}} \\
&\leq &\Phi \left( C\varphi \left( r\right) \ ||g||_{L^{1}(\mu )}\right)
\int_{B}\frac{g\left( y\right) d\mu \left( y\right) }{||g||_{L^{1}(\mu )}}%
=\Phi \left( C\varphi \left( r\right) \ ||g||_{L^{1}(\mu )}\right) ,
\end{eqnarray*}%
and so%
\begin{equation*}
\Phi ^{-1}\left( \int_{B}\Phi (w)d\mu (x)\right) \lesssim C\varphi \left(
r\right) \ ||\nabla _{A}w||_{L^{1}(\mu )}.
\end{equation*}%
The converse follows from Fatou's lemma, but we will not need this. Note
that (\ref{endpoint'}) is obtained from (\ref{Phi bump}) by replacing $%
g\left( y\right) dy$ with the point mass $|B|\alpha \delta _{x}\left(
y\right) $ so that $Tg\left( x\right) \rightarrow K(x,y)\ |B|\ \alpha $.

\begin{remark}
The inhomogeneous condition (\ref{endpoint'}) is in general stronger than
its homogeneous counterpart 
\begin{equation*}
\sup_{y\in B\left( 0,r_{0}\right) }\left\Vert K_{B\left( 0,r_{0}\right)
}\left( \cdot ,y\right) \left\vert B\left( 0,r_{0}\right) \right\vert
\right\Vert _{L^{\Phi }\left( \mu _{r_{0}}\right) }\leq C\varphi \left(
r_{0}\right) \ ,
\end{equation*}
but is equivalent to it when $\Phi $ is submultiplicative. We will not
however use this observation.
\end{remark}

Now we turn to the explicit near power bumps $\Phi $ in (\ref{def Phi m ext}%
), which satisfy 
\begin{equation*}
\Phi \left( t\right) =\Phi _{m}\left( t\right) =e^{\left( \left( \ln
t\right) ^{\frac{1}{m}}+1\right) ^{m}},\ \ \ \ \ t>e^{2^{m}},
\end{equation*}%
for $m\in \left( 1,\infty \right) $. Let $\psi (t)=\left( 1+\left( \ln
t\right) ^{-\frac{1}{m}}\right) ^{m}-1$ for $t>E=e^{2^{m}}$ and write $\Phi
\left( t\right) =t^{1+\psi \left( t\right) }$.

\begin{proposition}
\label{sob}Let $0<r_{0}<1$ and $C_{m}>0$. Suppose that the geometry $F$
satisfies the monotonicity property: 
\begin{equation}
\varphi \left( r\right) \equiv \frac{1}{|F^{\prime }(r)|}e^{C_{m}\left( 
\frac{\left\vert F^{\prime }\left( r\right) \right\vert ^{2}}{F^{\prime
\prime }(r)}+1\right) ^{m-1}}\text{ is an increasing function of }r\in
\left( 0,r_{0}\right) \text{.}  \label{mon prop}
\end{equation}
Then the $\left( \Phi ,\varphi \right) $-Sobolev inequality (\ref{Phi bump})
holds with geometry $F$, with $\varphi $ as in (\ref{mon prop}) and with $%
\Phi $ as in (\ref{def Phi m ext}), $m>1$.
\end{proposition}

For fixed $\Phi =\Phi _{m}$ with $m>1$, we now consider the geometry of
balls defined by%
\begin{eqnarray*}
F_{k,\sigma }\left( r\right) &=&\left( \ln \frac{1}{r}\right) \left( \ln
^{\left( k\right) }\frac{1}{r}\right) ^{\sigma }; \\
f_{k,\sigma }\left( r\right) &=&e^{-F_{k,\sigma }\left( r\right)
}=e^{-\left( \ln \frac{1}{r}\right) \left( \ln ^{\left( k\right) }\frac{1}{r}
\right) ^{\sigma }},
\end{eqnarray*}%
where $k\in \mathbb{N}$ and $\sigma >0$.

\begin{corollary}
\label{Sob Fsigma}The strong $\left( \Phi ,\varphi \right) $-Sobolev
inequality (\ref{Phi bump}) with $\Phi =\Phi _{m}$ as in (\ref{def Phi m ext}
), $m>1$, and geometry $F=F_{k,\sigma }$ holds if\newline
\qquad (\textbf{either}) $k\geq 2$ and $\sigma >0$ and $\varphi (r_{0})$ is
given by 
\begin{equation*}
\varphi (r_{0})=r_{0}^{1-C_{m}\frac{\left( \ln ^{\left( k\right) }\frac{1}{
r_{0}}\right) ^{\sigma \left( m-1\right) }}{\ln \frac{1}{r_{0}}}},\ \ \ \ \ 
\text{for }0<r_{0}\leq \beta _{m,\sigma },
\end{equation*}
for positive constants $C_{m}$ and $\beta _{m,\sigma }$ depending only on $m$
and $\sigma $;\newline
\qquad (\textbf{or}) $k=1$ and $\sigma <\frac{1}{m-1}$ and $\varphi (r_{0})$
is given by 
\begin{equation*}
\varphi (r_{0})=r_{0}^{1-C_{m}\frac{1}{\left( \ln \frac{1}{r_{0}}\right)
^{1-\sigma \left( m-1\right) }}},\ \ \ \ \ \text{for }0<r_{0}\leq \beta
_{m,\sigma },
\end{equation*}
for positive constants $C_{m}$ and $\beta _{m,\sigma }$ depending only on $m$
and $\sigma $.\newline
Conversely, the \emph{standard} $\left( \Phi ,\varphi \right) $-Sobolev
inequality (\ref{Phi bump'}) with $\Phi $ as in (\ref{def Phi m ext}), $m>1$
, fails if $k=1$ and $\sigma >\frac{1}{m-1}$.
\end{corollary}

\begin{proof}[Proof of Proposition \protect\ref{sob}]
It suffices to prove the endpoint inequality (\ref{endpoint'}), namely 
\begin{equation*}
\Phi ^{-1}\left( \sup_{y\in B}\int_{B}\Phi \left( K(x,y)|B|\alpha \right)
d\mu (x)\right) \leq C\alpha \varphi \left( r\left( B\right) \right) \ ,\ \
\ \ \ \alpha >0,
\end{equation*}%
for the balls and kernel associated with our geometry $F$, the Orlicz bump $%
\Phi $, and the function $\varphi \left( r\right) $ satisfying (\ref{mon
prop}). Fix parameters $m>1$ and $t_{m}>1$. Following the proof of \cite[%
Proposition 80]{KoRiSaSh19} we consider the specific function $\omega \left(
r\left( B\right) \right) $ given by 
\begin{equation*}
\omega \left( r\left( B\right) \right) =\frac{1}{t_{m}\left\vert F^{\prime
}\left( r\left( B\right) \right) \right\vert }.
\end{equation*}%
Using the submultiplicativity of $\Phi $ we have 
\begin{eqnarray*}
\int_{B}\Phi \left( K(x,y)|B|\alpha \right) d\mu (x) &=&\int_{B}\Phi \left( 
\frac{K(x,y)|B|}{\omega \left( r\left( B\right) \right) }\alpha \omega
\left( r\left( B\right) \right) \right) d\mu (x) \\
&\leq &\Phi \left( \alpha \omega \left( r\left( B\right) \right) \right)
\int_{B}\Phi \left( \frac{K(x,y)|B|}{\omega \left( r\left( B\right) \right) }%
\right) d\mu (x)
\end{eqnarray*}%
and we will now prove 
\begin{equation}
\int_{B}\Phi \left( \frac{K(x,y)|B|}{\omega \left( r\left( B\right) \right) }%
\right) d\mu (x)\leq C_{m}\varphi \left( r\left( B\right) \right) \left\vert
F^{\prime }\left( r\left( B\right) \right) \right\vert ,  \label{will prove}
\end{equation}%
for all small balls $B$ of radius $r\left( B\right) $ centered at the
origin. Altogether this will give us 
\begin{equation*}
\int_{B}\Phi \left( K(x,y)|B|\alpha \right) d\mu (x)\leq C_{m}\varphi \left(
r\left( B\right) \right) \left\vert F^{\prime }\left( r\left( B\right)
\right) \right\vert \Phi \left( \frac{\alpha }{t_{m}\left\vert F^{\prime
}\left( r\left( B\right) \right) \right\vert }\right) .
\end{equation*}%
Now we note that $x\Phi \left( y\right) =xy\frac{\Phi \left( y\right) }{y}%
\leq xy\frac{\Phi \left( xy\right) }{xy}=\Phi \left( xy\right) $ for $x\geq
1 $ since $\frac{\Phi \left( t\right) }{t}$ is monotone increasing. But from
(\ref{mon prop}) we have $\varphi \left( r\right) \left\vert F^{\prime
}\left( r\right) \right\vert =e^{C_{m}\left( \frac{\left\vert F^{\prime
}\left( r\right) \right\vert ^{2}}{F^{\prime \prime }(r)}+1\right)
^{m-1}}\gg 1$ and so 
\begin{equation*}
\int_{B}\Phi \left( K(x,y)|B|\alpha \right) d\mu (x)\leq \Phi \left(
C_{m}\varphi \left( r\left( B\right) \right) \left\vert F^{\prime }\left(
r\left( B\right) \right) \right\vert \alpha \frac{1}{t_{m}\left\vert
F^{\prime }\left( r\left( B\right) \right) \right\vert }\right) =\Phi \left( 
\frac{C_{m}}{t_{m}}\alpha \varphi \left( r\left( B\right) \right) \right) ,
\end{equation*}%
which is (\ref{endpoint'}) with $C=\frac{C_{m}}{t_{m}}$. Thus it remains to
prove (\ref{will prove}).

So we now take $B=B\left( 0,r_{0}\right) $ with $r_{0}\ll 1$ so that $\omega
\left( r\left( B\right) \right) =\omega \left( r_{0}\right) $. First, from 
\cite{KoRiSaSh19} we have the estimates 
\begin{equation*}
\left\vert B\left( 0,r_{0}\right) \right\vert \approx \frac{f(r_{0})}{%
|F^{\prime }(r_{0})|^{2}},
\end{equation*}%
and in $\Gamma (x,r)$ 
\begin{equation*}
K(x,y)\approx \frac{1}{h_{y_{1}-x_{1}}}\approx 
\begin{cases}
\begin{split}
\frac{1}{rf(x_{1})},\quad 0& <r=y_{1}-x_{1}<\frac{1}{|F^{\prime }(x_{1})|} \\
\frac{|F^{\prime }(x_{1}+r)|}{f(x_{1}+r)},\quad 0& <r=y_{1}-x_{1}\geq \frac{1%
}{|F^{\prime }(x_{1})|}
\end{split}%
\end{cases}%
.
\end{equation*}%
Next, write $\Phi (t)$ as 
\begin{equation*}
\Phi (t)=t^{1+\psi (t)},\ \ \ \ \ \text{for }t>0,
\end{equation*}%
where for $t\geq E$, 
\begin{eqnarray*}
t^{1+\psi (t)} &=&\Phi (t)=e^{\left( \left( \ln t\right) ^{\frac{1}{m}%
}+1\right) ^{m}}=t^{\left( 1+\left( \ln t\right) ^{-\frac{1}{m}}\right) ^{m}}
\\
&\Longrightarrow &\psi (t)=\left( 1+\left( \ln t\right) ^{-\frac{1}{m}%
}\right) ^{m}-1\approx \frac{m}{\left( \ln t\right) ^{1/m}},
\end{eqnarray*}%
and for $t<E$, 
\begin{eqnarray*}
t^{1+\psi (t)} &=&\Phi (t)=\frac{\Phi (E)}{E}t \\
&\Longrightarrow &\left( 1+\psi (t)\right) \ln t=\ln \frac{\Phi (E)}{E}+\ln t
\\
&\Longrightarrow &\psi (t)=\frac{\ln \frac{\Phi (E)}{E}}{\ln t}.
\end{eqnarray*}%
Now temporarily fix $y=\left( y_{1},y_{2}\right) \in B_{+}\left(
0,r_{0}\right) \equiv \left\{ x\in B\left( 0,r_{0}\right) :x_{1}>0\right\} $%
. We then have for $0<a<b<r_{0}$ that%
\begin{eqnarray*}
\mathcal{I}_{a,b}\left( y\right) &\equiv &\int_{\left\{ x\in B_{+}\left(
0,r_{0}\right) :a\leq y_{1}-x_{1}\leq b\right\} \cap \Gamma ^{\ast
}(y,r_{0})}\Phi \left( K_{B\left( 0,r_{0}\right) }\left( x,y\right) \frac{%
\left\vert B\left( 0,r_{0}\right) \right\vert }{\omega \left( r_{0}\right) }%
\right) \frac{dx}{\left\vert B\left( 0,r_{0}\right) \right\vert } \\
&=&\int_{y_{1}-b}^{y_{1}-a}\left\{
\int_{y_{2}-h_{y_{1}-x_{1}}}^{y_{2}+h_{y_{1}-x_{1}}}\Phi \left( \frac{1}{%
h_{y_{1}-x_{1}}}\left\vert B\left( 0,r_{0}\right) \right\vert \frac{%
\left\vert B\left( 0,r_{0}\right) \right\vert }{\omega \left( r_{0}\right) }%
\right) dx_{2}\right\} \frac{dx_{1}}{\left\vert B\left( 0,r_{0}\right)
\right\vert } \\
&=&\int_{y_{1}-b}^{y_{1}-a}2h_{y_{1}-x_{1}}\Phi \left( \frac{1}{%
h_{y_{1}-x_{1}}}\frac{\left\vert B\left( 0,r_{0}\right) \right\vert }{\omega
\left( r_{0}\right) }\right) \frac{dx_{1}}{\left\vert B\left( 0,r_{0}\right)
\right\vert } \\
&=&\int_{y_{1}-b}^{y_{1}-a}2h_{y_{1}-x_{1}}\left( \frac{1}{h_{y_{1}-x_{1}}}%
\frac{\left\vert B\left( 0,r_{0}\right) \right\vert }{\omega \left(
r_{0}\right) }\right) \left( \frac{1}{h_{y_{1}-x_{1}}}\frac{\left\vert
B\left( 0,r_{0}\right) \right\vert }{\omega \left( r_{0}\right) }\right)
^{\psi \left( \frac{1}{h_{y_{1}-x_{1}}}\frac{\left\vert B\left(
0,r_{0}\right) \right\vert }{\omega \left( r_{0}\right) }\right) }\frac{%
dx_{1}}{\left\vert B\left( 0,r_{0}\right) \right\vert }
\end{eqnarray*}%
which simplifies to 
\begin{eqnarray*}
\mathcal{I}_{a,b}\left( y\right) &=&\frac{2}{\omega \left( r_{0}\right) }%
\int_{y_{1}-b}^{y_{1}-a}\left( \frac{1}{h_{y_{1}-x_{1}}}\frac{\left\vert
B\left( 0,r_{0}\right) \right\vert }{\omega \left( r_{0}\right) }\right)
^{\psi \left( \frac{1}{h_{y_{1}-x_{1}}}\frac{\left\vert B\left(
0,r_{0}\right) \right\vert }{\omega \left( r_{0}\right) }\right) }dx_{1} \\
&=&\frac{2}{\omega \left( r_{0}\right) }\int_{a}^{b}\left( \frac{1}{h_{r}}%
\frac{\left\vert B\left( 0,r_{0}\right) \right\vert }{\omega \left(
r_{0}\right) }\right) ^{\psi \left( \frac{1}{h_{r}}\frac{\left\vert B\left(
0,r_{0}\right) \right\vert }{\omega \left( r_{0}\right) }\right) }dr.
\end{eqnarray*}%
Thus we have 
\begin{eqnarray*}
&&\int_{B_{+}\left( 0,r_{0}\right) }\Phi \left( K_{B\left( 0,r_{0}\right)
}\left( x,y\right) \frac{\left\vert B\left( 0,r_{0}\right) \right\vert }{%
\omega \left( r_{0}\right) }\right) \frac{dx}{\left\vert B\left(
0,r_{0}\right) \right\vert } \\
&=&\mathcal{I}_{0,y_{1}}\left( x\right) \\
&=&\frac{2}{\omega \left( r_{0}\right) }\int_{0}^{y_{1}}\left( \frac{1}{h_{r}%
}\frac{\left\vert B\left( 0,r_{0}\right) \right\vert }{\omega \left(
r_{0}\right) }\right) ^{\psi \left( \frac{1}{h_{r}}\frac{\left\vert B\left(
0,r_{0}\right) \right\vert }{\omega \left( r_{0}\right) }\right) }dr\ .
\end{eqnarray*}

To prove (\ref{will prove}) it thus suffices to show 
\begin{equation}
\mathcal{I}_{0,y_{1}}=\frac{1}{\omega \left( r_{0}\right) }
\int_{0}^{y_{1}}\left( \frac{\left\vert B\left( 0,r_{0}\right) \right\vert }{%
h_{r}\omega \left( r_{0}\right) }\right) ^{\psi \left( \frac{ \left\vert
B\left( 0,r_{0}\right) \right\vert }{h_{r}\omega \left( r_{0}\right) }%
\right) }dr\leq C_{m}\ \varphi \left( r_{0}\right) \left\vert F^{\prime
}\left( r_{0}\right) \right\vert \ ,  \label{the integral}
\end{equation}
where $C_{0}$ is a sufficiently large positive constant.

To prove this we divide the interval $\left( 0,y_{1}\right) $ of integration
in $r$ into three regions:

(\textbf{1}): the small region $\mathcal{S}$ where $\frac{|B(0,r_{0})|}{
h_{r}\omega \left( r_{0}\right) }\leq E$,

(\textbf{2}): the big region $\mathcal{R}_{1}$ that is disjoint from $%
\mathcal{S}$ and where $r=y_{1}-x_{1}<\frac{1}{\left\vert F^{\prime }\left(
x_{1}\right) \right\vert }$ and

(\textbf{3}): the big region $\mathcal{R}_{2}$ that is disjoint from $%
\mathcal{S}$ and where $r=y_{1}-x_{1}\geq \frac{1}{\left\vert F^{\prime
}\left( x_{1}\right) \right\vert }$.

In the small region $\mathcal{S}$ we use that $\Phi $ is linear on $\left[
0,E\right] $\ to obtain that the integral in the right hand side of (\ref%
{the integral}), when restricted to those $r\in \left( 0,y_{1}\right) $ for
which $\frac{|B(0,r_{0})|}{h_{r}\omega \left( r_{0}\right) }\leq E$, is
equal to 
\begin{eqnarray*}
&&\frac{1}{\omega \left( r_{0}\right) }\int_{0}^{y_{1}}\left( \frac{
\left\vert B\left( 0,r_{0}\right) \right\vert }{h_{r}\omega \left(
r_{0}\right) }\right) ^{\frac{\ln \frac{\Phi (E)}{E}}{\ln \left( \frac{
\left\vert B\left( 0,r_{0}\right) \right\vert }{h_{r}\omega \left(
r_{0}\right) }\right) }}dr \\
&=&\frac{1}{\omega \left( r_{0}\right) }\int_{0}^{y_{1}}e^{\ln \frac{ \Phi
(E)}{E}}dr=\frac{1}{\omega \left( r_{0}\right) }\frac{\Phi (E)}{E} y_1 \\
&\leq &\frac{\Phi (E)}{E}t_{m}\ r_{0}\left\vert F^{\prime }\left(
r_{0}\right) \right\vert \ ,
\end{eqnarray*}
since $\omega \left( r_{0}\right) =\frac{1}{t_{m}\left\vert F^{\prime
}\left( r_{0}\right) \right\vert }$. We now turn to the first big region $%
\mathcal{R}_{1}$ where we have $h_{y_{1}-x_{1}}\approx rf(x_{1})=rf(y_1-r)$.
The integral to be evaluated is 
\begin{equation*}
\frac{1}{\omega \left( r_{0}\right) }\int_{0}^{y_1 }\left( \frac{|B(0,r_{0})|%
}{h_{r}\omega \left( r_{0}\right) }\right) ^{\psi \left( \frac{|B(0,r_{0})|}{%
h_{r}\omega \left( r_{0}\right) }\right) }dr,\quad\text{where}\quad \frac{%
|B(0,r_{0})|}{h_{r}\omega \left( r_{0}\right) }\approx \frac{|B(0,r_{0})|}{%
rf(y_1-r)\omega \left( r_{0}\right) }
\end{equation*}
Now we note that since $x_1<y_1$, we have $\frac{1}{\left\vert F^{\prime
}\left( x_{1}\right) \right\vert }\leq \frac{1}{\left\vert F^{\prime }\left(
y_{1}\right) \right\vert }$, and thus in this region we have $x_1<y_1\leq
x_1+\frac{1}{\left\vert F^{\prime }\left( y_{1}\right) \right\vert }$, and
it is sufficient to evaluate 
\begin{equation*}
\frac{1}{\omega \left( r_{0}\right) }\int_{0}^{\frac{1}{\left\vert F^{\prime
}\left( y_{1}\right) \right\vert } }\left( \frac{|B(0,r_{0})|}{h_{r}\omega
\left( r_{0}\right) }\right) ^{\psi \left( \frac{|B(0,r_{0})|}{h_{r}\omega
\left( r_{0}\right) }\right) }dr.
\end{equation*}
From the inequalities for $y_1$ it also follows that $f(x_1)\approx f(y_1)$,
so $h_{y_{1}-x_{1}}\approx rf(y_{1})$. Write 
\begin{equation*}
\frac{|B(0,r_{0})|}{h_{r}\omega \left( r_{0}\right) }\leq C^{\prime }\frac{%
|B(0,r_{0})|}{rf(y_1)\omega \left( r_{0}\right) }\leq C\frac{t_mf(r_{0})}{%
rf(y_1)\left\vert F^{\prime }\left( r_{0}\right)\right\vert },
\end{equation*}
and we will now evaluate the following integral 
\begin{equation*}
\frac{1}{\omega \left( r_{0}\right) }\int_{0}^{\frac{1}{|F^{\prime }(y_{1})|}%
}\left( \frac{A}{r}\right) ^{\psi \left( \frac{A}{r}\right) }dr,\quad\text{%
where}\quad A=C\frac{t_mf(r_{0})}{f(y_1)\left\vert F^{\prime }\left(
r_{0}\right)\right\vert }.
\end{equation*}%
Making a change of variables 
\begin{equation*}
R=\frac{A}{r}=\frac{A\left( y_{1}\right) }{r},
\end{equation*}%
we obtain%
\begin{equation*}
\frac{1}{\omega \left( r_{0}\right) }\int_{0}^{\frac{1}{|F^{\prime }(y_{1})|}%
}\left( \frac{A}{r}\right) ^{\psi \left( \frac{A}{r}\right) }dr=\frac{1}{
\omega \left( r_{0}\right) }A\int_{A\left\vert F^{\prime }\left(
y_{1}\right) \right\vert }^{\infty }R^{\psi (R)-2}dR.
\end{equation*}%
Integrating by parts gives 
\begin{align*}
\int_{A\left\vert F^{\prime }\left( y_{1}\right) \right\vert }^{\infty
}R^{\psi (R)-2}dR& =\int_{A\left\vert F^{\prime }\left( y_{1}\right)
\right\vert }^{\infty }R^{\psi (R)+1}\left( -\frac{1}{2R^{2}}\right)
^{\prime }dR \\
& =-\frac{R^{\psi (R)+1}}{2R^{2}}\big|_{A\left\vert F^{\prime }\left(
x_{1}\right) \right\vert }^{\infty }+\int_{A\left\vert F^{\prime }\left(
x_{1}\right) \right\vert }^{\infty }\left( R^{\psi (R)+1}\right) ^{\prime }%
\frac{1}{2R^{2}}dR \\
& \leq \frac{\left( A|F^{\prime }(y_{1})|\right) ^{\psi \left( A|F^{\prime
}(y_{1})|\right) }}{2A|F^{\prime }(y_{1})|}+\int_{A\left\vert F^{\prime
}\left( y_{1}\right) \right\vert }^{\infty }\frac{1}{2}R^{\psi (R)-2}\left(
1+C\frac{m-1}{\left( \ln R\right) ^{\frac{1}{m}}}\right) dR \\
& \leq \frac{\left( A|F^{\prime }(y_{1})|\right) ^{\psi \left( A|F^{\prime
}(y_{1})|\right) }}{2A|F^{\prime }(y_{1})|}+\frac{1+C\frac{m-1}{\left( \ln
E\right) ^{\frac{1}{m}}}}{2}\int_{A\left\vert F^{\prime }\left( y_{1}\right)
\right\vert }^{\infty }R^{\psi (R)-2}dR,
\end{align*}%
where we used 
\begin{equation*}
\left\vert \psi ^{\prime }(R)\right\vert \leq C\frac{1}{R}\frac{1}{\left(
\ln R\right) ^{\frac{m+1}{m}}}.
\end{equation*}%
Taking $E$ large enough depending on $m$ we can assure 
\begin{equation*}
\frac{1+C\frac{m-1}{\left( \ln E\right) ^{\frac{1}{m}}}}{2}\leq \frac{3}{4},
\end{equation*}%
which gives 
\begin{equation*}
\int_{A\left\vert F^{\prime }\left( y_{1}\right) \right\vert }^{\infty
}R^{\psi (R)-2}dR\lesssim \frac{\left( A|F^{\prime }(y_{1})|\right) ^{\psi
\left( A|F^{\prime }(y_{1})|\right) }}{A|F^{\prime }(y_{1})|},
\end{equation*}%
and therefore 
\begin{eqnarray*}
\mathcal{I}_{0,\frac{1}{\left\vert F^{\prime }\left( y_{1}\right)
\right\vert }}\left( x\right) &\lesssim&\frac{1}{\omega \left( r_{0}\right) }
A\int_{A\left\vert F^{\prime }\left( y_{1}\right) \right\vert }^{\infty
}R^{\psi (R)-2}dR \\
&\lesssim &\frac{1}{\omega \left( r_{0}\right) |F^{\prime }(y_{1})|}\left(
A\left( y_{1}\right) |F^{\prime }(y_{1})|\right) ^{\psi \left(
A(y_{1})|F^{\prime }(y_{1})|\right) }.
\end{eqnarray*}
We now look for the maximum of the function on the right hand side 
\begin{eqnarray*}
\mathcal{F}(y_{1}) &\equiv &\frac{1}{\omega \left( r_{0}\right) |F^{\prime
}(y_{1})|}\left( A\left( y_{1}\right) |F^{\prime }(y_{1})|\right) ^{\psi
\left( A(y_{1})|F^{\prime }(y_{1})|\right) } \\
&=&t_{m}\left\vert F^{\prime }\left( r_{0}\right) \right\vert \frac{1}{
\left\vert F^{\prime }\left( y_{1}\right) \right\vert }\left( c(r_{0})\frac{
\left\vert F^{\prime }\left( y_{1}\right) \right\vert }{f\left( y_{1}\right) 
}\right) ^{\psi \left( c(r_{0})\frac{\left\vert F^{\prime }\left(
y_{1}\right) \right\vert }{f\left( y_{1}\right) }\right) }
\end{eqnarray*}%
where 
\begin{equation*}
c(r_{0})= f\left( y_{1}\right) A\left( y_{1}\right) =\frac{Ct_{m}\ f(r_{0})}{%
|F^{\prime }(r_{0})|}.
\end{equation*}%
Using the definition of $\psi (t)$ and $B\left( y_{1}\right) \equiv \ln %
\left[ c(r_{0})\frac{\left\vert F^{\prime }\left( y_{1}\right) \right\vert }{
f\left( y_{1}\right) }\right] $, we can rewrite $\mathcal{F}(y_{1})$ as 
\begin{equation}
\mathcal{F}(y_{1})=t_{m}\left\vert F^{\prime }\left( r_{0}\right)
\right\vert \frac{1}{\left\vert F^{\prime }\left( y_{1}\right) \right\vert }%
\exp \left( \left( 1+B\left( y_{1}\right) ^{\frac{1}{m}}\right) ^{m}-B\left(
y_{1}\right) \right) .  \label{cali_F}
\end{equation}%
Let $y_{1}^{\ast }\in \left( 0,r_{0}\right] $ be the point at which $%
\mathcal{F}$ takes its maximum. Differentiating $\mathcal{F}(y_{1})$ with
respect to $y_{1}$ and then setting the derivative equal to zero, we obtain
that $y_{1}^{\ast }$ satisfies the equation, 
\begin{equation*}
\frac{F^{\prime \prime }(y_{1}^{\ast })}{|F^{\prime }(y_{1}^{\ast })|^{2}}%
=\left( \left( 1+B\left( y_{1}^{\ast }\right) ^{-\frac{1}{m}}\right)
^{m-1}-1\right) \left( 1+\frac{F^{\prime \prime }(y_{1}^{\ast })}{|F^{\prime
}(y_{1}^{\ast })|^{2}}\right) .
\end{equation*}%
Simplifying gives the following implicit expression for $y_{1}^{\ast }$ that
maximizes $\mathcal{F}(y_{1})$ 
\begin{equation*}
B\left( y_{1}^{\ast }\right) =\ln \left[ c(r_{0})\frac{\left\vert F^{\prime
}\left( y_{1}^{\ast }\right) \right\vert }{f\left( y_{1}^{\ast }\right) }%
\right] =\left( \left( 1+\frac{F^{\prime \prime }(y_{1}^{\ast })}{|F^{\prime
}(y_{1}^{\ast })|^{2}+F^{\prime \prime }(y_{1}^{\ast })}\right) ^{\frac{1}{
m-1}}-1\right) ^{-m}.
\end{equation*}%
To estimate $\mathcal{F}(y_{1}^{\ast })$ in an effective way, we set $%
b\left( y_{1}^{\ast }\right) \equiv \frac{F^{\prime \prime }(y_{1}^{\ast })}{
|F^{\prime }(y_{1}^{\ast })|^{2}+F^{\prime \prime }(y_{1}^{\ast })}$ and
begin with 
\begin{align*}
& \left( 1+B\left( y_{1}\right) ^{\frac{1}{m}}\right) ^{m}-B\left(
y_{1}\right) =\left( 1+\left( \ln \left[ c(r_{0})\frac{\left\vert F^{\prime
}\left( y_{1}^{\ast }\right) \right\vert }{f\left( y_{1}^{\ast }\right) }%
\right] \right) ^{\frac{1}{m}}\right) ^{m}-\ln \left[ c(r_{0})\frac{
\left\vert F^{\prime }\left( y_{1}^{\ast }\right) \right\vert }{f\left(
y_{1}^{\ast }\right) }\right] \\
& =\frac{\left( 1+\frac{F^{\prime \prime }(y_{1}^{\ast })}{|F^{\prime
}(y_{1}^{\ast })|^{2}+F^{\prime \prime }(y_{1}^{\ast })}\right) ^{\frac{m}{
m-1}}-1}{\left( \left( 1+\frac{F^{\prime \prime }(y_{1}^{\ast })}{|F^{\prime
}(y_{1}^{\ast })|^{2}+F^{\prime \prime }(y_{1}^{\ast })}\right) ^{\frac{1}{
m-1}}-1\right) ^{m}}=\frac{\left( 1+b\left( y_{1}^{\ast }\right) \right) ^{ 
\frac{m}{m-1}}-1}{\left( \left( 1+b\left( y_{1}^{\ast }\right) \right) ^{ 
\frac{1}{m-1}}-1\right) ^{m}} \\
& \leq C_{m}\left( \frac{1}{b\left( y_{1}^{\ast }\right) }\right)
^{m-1}=C_{m}\left( \frac{|F^{\prime }(y_{1}^{\ast })|^{2}+F^{\prime \prime
}(y_{1}^{\ast })}{F^{\prime \prime }(y_{1}^{\ast })}\right)
^{m-1}=C_{m}\left( 1+\frac{|F^{\prime }(y_{1}^{\ast })|^{2}}{F^{\prime
\prime }(y_{1}^{\ast })}\right) ^{m-1},
\end{align*}%
where in the last inequality we used (\textbf{1}) the fact that $b\left(
y_{1}^{\ast }\right) =\frac{F^{\prime \prime }(y_{1}^{\ast })}{|F^{\prime
}(y_{1}^{\ast })|^{2}+F^{\prime \prime }(y_{1}^{\ast })}<1$ provided $%
y_{1}^{\ast }\leq r$, which we may assume since otherwise we are done, and (%
\textbf{2}) the inequality 
\begin{equation*}
\frac{\left( 1+b\right) ^{\frac{m}{m-1}}-1}{\left( \left( 1+b\right) ^{\frac{
1}{m-1}}-1\right) ^{m}}\leq \frac{1}{2}m\left( 2m-1\right) \left( m-1\right)
^{2m}b^{1-m},\ \ \ \ \ 0\leq b<1,
\end{equation*}%
which follows easily from upper and lower estimates on the binomial series.
Combining this with (\ref{cali_F}) we thus obtain the following upper bound 
\begin{equation*}
\mathcal{F}(y_{1})\leq t_{m}\left\vert F^{\prime }\left( r_{0}\right)
\right\vert \frac{1}{\left\vert F^{\prime }\left( y_{1}^{\ast }\right)
\right\vert }e^{C_{m}\left( 1+\frac{|F^{\prime }(y_{1}^{\ast })|^{2}}{
F^{\prime \prime }(y_{1}^{\ast })}\right) ^{m-1}}=t_{m}\left\vert F^{\prime
}\left( r_{0}\right) \right\vert \ \varphi \left( y_{1}^{\ast }\right) ,
\end{equation*}%
with $\varphi $ as in (\ref{mon prop}). Using the monotonicity of $\varphi
\digamma $ we therefore obtain 
\begin{equation*}
\mathcal{I}_{0,\frac{1}{\left\vert F^{\prime }\left( y_{1}\right)
\right\vert }}\left( x\right) \lesssim \mathcal{F}(y_{1})\leq
t_{m}\left\vert F^{\prime }\left( r_{0}\right) \right\vert \ \varphi \left(
r_{0}\right) =t_{m}\left\vert F^{\prime }\left( r_{0}\right) \right\vert
\varphi \left( r_{0}\right) ,
\end{equation*}%
which is the estimate required in (\ref{the integral}).

For the second big region $\mathcal{R}_{2}$ we have 
\begin{equation*}
\frac{1}{h_{y_{1}-x_{1}}}\approx \frac{|F^{\prime }(x_{1}+r)|}{f(x_{1}+r)}=%
\frac{|F^{\prime }(y_{1})|}{f(y_{1})},
\end{equation*}
and the integral to be estimated becomes 
\begin{align*}
I_{\mathcal{R}_{2}}&\equiv \frac{1}{\omega \left( r_{0}\right) }%
\int_{x_{1}\in \mathcal{R}_{2}}\left( c(r_{0})\frac{\left\vert F^{\prime
}\left( y_{1}\right) \right\vert }{f\left( y_{1}\right) }\right) ^{\psi
\left( c(r_{0})\frac{\left\vert F^{\prime }\left( y_{1}\right) \right\vert }{%
f\left( y_{1}\right) }\right) }dx_{1} \\
&\leq \frac{y_{1}}{\omega \left( r_{0}\right) }\left( c(r_{0})\frac{%
\left\vert F^{\prime }\left( y_{1}\right) \right\vert }{f\left( y_{1}\right) 
}\right) ^{\psi \left( c(r_{0})\frac{\left\vert F^{\prime }\left(
y_{1}\right) \right\vert }{f\left( y_{1}\right) }\right) } \\
&=t_m|F^{\prime }(r_{0})|y_{1}\left( c(r_{0})\frac{\left\vert F^{\prime
}\left( y_{1}\right) \right\vert }{f\left( y_{1}\right) }\right) ^{\psi
\left( c(r_{0})\frac{\left\vert F^{\prime }\left( y_{1}\right) \right\vert }{%
f\left( y_{1}\right) }\right) },
\end{align*}
where 
\begin{equation*}
c(r_{0})=\frac{t_{m}\ f(r_{0})}{|F^{\prime }(r_{0})|}.
\end{equation*}%
We now look for the maximum of the function 
\begin{equation*}
\mathcal{G}(y_{1}) \equiv t_{m}\left\vert F^{\prime }\left( r_{0}\right)
\right\vert y_1\left( c(r_{0})\frac{ \left\vert F^{\prime }\left(
y_{1}\right) \right\vert }{f\left( y_{1}\right) }\right) ^{\psi \left(
c(r_{0})\frac{\left\vert F^{\prime }\left( y_{1}\right) \right\vert }{%
f\left( y_{1}\right) }\right) },
\end{equation*}%
and look for the maximum of $\mathcal{G}(y_{1})$ on $\left( 0,r_{0}\right] $%
. We claim that a bound for $\mathcal{G}$ can be obtained in a similar way
and yields 
\begin{equation*}
\mathcal{G}(y_{1})\leq C_{m}\left\vert F^{\prime }\left( r_{0}\right)
\right\vert \varphi \left( r_{0}\right) ,
\end{equation*}%
where $\varphi \left( r_{0}\right) $ satisfies (\ref{mon prop}) with a
constant $C_{m}$ slightly bigger than in the case of $\mathcal{F}$. Indeed,
rewriting $\mathcal{G}(y_{1})$ in a form similar to (\ref{cali_F}) we have 
\begin{align*}
\mathcal{G}(y_{1})&=t_{m}\left\vert F^{\prime }\left( r_{0}\right)
\right\vert y_{1} \exp \left( \left( 1+\left( \ln \left[ c(r_{0})\frac{%
\left\vert F^{\prime }\left( y_{1}\right) \right\vert }{f\left( y_{1}\right) 
}\right] \right) ^{ \frac{1}{m}}\right) ^{m}-\ln \left[ c(r_{0})\frac{%
\left\vert F^{\prime }\left( y_{1}\right) \right\vert }{f\left( y_{1}\right) 
}\right] \right) \\
&=t_{m}\left\vert F^{\prime }\left( r_{0}\right) \right\vert y_{1} \exp
\left( \left( 1+B\left( y_{1}\right) ^{\frac{1}{m}}\right) ^{m}-B\left(
y_{1}\right) \right)
\end{align*}%
Again, we differentiate and equate the derivative to zero to obtain the
following implicit expression for $y_{1}^{\ast }$ maximizing $\mathcal{G}%
(y_{1})$: 
\begin{equation*}
1 =\left( \left( 1+\left( \ln \left[ c(r_{0})\frac{\left\vert F^{\prime
}\left( y_{1}^{\ast }\right) \right\vert }{f\left( y_{1}^{\ast }\right) }%
\right] \right) ^{-\frac{1}{m}}\right) ^{m-1}-1\right) y_{1}\left(
|F^{\prime }(y_{1}^{\ast })|+\frac{ F^{\prime \prime }(y_{1}^{\ast })}{%
|F^{\prime }(y_{1}^{\ast })|}\right) .
\end{equation*}
A calculation similar to the one for the function $\mathcal{F}$ gives 
\begin{align*}
\left( 1+\left( \ln \left[ c(r_{0})\frac{\left\vert F^{\prime }\left(
y_{1}^{\ast }\right) \right\vert }{f\left( y_{1}^{\ast }\right) }\right]
\right) ^{\frac{1}{m}}\right) ^{m}-& \ln \left[ c(r_{0})\frac{\left\vert
F^{\prime }\left( y_{1}^{\ast }\right) \right\vert }{f\left( y_{1}^{\ast
}\right) }\right] =\frac{\left( 1+\frac{\left\vert F^{\prime }\left(
y_{1}^{\ast }\right) \right\vert }{y_{1}^{\ast }|F^{\prime }(y_{1}^{\ast
})|^{2}+y_{1}^{\ast }F^{\prime \prime}(y_{1}^{\ast })}\right) ^{\frac{m}{m-1}%
}-1}{\left( \left(1+\frac{\left\vert F^{\prime }\left( y_{1}^{\ast }\right)
\right\vert }{y_{1}^{\ast }|F^{\prime }(y_{1}^{\ast })|^{2}+y_{1}^{\ast
}F^{\prime \prime}(y_{1}^{\ast })}\right) ^{\frac{1}{m-1}}-1\right) ^{m}} \\
& \leq C_{m}\left( \frac{y_{1}^{\ast }|F^{\prime }(y_{1}^{\ast
})|^{2}+y_{1}^{\ast }F^{\prime \prime}(y_{1}^{\ast })}{\left\vert F^{\prime
}\left( y_{1}^{\ast }\right) \right\vert }\right) ^{m-1}\leq \tilde{C}%
_{m}\left( 1+\frac{|F^{\prime }(y_{1}^{\ast })|^{2}}{F^{\prime \prime
}(y_{1}^{\ast })}\right) ^{m-1},
\end{align*}
where we used $|F^{\prime }(r)/F^{\prime \prime}(r)|\approx r$. From this
and the monotonicity condition we obtain 
\begin{equation*}
I_{\mathcal{R}_{2}}\lesssim\mathcal{G}(y_{1})\leq C_{m}\left\vert F^{\prime
}\left( r_{0}\right) \right\vert \varphi \left( r_{0}\right),
\end{equation*}
which concludes the estimate for the region $\mathcal{R}_{2}$.
\end{proof}

Now we turn to the proof of Corollary \ref{Sob Fsigma}.

\begin{proof}[Proof of Corollary \protect\ref{Sob Fsigma}]
We must first check that the monotonicity property (\ref{mon prop}) holds
for the indicated geometries $F_{k,\sigma }$, where 
\begin{eqnarray*}
f\left( r\right) &=&f_{k,\sigma }\left( r\right) \equiv \exp \left\{ -\left(
\ln \frac{1}{r}\right) \left( \ln ^{\left( k\right) }\frac{1}{r}\right)
^{\sigma }\right\} ; \\
F\left( r\right) &=&F_{k,\sigma }\left( r\right) \equiv \left( \ln \frac{1}{%
r }\right) \left( \ln ^{\left( k\right) }\frac{1}{r}\right) ^{\sigma }.
\end{eqnarray*}
Consider first the case $k=1$. Then $F\left( r\right) =F_{1,\sigma }\left(
r\right) =\left( \ln \frac{1}{r}\right) ^{1+\sigma }$ satisfies 
\begin{equation*}
F^{\prime }\left( r\right) =-\left( 1+\sigma \right) \frac{\left( \ln \frac{%
1 }{r}\right) ^{\sigma }}{r}\text{ and }F^{\prime \prime }\left( r\right)
=-\left( 1+\sigma \right) \left\{ -\frac{\left( \ln \frac{1}{r}\right)
^{\sigma }}{r^{2}}-\sigma \frac{\left( \ln \frac{1}{r}\right) ^{\sigma -1}}{
r^{2}}\right\} ,
\end{equation*}
which shows that 
\begin{eqnarray*}
\varphi \left( r\right) &=&\frac{1}{1+\sigma }\exp \left\{ -\ln \frac{1}{r}
-\sigma \ln \ln \frac{1}{r}+C_{m}\left( \frac{\left( 1+\sigma \right) ^{2} 
\frac{\left( \ln \frac{1}{r}\right) ^{2\sigma }}{r^{2}}}{\left( 1+\sigma
\right) \left\{ \frac{\left( \ln \frac{1}{r}\right) ^{\sigma }}{r^{2}}
+\sigma \frac{\left( \ln \frac{1}{r}\right) ^{\sigma -1}}{r^{2}}\right\} }
+1\right) ^{m-1}\right\} \\
&=&\frac{1}{1+\sigma }\exp \left\{ -\ln \frac{1}{r}-\sigma \ln \ln \frac{1}{%
r }+C_{m}\left( 1+\sigma \right) ^{m-1}\left( \frac{\left( \ln \frac{1}{r}
\right) ^{\sigma }}{\left\{ 1+\sigma \frac{1}{\ln \frac{1}{r}}\right\} }+ 
\frac{1}{1+\sigma }\right) ^{m-1}\right\} ,
\end{eqnarray*}
is increasing in $r$ provided both $\sigma \left( m-1\right) <1$ and $0\leq
r\leq \alpha _{m,\sigma }$, where $\alpha _{m,\sigma }$ is a positive
constant depending only on $m$ and $\sigma $. Hence we have the upper bound 
\begin{equation*}
\varphi \left( r\right) \leq \exp \left\{ -\ln \frac{1}{r}+C_{m}\left( \ln 
\frac{1}{r}\right) ^{\sigma \left( m-1\right) }\right\} =r^{1-C_{m}\frac{1}{
\left( \ln \frac{1}{r}\right) ^{1-\sigma \left( m-1\right) }}},\ \ \ \ \
0\leq r\leq \beta _{m,\sigma },
\end{equation*}
where $\beta _{m,\sigma }>0$ is chosen even smaller than $\alpha _{m,\sigma
} $ if necessary.

Thus in the case $\Phi =\Phi _{m}$ with $m>2$ and $F=F_{\sigma }$ with $%
0<\sigma <\frac{1}{m-1}$, we see that the norm $\varphi \left( r_{0}\right) $
of the Sobolev embedding satisfies 
\begin{equation*}
\varphi \left( r_{0}\right) \leq r_{0}^{1-C_{m}\frac{1}{\left( \ln \frac{1}{
r_{0}}\right) ^{1-\sigma \left( m-1\right) }}},\ \ \ \ \ \text{for }
0<r_{0}\leq \beta _{m,\sigma },
\end{equation*}
and hence that 
\begin{equation*}
\frac{\varphi \left( r_{0}\right) }{r_{0}}\leq \left( \frac{1}{r_{0}}\right)
^{\frac{C_{m}}{\left( \ln \frac{1}{r_{0}}\right) ^{1-\sigma \left(
m-1\right) }}}\ \ \ \ \ \text{for }0<r_{0}\leq \beta _{m,\sigma }.
\end{equation*}

Now consider the case $k\geq 2$. Our first task is to show that $F_{k,\sigma
}$ satisfies the structure conditions in Definition \ref{structure
conditions}. Only condition (5) is not obvious, so we now turn to that. We
have $F\left( r\right) =F_{k,\sigma }\left( r\right) =\left( \ln \frac{1}{r}%
\right) \left( \ln ^{\left( k\right) }\frac{1}{r}\right) ^{\sigma }$
satisfies 
\begin{eqnarray*}
F^{\prime }\left( r\right) &=&-\frac{\left( \ln ^{\left( k\right) }\frac{1}{r%
}\right) ^{\sigma }}{r}-\left( \ln \frac{1}{r}\right) \frac{\sigma \left(
\ln ^{\left( k\right) }\frac{1}{r}\right) ^{\sigma -1}}{\left( \ln ^{\left(
k-1\right) }\frac{1}{r}\right) \left( \ln ^{\left( k-2\right) }\frac{1}{r}%
\right) ...\left( \ln \frac{1}{r}\right) r} \\
&=&-\frac{\left( \ln ^{\left( k\right) }\frac{1}{r}\right) ^{\sigma }}{r}-%
\frac{\sigma \left( \ln ^{\left( k\right) }\frac{1}{r}\right) ^{\sigma -1}}{%
\left( \ln ^{\left( k-1\right) }\frac{1}{r}\right) \left( \ln ^{\left(
k-2\right) }\frac{1}{r}\right) ...\left( \ln ^{\left( 2\right) }\frac{1}{r}%
\right) r} \\
&=&-\frac{\left( \ln ^{\left( k\right) }\frac{1}{r}\right) ^{\sigma }}{r}%
\left\{ 1+\frac{\sigma }{\left( \ln ^{\left( k\right) }\frac{1}{r}\right)
\left( \ln ^{\left( k-1\right) }\frac{1}{r}\right) \left( \ln ^{\left(
k-2\right) }\frac{1}{r}\right) ...\left( \ln ^{\left( 2\right) }\frac{1}{r}%
\right) }\right\} \\
&=&-\frac{F\left( r\right) }{r\ln \frac{1}{r}}\left\{ 1+\frac{\sigma }{%
\left( \ln ^{\left( k\right) }\frac{1}{r}\right) \left( \ln ^{\left(
k-1\right) }\frac{1}{r}\right) ...\left( \ln ^{\left( 2\right) }\frac{1}{r}%
\right) }\right\} \equiv -\frac{F\left( r\right) \Lambda _{k}\left( r\right) 
}{r\ln \frac{1}{r}},
\end{eqnarray*}%
and 
\begin{eqnarray*}
F^{\prime \prime }\left( r\right) &=&-\frac{F^{\prime }\left( r\right)
\Lambda _{k}\left( r\right) }{r\ln \frac{1}{r}}-\frac{F\left( r\right)
\Lambda _{k}^{\prime }\left( r\right) }{r\ln \frac{1}{r}}-F\left( r\right)
\Lambda _{k}\left( r\right) \frac{d}{dr}\left( \frac{1}{r\ln \frac{1}{r}}%
\right) \\
&=&-\frac{F^{\prime }\left( r\right) \Lambda _{k}\left( r\right) }{r\ln 
\frac{1}{r}}-\frac{F\left( r\right) \Lambda _{k}^{\prime }\left( r\right) }{%
r\ln \frac{1}{r}}+\frac{F\left( r\right) \Lambda _{k}\left( r\right) }{%
r^{2}\ln \frac{1}{r}}\left( 1-\frac{1}{\ln \frac{1}{r}}\right) ,
\end{eqnarray*}%
where 
\begin{eqnarray*}
\Lambda _{k}^{\prime }\left( r\right) &=&\frac{d}{dr}\left( \frac{\sigma }{%
\left( \ln ^{\left( k\right) }\frac{1}{r}\right) \left( \ln ^{\left(
k-1\right) }\frac{1}{r}\right) ...\left( \ln ^{\left( 2\right) }\frac{1}{r}%
\right) }\right) \\
&=&-\sigma \sum_{j=2}^{k}\frac{\left( \ln ^{\left( j\right) }\frac{1}{r}%
\right) }{\left( \ln ^{\left( k\right) }\frac{1}{r}\right) ...\left( \ln
^{\left( 2\right) }\frac{1}{r}\right) }\frac{1}{\left( \ln ^{\left(
j-1\right) }\frac{1}{r}\right) ...\left( \ln \frac{1}{r}\right) r} \\
&=&-\sigma \frac{1}{\left( \ln ^{\left( k\right) }\frac{1}{r}\right)
...\left( \ln ^{\left( 2\right) }\frac{1}{r}\right) r}\sum_{j=2}^{k}\frac{%
\ln ^{\left( j\right) }\frac{1}{r}}{\left( \ln ^{\left( j-1\right) }\frac{1}{%
r}\right) ...\left( \ln \frac{1}{r}\right) } \\
&=&-\sigma \frac{1}{\left( \ln ^{\left( k\right) }\frac{1}{r}\right)
...\left( \ln ^{\left( 2\right) }\frac{1}{r}\right) r}\left( \frac{\ln
^{\left( 2\right) }\frac{1}{r}}{\ln \frac{1}{r}}+\sum_{j=3}^{k}\frac{\ln
^{\left( j\right) }\frac{1}{r}}{\left( \ln ^{\left( j-1\right) }\frac{1}{r}%
\right) ...\left( \ln \frac{1}{r}\right) }\right) \\
&=&-\sigma \frac{1}{\left( \ln ^{\left( k\right) }\frac{1}{r}\right)
...\left( \ln ^{\left( 2\right) }\frac{1}{r}\right) \left( \ln \frac{1}{r}%
\right) r}\left( \ln ^{\left( 2\right) }\frac{1}{r}+\sum_{j=3}^{k}\frac{\ln
^{\left( j\right) }\frac{1}{r}}{\left( \ln ^{\left( j-1\right) }\frac{1}{r}%
\right) ...\left( \ln ^{\left( 2\right) }\frac{1}{r}\right) }\right) .
\end{eqnarray*}%
Now 
\begin{equation*}
\ln ^{\left( 2\right) }\frac{1}{r}+\sum_{j=3}^{k}\frac{\ln ^{\left( j\right)
}\frac{1}{r}}{\left( \ln ^{\left( j-1\right) }\frac{1}{r}\right) ...\left(
\ln ^{\left( 2\right) }\frac{1}{r}\right) }\approx \ln ^{\left( 2\right) }%
\frac{1}{r},
\end{equation*}%
and so 
\begin{equation*}
-\Lambda _{k}^{\prime }\left( r\right) \approx \left\{ 
\begin{array}{ccc}
\frac{\sigma }{\left( \ln \frac{1}{r}\right) r} & \text{ for } & k=2 \\ 
\frac{\sigma }{\left( \ln ^{\left( k\right) }\frac{1}{r}\right) ...\left(
\ln ^{\left( 3\right) }\frac{1}{r}\right) \left( \ln \frac{1}{r}\right) r} & 
\text{ for } & k\geq 3%
\end{array}%
\right. .
\end{equation*}%
We also have $\Lambda _{k}\left( r\right) \approx 1$, which then gives 
\begin{equation*}
-F^{\prime }\left( r\right) \approx \frac{F\left( r\right) }{r\ln \frac{1}{r}%
},
\end{equation*}%
and 
\begin{equation*}
F^{\prime \prime }\left( r\right) \approx \frac{F\left( r\right) }{%
r^{2}\left( \ln \frac{1}{r}\right) ^{2}}+\frac{\sigma F\left( r\right) }{%
\left( \ln ^{\left( k\right) }\frac{1}{r}\right) ...\left( \ln ^{\left(
3\right) }\frac{1}{r}\right) \left( \ln \frac{1}{r}\right) ^{2}r^{2}}+\frac{%
F\left( r\right) }{r^{2}\ln \frac{1}{r}}\approx \frac{F\left( r\right) }{%
r^{2}\ln \frac{1}{r}}.
\end{equation*}%
From these two estimates we immediately obtain structure condition (5) of
Definition \ref{structure conditions}.

We also have 
\begin{equation*}
\frac{\left\vert F^{\prime }\left( r\right) \right\vert ^{2}}{F^{\prime
\prime }(r)}\approx \frac{F\left( r\right) ^{2}}{\left( r\ln \frac{1}{r}
\right) ^{2}}\frac{r^{2}\ln \frac{1}{r}}{F\left( r\right) }=\frac{F\left(
r\right) }{\ln \frac{1}{r}}=\left( \ln ^{\left( k\right) }\frac{1}{r}\right)
^{\sigma },\ \ \ \ \ 0\leq r\leq \beta _{m,\sigma }\ ,
\end{equation*}
and then from the definition of $\varphi \left( r\right) \equiv \frac{1}{
|F^{\prime }(r)|}e^{C_{m}\left( \frac{\left\vert F^{\prime }\left( r\right)
\right\vert ^{2}}{F^{\prime \prime }(r)}+1\right) ^{m-1}}$ in (\ref{mon prop}
), we obtain 
\begin{eqnarray*}
\varphi \left( r\right) &=&\frac{1}{|F^{\prime }(r)|}e^{C_{m}\left( \frac{
\left\vert F^{\prime }\left( r\right) \right\vert ^{2}}{F^{\prime \prime
}(r) }+1\right) ^{m-1}}\approx r\frac{e^{C_{m}\left( \ln ^{\left( k\right) }%
\frac{ 1}{r}\right) ^{\sigma \left( m-1\right) }}}{\left( \ln ^{\left(
k\right) } \frac{1}{r}\right) ^{\sigma }} \\
&\lesssim &re^{C_{m}\left( \ln ^{\left( k\right) }\frac{1}{r}\right)
^{\sigma \left( m-1\right) }}\approx r^{1-C_{m}\frac{\left( \ln ^{\left(
k\right) }\frac{1}{r}\right) ^{\sigma \left( m-1\right) }}{\ln \frac{1}{r}}
},\ \ \ \ \ 0\leq r\leq \beta _{m,\sigma }.
\end{eqnarray*}
This completes the proof of the monotonicity property (\ref{mon prop}) and
the estimates for $\varphi \left( r\right) $ for each of the two cases in
Corollary \ref{Sob Fsigma}.

Finally, we must show that the standard $\left( \Phi ,\varphi \right) $
-Sobolev inequality (\ref{Phi bump'}) with $\Phi $ as in (\ref{def Phi m ext}
), $m>1$, fails if $k=1$ and $\sigma >\frac{1}{m-1}$, and for this it is
convenient to use the identity $\left\vert \nabla _{A}v\right\vert
=\left\vert \frac{\partial v}{\partial r}\right\vert $ for radial functions $%
v$, see \cite[Appendix C.]{KoRiSaSh19}. Take $f\left( r\right) =f_{1,\sigma
}\left( r\right) =r^{\left( \ln \frac{1}{r}\right) ^{\sigma }}$ and with $%
\eta \left( r\right) \equiv \left\{ 
\begin{array}{ccc}
1 & \text{ if } & 0\leq r\leq \frac{r_{0}}{2} \\ 
2\left( 1-\frac{r}{r_{0}}\right) & \text{ if } & \frac{r_{0}}{2}\leq r\leq
r_{0}%
\end{array}%
\right. $, we define the radial function 
\begin{equation*}
w\left( x,y\right) =w\left( r\right) =e^{\left( \ln \frac{1}{r}\right)
^{\sigma +1}}=\frac{\eta \left( r\right) }{f\left( r\right) },\ \ \ \ \
0<r<r_{0}.
\end{equation*}%
From $\left\vert \nabla _{A}r\right\vert =1$, we obtain the equality $%
\left\vert \nabla _{A}w\left( x,y\right) \right\vert =\left\vert \nabla
_{A}r\right\vert \left\vert w^{\prime }\left( r\right) \right\vert
=\left\vert w^{\prime }\left( r\right) \right\vert $, and combining this
with $\left\vert \nabla _{A}\eta \left( r\right) \right\vert \leq \frac{2}{%
r_{0}}\mathbf{1}_{\left[ \frac{r_{0}}{2},r_{0}\right] }$ and the estimate
(7.8) from \cite{KoRiSaSh19}, we have 
\begin{eqnarray*}
\int \int_{B\left( 0,r_{0}\right) }\left\vert \nabla _{A}w\left( x,y\right)
\right\vert dxdy &\lesssim &\int_{0}^{r_{0}}\left\vert w^{\prime }\left(
r\right) \right\vert \frac{f\left( r\right) }{\left\vert F^{\prime }\left(
r\right) \right\vert }dr+\frac{2}{r_{0}}\int_{\frac{r_{0}}{2}}^{r_{0}}\frac{1%
}{\left\vert F^{\prime }\left( r\right) \right\vert }dr \\
&\approx &\int_{0}^{r_{0}}\frac{f^{\prime }\left( r\right) }{f\left(
r\right) ^{2}}\frac{f\left( r\right) ^{2}}{f^{\prime }\left( r\right) }dr+%
\frac{2}{r_{0}}\int_{\frac{r_{0}}{2}}^{r_{0}}Crdr\approx r_{0}\ .
\end{eqnarray*}%
On the other hand, $\Phi _{m}\left( t\right) \geq t^{1+\frac{m}{\left( \ln
t\right) ^{\frac{1}{m}}}}$ and $\left\vert F^{\prime }\left( r\right)
\right\vert =\left( \sigma +1\right) \left( \ln \frac{1}{r}\right) ^{\sigma }%
\frac{1}{r}$, so we obtain 
\begin{eqnarray*}
&&\int \int_{B\left( 0,r_{0}\right) }\Phi _{m}\left( w\left( x,y\right)
\right) dxdy \\
&\gtrsim &\int_{0}^{\frac{r_{0}}{2}}\Phi _{m}\left( \frac{1}{f\left(
r\right) }\right) \frac{f\left( r\right) }{\left\vert F^{\prime }\left(
r\right) \right\vert }dr\geq \int_{0}^{\frac{r_{0}}{2}}\left( \frac{1}{%
f\left( r\right) }\right) ^{1+\frac{m}{F\left( r\right) ^{\frac{1}{m}}}}%
\frac{f\left( r\right) }{\left\vert F^{\prime }\left( r\right) \right\vert }%
dr \\
&\approx &\int_{0}^{\frac{r_{0}}{2}}f\left( r\right) ^{-\frac{m}{\left( \ln 
\frac{1}{r}\right) ^{\frac{\sigma }{m}}}}\frac{1}{\left( \ln \frac{1}{r}%
\right) ^{\sigma }\frac{1}{r}}dr=\int_{0}^{\frac{r_{0}}{2}}e^{m\left( \ln 
\frac{1}{r}\right) ^{\left( \sigma +1\right) \left( 1-\frac{1}{m}\right) }}%
\frac{rdr}{\left( \ln \frac{1}{r}\right) ^{\sigma }}=\infty
\end{eqnarray*}%
if $\left( \sigma +1\right) \left( 1-\frac{1}{m}\right) >1$, i.e. $\sigma >%
\frac{1}{m-1}$. This finishes the proof of corollary \ref{Sob Fsigma}.
\end{proof}

\subsection{Proof of the geometric theorems}

In this section we prove the geometric Theorems \ref{th-geom-bound} and \ref%
{th-geom-max} as consequence of the abstract Theorems \ref{th-abs-bound} and %
\ref{th-abs-max}, and of the geometric Orlicz-Sobolev inequality established
in Section \ref{Sec Orlicz}.

\begin{proof}[Proof of Theorem \protect\ref{th-geom-bound}]
Theorem \ref{th-geom-bound} is a consequence of the abstract Theorem \ref%
{th-abs-bound} and the geometric results described in Section \ref{Sec
Orlicz}, once we show that under the hypotheses of Theorem \ref%
{th-geom-bound} conditions (\ref{cond-1}), (\ref{cond-2}), and (\ref{cond-3}%
) of Theorem \ref{th-abs-bound} are satisfied.

Since the matrix $A\left( x\right) $ in (\ref{sing-matrix}) is elliptic away
from the line $x_{1}=0$ and it is independent of the second variable $x_{2}$%
, it suffices to prove the theorem for a ball $B\left( 0,r_{0}\right)
\Subset \Omega $. By Corollary \ref{Sob Fsigma} in Section \ref{Sec Orlicz},
when $k=1$ and $0<\sigma <\frac{1}{m-1}$ or $k\geq 2$ and $\sigma >0$, we
have that there exists $0<r_{0}=r_{0}\left( m,\sigma \right) $ such that the
single scale $\left( \Phi ,A,\varphi \right) $-Orlicz-Sobolev bump
inequality (\ref{OS ineq}) holds with $\Phi =\Phi _{m}$ at $\left(
0,r\right) $ for some $m>2$ and superradius ${\varphi }\left( r\right) $
given by%
\begin{equation}
\frac{\varphi \left( r\right) }{r}=\exp \left( C_{m}\left( \ln ^{\left(
k\right) }\frac{1}{r}\right) ^{\sigma \left( m-1\right) }\right) ,\qquad 
\text{for }0<r\leq r_{0}\leq 1.  \label{phi-grow}
\end{equation}

Hence condition (\ref{cond-2}) from Theorem \ref{th-abs-bound} is satisfied
because of condition (2) of Theorem \ref{th-geom-bound}.

It suffices to consider the case that $u$ is a weak \emph{subsolution} of (%
\ref{eq-quasi}) in $\Omega $, with right hand side pair as in condition (1)
of Theorem \ref{th-geom-bound}. Write $\tilde{A}\left( x\right) =\mathcal{A}%
\left( x,u\left( x\right) \right) $ as before. Since $\phi _{0}\in L^{\Phi
^{\ast }}\left( B\left( 0,r\right) \right) $, and $\vec{\phi}_{1}\in
L^{\infty }\left( B\left( 0,r\right) \right) $, then the pair $\left( \phi
_{0},\vec{\phi}_{1}\right) $ is strongly $A$-admissible at $\left(
0,r\right) $ by Proposition \ref{prop:admiss-suff}, so condition (\ref%
{cond-1}) from Theorem \ref{th-abs-bound} holds.

Finally, given $B\left( x,r_{0}\right) \Subset \Omega $, the existence of an 
$\left( A,d\right) $-\emph{standard} accumulating sequence of Lipschitz
cutoff functions at $\left( x,r_{0}\right) $ follows directly from Lemma \ref%
{spec_cutoff_lemma} above, so condition (\ref{cond-3}) from Theorem \ref%
{th-abs-bound} holds. Therefore, applying Theorem \ref{th-abs-bound}, $u$ is 
\emph{locally bounded above} in $\Omega $.
\end{proof}

The proof of our second application, the geometric maximum principle Theorem %
\ref{th-geom-max}, also proceeds by showing that under the conditions of
Theorem \ref{th-geom-max} all the hypotheses of the abstract counterpart,
Theorem \ref{th-abs-max}, are satisfied.

\begin{proof}[Proof of Theorem \protect\ref{th-geom-max}]
We will show that under the hypotheses of the theorem the pair $\left( \phi
_{0},\vec{\phi}_{1}\right) $ is $A$-admissible in $\Omega $, and that the
global $\left( \Phi ,A\right) $-Orlicz-Sobolev bump inequality (\ref{OS
global}) holds in $\Omega $ with $\Phi =\Phi _{m}$ for some $m>2$.

First, since $\phi _{0}\in L^{\Phi ^{\ast }}\left( \Omega \right) $, and $%
\vec{\phi}_{1}\in L^{\infty }\left( \Omega \right) $, from Proposition \ref%
{prop:admiss-suff} in Section \ref{Sec Orlicz} it follows that the right
hand side pair $\left( \phi _{0},\vec{\phi}_{1}\right) $ is strongly $A$%
-admissible in $\Omega $ (and therefore it is $A$-admissble).

So, it only remains to show that the global $\left( \Phi ,A\right) $%
-Orlicz-Sobolev bump inequality (\ref{OS global}) holds in $\Omega $ for
some Young function $\Phi =\Phi _{m}$ with $m>2$. This is proved in the same
way as in \cite[Proposition 81]{KoRiSaSh19}, where the global Sobolev
inequality is obtained from the local $\left( \Phi ,A,{\varphi }\right) $%
-Orlicz-Sobolev inequality. Indeed, we have seen above that such local
inequality holds for $\Phi _{m}$ ($m>2$) at $B\left( x,r\right) $ for $%
0<r\leq r_{0}\left( m,\sigma \right) $, and 
\begin{equation*}
\varphi \left( r\right) =r\exp \left( C_{m}\left( \ln ^{\left( k\right) }%
\frac{1}{r}\right) ^{\sigma \left( m-1\right) }\right) ,\qquad \text{for }%
0<r\leq r_{0}\leq 1.
\end{equation*}%
Since $\Omega $ is bounded, we can cover it with a finite number of balls $%
\Omega \subset \bigcup_{j=1}^{N}B\left( x_{j},r_{0}\right) $, and given $%
\Omega ^{\prime }\Subset \Omega $ we let $\eta _{j}$ be a partition of unity
subordinated to $\left\{ \left\{ B\left( x_{j},r_{0}\right) \right\}
_{j=1}^{n},\Omega ,\Omega ^{\prime }\right\} $, i.e. $\eta _{j}\in
C_{0}^{\infty }\left( B_{j}\right) $, $0\leq \eta _{j}\leq 1$, $\left\Vert
\nabla _{A}\eta _{j}\right\Vert _{L^{\infty }}\leq \frac{C}{\mathrm{dist}%
\left( \Omega ^{\prime },\partial \Omega \right) }=C_{0}$, and $\sum \eta
_{j}\equiv 1$ on $\overline{\Omega ^{\prime }}$. Suppose $v\in \mathrm{Lip}%
_{\mathrm{c}}\left( \Omega \right) $ with $\mathrm{supp}v\subset \Omega
^{\prime }$, then 
\begin{equation*}
v=\sum_{j=1}^{N}v\eta _{j}\equiv \sum_{j=1}^{N}v_{j},
\end{equation*}%
and by inequality (\ref{finite}) we have 
\begin{eqnarray*}
\Phi ^{-1}\left( \int_{\Omega }\Phi \left( \left\vert v\right\vert \right)
~dx\right) &\leq &C_{\Phi ,N}\sum_{j=1}^{N}\Phi ^{-1}\left( \int_{B_{j}}\Phi
\left( \left\vert v_{j}\right\vert \right) ~dx\right) \\
&=&C_{\Phi ,N}\sum_{j=1}^{N}\Phi ^{-1}\left( \left\vert B_{j}\right\vert
\int_{B_{j}}\Phi \left( \left\vert v_{j}\right\vert \right) ~\frac{dx}{%
\left\vert B_{j}\right\vert }\right) \\
&\leq &C_{\Phi ,N}\sum_{j=1}^{N}\Phi ^{-1}\left( \int_{B_{j}}\Phi \left(
\left\vert v_{j}\right\vert \right) ~\frac{dx}{\left\vert B_{j}\right\vert }%
\right) \\
&\leq &C_{\Phi ,N}{\varphi }\left( r_{0}\right)
\sum_{j=1}^{N}\int_{B_{j}}\left\vert \nabla _{A}v_{j}\right\vert ~\frac{dx}{%
\left\vert B_{j}\right\vert }.
\end{eqnarray*}%
where we used that $0<\left\vert B_{j}\right\vert \leq 1$. Note that here we
are using the local Orlicz-Sobolev inequality at balls centered at arbitrary
points in $\Omega $. This is allowed because the hypotheses of Corollary \ref%
{Sob Fsigma} are still satisfied away form the line $x_{1}=0$, where the
geometry is the most singular. Letting $C_{1}=\max_{j}\left\vert
B_{j}\right\vert ^{-1}$, it follows that%
\begin{equation*}
\Phi ^{-1}\left( \int_{\Omega }\Phi \left( \left\vert v\right\vert \right)
~dx\right) \leq C_{\Phi ,N}C_{1}{\varphi }\left( r_{0}\right) \left\Vert
\nabla _{A}v\right\Vert _{L^{1}\left( \Omega \right) }+C_{\Phi ,N}C_{0}C_{1}{%
\varphi }\left( r_{0}\right) \left\Vert v\right\Vert _{L^{1}\left( \Omega
\right) }.
\end{equation*}%
Since the matrix $A$ is non-singular in the $x_{1}$-direction we have the $%
\left( 1,1\right) $-Sobolev estimate 
\begin{equation*}
\left\Vert v\right\Vert _{L^{1}\left( \Omega \right) }\leq C\mathrm{diam}%
\Omega \left\Vert \nabla _{A}v\right\Vert _{L^{1}\left( \Omega \right)
}\qquad \text{for all }v\in \mathrm{Lip}_{\mathrm{c}}\left(
\Omega \right) .
\end{equation*}%
Substituting this into the previous inequality yields the global
Orlicz-Sobolev inequality 
\begin{equation*}
\Phi ^{-1}\left( \int_{\Omega }\Phi \left( \left\vert v\right\vert \right)
~dx\right) \leq C_{\Phi ,\Omega }\left\Vert \nabla _{A}v\right\Vert
_{L^{1}\left( \Omega \right) }
\end{equation*}%
as wanted.
\end{proof}

\bibliographystyle{amsplain}
\bibliography{libraryMoser}

\providecommand{\bysame}{\leavevmode\hbox to3em{\hrulefill}\thinspace}
\providecommand{\MR}{\relax\ifhmode\unskip\space\fi MR }
\providecommand{\MRhref}[2]{%
  \href{http://www.ams.org/mathscinet-getitem?mr=#1}{#2}
}
\providecommand{\href}[2]{#2}
\begin{thebibliography}{10}

\bibitem{CrRo21}
David Cruz-Uribe and Scott Rodney, \emph{Bounded weak solutions to elliptic
  {PDE} with data in {O}rlicz spaces}, J. Differential Equations \textbf{297}
  (2021), 409--432. \MR{4280269}

\bibitem{CrRo23}
\bysame, \emph{A note on the limit of {O}rlicz norms}, Real Anal. Exchange
  \textbf{48} (2023), no.~1, 77--81. \MR{4556453}

\bibitem{DiFaMoRo23}
Giuseppe Di~Fazio, Maria~Stella Fanciullo, Dario~Daniele Monticelli, Scott
  Rodney, and Pietro Zamboni, \emph{Matrix weights and regularity for
  degenerate elliptic equations}, Nonlinear Anal. \textbf{237} (2023), Paper
  No. 113363, 24. \MR{4634961}

\bibitem{GiTr}
David Gilbarg and Neil~S. Trudinger, \emph{Elliptic partial differential
  equations of second order}, Classics in Mathematics, Springer-Verlag, Berlin,
  2001, Reprint of the 1998 edition. \MR{1814364}

\bibitem{KoMaRi}
Lyudmila Korobenko, Diego Maldonado, and Cristian Rios, \emph{From {S}obolev
  inequality to doubling}, Proc. Amer. Math. Soc. \textbf{143} (2015), no.~9,
  4017--4028. \MR{3359590}

\bibitem{KoRiSaSh1}
Lyudmila Korobenko, Cristian Rios, Eric Sawyer, and Ruipeng Shen, \emph{Local
  boundedness, maximum principles, and continuity of solutions to infinitely
  degenerate elliptic equations}, 2015.

\bibitem{KoRiSaSh19}
Lyudmila Korobenko, Cristian Rios, Eric Sawyer, and Ruipeng Shen, \emph{Local
  boundedness, maximum principles, and continuity of solutions to infinitely
  degenerate elliptic equations with rough coefficients}, Mem. Amer. Math. Soc.
  \textbf{269} (2021), no.~1311, vii+130. \MR{4224718}

\bibitem{KoSa21}
Lyudmila Korobenko and Eric Sawyer, \emph{Continuity of infinitely degenerate
  weak solutions via the trace method}, J. Funct. Anal. \textbf{281} (2021),
  no.~9, Paper No. 109170, 30. \MR{4287784}

\bibitem{SaWh4}
Eric~T. Sawyer and Richard~L. Wheeden, \emph{H\"{o}lder continuity of weak
  solutions to subelliptic equations with rough coefficients}, Mem. Amer. Math.
  Soc. \textbf{180} (2006), no.~847, x+157. \MR{2204824}

\bibitem{SaWh3}
\bysame, \emph{Degenerate {S}obolev spaces and regularity of subelliptic
  equations}, Trans. Amer. Math. Soc. \textbf{362} (2010), no.~4, 1869--1906.
  \MR{2574880}

\bibitem{Zhong}
Xiao Zhong, \emph{Discontinuous solutions of linear, degenerate elliptic
  equations}, J. Math. Pures Appl. (9) \textbf{90} (2008), no.~1, 31--41.
  \MR{2435212}

\end{thebibliography}

\end{document}